\documentclass[final,onefignum,onetabnum]{siamonline190516}

\usepackage{lipsum}
\usepackage{amsfonts}
\usepackage{epstopdf}
\ifpdf
  \DeclareGraphicsExtensions{.eps,.pdf,.png,.jpg}
\else
  \DeclareGraphicsExtensions{.eps}
\fi

\usepackage{datetime}
\newdateformat{monthyeardate}{%
  \monthname[\THEMONTH] \THEDAY, \THEYEAR}

\usepackage{academicons}
\usepackage{xcolor}
\renewcommand{\orcid}[1]{\href{https://orcid.org/#1}{\textcolor[HTML]{A6CE39}{orcid.org/#1}}}

\usepackage{amsmath}
\allowdisplaybreaks
\usepackage{amssymb}
\usepackage{commath}
\usepackage{mathtools}
\usepackage{bbm}
\usepackage{bm}

\usepackage{color}
\usepackage{graphicx}
\usepackage[small]{caption}
\usepackage{subcaption}

\usepackage{relsize}
\usepackage{adjustbox}
\usepackage{algorithm}
\usepackage[noend]{algpseudocode}
\usepackage{booktabs}
\usepackage{tikz}
\usepackage{pifont}

\usepackage{enumitem}
\setlist[enumerate]{leftmargin=.5in}
\setlist[itemize]{leftmargin=.5in}

\newsiamremark{remark}{Remark}
\newsiamremark{example}{Example}
\newsiamremark{hypothesis}{Hypothesis}
\crefname{hypothesis}{Hypothesis}{Hypotheses}
\newsiamthm{claim}{Claim}

\headers{Towards stability of RBF-CFs}{Glaubitz and Reeger}

\title{Towards stability of radial basis function based cubature formulas\thanks{\monthyeardate\today \corresponding{Jan Glaubitz (\email{Jan.Glaubitz@Dartmouth.edu}, \orcid{0000-0002-3434-5563})}
\funding{This work was partially supported by AFOSR \#F9550-18-1-0316 and ONR \#N00014-20-1-2595 (Glaubitz).}
\disclaimer{The views expressed in this academic research paper are those of the authors and do not reflect the official policy or position of the United States Government or Department of Defense.  In accordance with the Air Force Instruction 51-303, it is not copyrighted, but is the property of the United States government.}
}}

\author{Jan Glaubitz\thanks{Department of Mathematics, Dartmouth College, Hanover, NH 03755, USA}
\and
Jonah Reeger\thanks{Senior Research Mathematician, Sensors Directorate, Air Force Research Laboratory, Wright--Patterson Air Force Base, OH 45433, USA}
}

\usepackage{amsopn}

\usepackage[normalem]{ulem}

\DeclareMathOperator*{\argmin}{arg\,min}
\DeclareMathOperator*{\cond}{cond}
\DeclareMathOperator{\arcsinh}{arcsinh}
\newcommand{\scp}[2]{\left\langle{#1,\, #2}\right\rangle}

\newcommand{\intd}{\, \mathrm{d}}
\newcommand{\N}{\mathbb{N}}

\newcommand{\R}{\mathbb{R}}

\usepackage{lineno}

\ifpdf
\hypersetup{
  pdftitle={Glaubitz2021Stability},
  pdfauthor={Jan Glaubitz}
}
\fi

\begin{document}

\maketitle

\begin{abstract}
Cubature formulas (CFs) based on radial basis functions (RBFs) have become an important tool for multivariate numerical integration of scattered data. 
Although numerous works have been published on such RBF-CFs, their stability theory can still be considered as underdeveloped. 
Here, we strive to pave the way towards a more mature stability theory for RBF-CFs. 
In particular, we prove stability for RBF-CFs based on compactly supported RBFs under certain conditions on the shape parameter and the data points. 
Moreover, it is shown that asymptotic stability of many RBF-CFs is independent of polynomial terms, which are often included in RBF approximations. 
While our findings provide some novel conditions for stability of RBF-CFs, the present work also demonstrates that there are still many gaps to fill in future investigations. 
\end{abstract}



\section{Introduction} 
\label{sec:introduction} 

Numerical integration is an omnipresent task in mathematics and myriad applications. 
While these are too numerous to list fully, prominent examples include numerical differential equations \cite{hesthaven2007nodal,quarteroni2008numerical,ames2014numerical}, machine learning \cite{murphy2012machine}, finance \cite{glasserman2013monte}, and biology \cite{manly2006randomization}. 
In many cases, the problem can be formulated as follows.  
Let $\Omega \subset \R^D$ be a bounded domain with positive volume, $|\Omega| > 0$.
Given $N$ distinct data pairs $\{ (\mathbf{x}_n,f_n) \}_{n=1}^N \subset \Omega \times \R$ with $f: \Omega \to \R$ and $f_n := f(\mathbf{x}_n)$, the aim is to approximate the weighted integral 
\begin{equation}\label{eq:I}
  I[f] := \int_\Omega f(\boldsymbol{x}) \omega(\boldsymbol{x}) \intd \boldsymbol{x}
\end{equation}
by an \emph{$N$-point CF}. 
That is, by a weighted finite sum over the given function of the form 
\begin{equation}\label{eq:CR}
  C_N[f] = \sum_{n=1}^N w_n f(\mathbf{x}_n).
\end{equation} 
Here, the distinct points $\{ \mathbf{x}_n \}_{n=1}^N$ are called \emph{data points} and the $\{ w_n \}_{n=1}^N$ are referred to as \emph{cubature weights}. 
Many CFs are derived based on the idea to first approximate the (unknown) function $f$ and to exactly integrate this approximation then \cite{haber1970numerical,stroud1971approximate,engels1980numerical,cools1997constructing,krommer1998computational,cools2003encyclopaedia,krylov2006approximate,davis2007methods,brass2011quadrature,trefethen2017cubature}. 
Arguably, most of the existing CFs have been derived to be exact for polynomials up to a certain degree. 
See \cite{maxwell1877approximate,mysovskikh1980approximation,cools2001cubature,mysovskikh2001cubature,cools2003encyclopaedia,trefethen2021exactness}, in addition to the above references.  

That said, in recent years CFs based on the exact integration of RBFs have received a growing amount of interest \cite{sommariva2005integration,sommariva2006numerical,sommariva2006meshless,punzi2008meshless,aziz2012numerical,fuselier2014kernel,reeger2016numericalA,reeger2016numericalB,watts2016radial,reeger2018numerical,sommariva2021rbf}. 
The increased used of RBFs for numerical integration, as well as numerical differential equations \cite{kansa1990multiquadrics,iske1996structure,fasshauer1996solving,kansa2000circumventing,larsson2003numerical,iske2003radial,shu2007integrated,fornberg2015solving,flyer2016enhancing,glaubitz2021stabilizing,glaubitz2021towards}, seems to be only logical, considering their story of success in the last few decades. 
In fact, since their introduction in Hardy’s work \cite{hardy1971multiquadric} on cartography in 1971, RBFs have become a powerful tool in numerical analysis, including multivariate interpolation and approximation theory  \cite{buhmann2000radial,buhmann2003radial,wendland2004scattered,fasshauer2007meshfree,iske2011scattered,fornberg2015primer}.

Even though RBF-CFs have been proposed and applied in numerous works by now, their stability theory can still be considered as under-developed,
especially when compared to more traditional---e.\,g\ polynomial based---methods.
Stability of RBF-CFs was broached, for instance, in \cite{sommariva2005integration,sommariva2006numerical,punzi2008meshless}. 
However, to the best of our knowledge, an exhaustive stability theory for RBF-CFs is still missing in the literature. 
In particular, theoretical results providing clear conditions---e.\,g., on the kernel, the data points, the weight function, the degree of potentially added polynomial terms---under which stability of RBF-CFs is ensured are rarely encountered. 

\subsection{Our Contribution} 

The present work strives to at least partially fill this gap in the RBF literature. 
This is done by providing a detailed theoretical and numerical investigation on stability of RBF-CFs for different families of kernels. 
These include, compactly supported and Gaussian RBFs as well as polyharmonic splines (PHS).  

In particular, we report on the following findings. 
(1) Stability of RBF-CFs is connected to the Lebesgue constant of the underlying RBF interpolant. 
Consequently, it is demonstrated that a low stability measure for RBF-CFs is promoted by a low Lebesgue constant. 
That said, it is also shown that in many cases RBF-CFs have significantly better stability properties than one might expect based on the underlying RBF interpolant. 
(2) We provide a provable sufficient condition for compactly supported RBFs to yield stable RBF-CF (see \cref{thm:main} in \cref{sec:compact}). 
The result is independent of the degree of the polynomial term that is included in the RBF interpolant and assumes the data points to come from an equidistributed (space-filling) sequence. 
This result is obtained by leveraging a beautiful connection to discrete orthogonal polynomials and is partially motivated by arguments that frequently occur in least-squares quadrature/cubature formulas \cite{huybrechs2009stable,migliorati2018stable,glaubitz2020stableQFs}.  
(3) At least numerically, we find the aforementioned sufficient condition to also be close to necessary in many cases. 
This might be considered as a discouraging result for compactly supported RBF-CFs since the sufficient condition makes some harsh restrictions on the shape parameter. 
(4) Finally, the asymptotic stability of pure RBF-CFs is connected to the asymptotic stability of the same RBF-CF but augmented with polynomials of a fixed arbitrary degree. 
Essentially, we are able to show that for a sufficiently large number of data points, stability of RBF-CFs is independent of the presence of polynomials in the RBF interpolant. 

While there are certainly further stability results desired, in addition to the ones presented here, we believe this work to be a valuable step towards a more mature stability theory for RBF-CFs.    

\subsection{Outline} 

The rest of this work is organized as follows. 
We start by collecting some preliminaries on RBF interpolants and CFs in \cref{sec:prelim}. 
In \cref{sec:stability} a few initial comments on stability of (RBF-)CFs are offered. 
Building up on these, it is demonstrated in \cref{sec:initial} that RBF-CFs in many cases have superior stability properties compared to RBF interpolation. 
Next, \cref{sec:compact} contains our theoretical main result regarding stability of RBF-CFs based on compactly supported kernels. 
Furthermore, in \cref{sec:connection} it is proven that, under certain assumptions, asymptotic stability of RBF-CFs is independent of the polynomial terms that might be included in the RBF interpolant. 
The aforementioned theoretical findings are accompanied by various numerical tests in  \cref{sec:numerical} 
Finally, concluding thoughts are offered in \cref{sec:summary}. 
\section{Preliminaries} 
\label{sec:prelim} 

We start by collecting some preliminaries on RBF interpolants (\cref{sub:prelim_RBFs}) as well as RBF-CFs (\cref{sub:prelim_CFs}).

\subsection{Radial Basis Function Interpolation}
\label{sub:prelim_RBFs} 

RBFs are often considered a powerful tool in numerical analysis, including multivariate interpolation and approximation theory \cite{buhmann2000radial,buhmann2003radial,wendland2004scattered,fasshauer2007meshfree,iske2011scattered,fornberg2015primer}.
In the context of the present work, we are especially interested in RBF interpolants. 
Let $f: \R^D \supset \Omega \to \R$ be a scalar valued function. 
Given a set of distinct \emph{data points} (in context of RBFs sometimes also referred to as \emph{centers}), the \emph{RBF interpolant} of $f$ is of the form 
\begin{equation}\label{eq:RBF-interpol}
  (s_{N,d}f)(\boldsymbol{x}) 
    = \sum_{n=1}^N \alpha_n \varphi( \varepsilon_{n} \| \boldsymbol{x} - \mathbf{x}_n \|_2 ) + \sum_{k=1}^K \beta_k p_k(\boldsymbol{x}).
\end{equation} 
Here, $\varphi: \R_0^+ \to \R$ is the \emph{RBF} (also called \emph{kernel}), $\{p_k\}_{k=1}^K$ is a basis of the space of all algebraic polynomials up to degree $d$, $\mathbb{P}_d(\Omega)$, and the $\varepsilon_{n}$'s are nonnegative shape parameters.
Furthermore, the RBF interpolant \cref{eq:RBF-interpol} is uniquely determined by the conditions 
\begin{alignat}{2}
	(s_{N,d}f)(\mathbf{x}_n) 
    		& = f(\mathbf{x}_n), \quad 
		&& n=1,\dots,N, \label{eq:interpol_cond} \\ 
  	\sum_{n=1}^N \alpha_n p_\mathbf{k}(\mathbf{x}_n) 
    		& = 0 , \quad 
		&& k=1,\dots,K. \label{eq:cond2}
\end{alignat} 
In this work, we shall focus on the popular choices of RBFs listed in \cref{tab:RBFs}. 
A more complete list of RBFs and their properties can be found in the monographs \cite{buhmann2003radial,wendland2004scattered,fasshauer2007meshfree,fornberg2015primer} and references therein.

\begin{remark}[Implementation of $\varphi(r) = r^{2k} \log r$]
The polyharmonic splines (PHS) of the form $\varphi(r) = r^{2k} \log r$ are usually implemented as $\varphi(r) = r^{2k-1} \log( r^r )$ to avoid numerical problems at $r=0$, where "$\log (0) = -\infty$".
\end{remark}

\begin{table}[t]
  \centering 
  \renewcommand{\arraystretch}{1.3}
  \begin{tabular}{c|c|c|c}
    RBF & $\varphi(r)$ & parameter & order \\ \hline 
    Gaussian & $\exp( -(\varepsilon r)^2)$ & $\varepsilon>0$ & 0 \\ 
    Wendland's & $\varphi_{D,k}(r)$, see \cite{wendland1995piecewise} & $D, k \in \N_0$ & 0 \\ 
    Polyharmonic splines & $r^{2k-1}$ & $k \in \N$ & $k$ \\ 
    & $r^{2k} \log r$ & $k \in \N$ & $k+1$ 
  \end{tabular} 
  \caption{Some popular RBFs}
  \label{tab:RBFs}
\end{table}

Note that \cref{eq:interpol_cond} and \cref{eq:cond2} can be reformulated as a linear system for the coefficient vectors $\boldsymbol{\alpha} = [\alpha_1,\dots,\alpha_N]^T$ and $\boldsymbol{\beta} = [\beta_1,\dots,\beta_K]^T$. 
This linear system is given by 
\begin{equation}\label{eq:system} 
	\begin{bmatrix} \Phi & P \\ P^T & 0 \end{bmatrix}
	\begin{bmatrix} \boldsymbol{\alpha} \\ \boldsymbol{\beta} \end{bmatrix} 
	= 
	\begin{bmatrix} \mathbf{f} \\ \mathbf{0} \end{bmatrix}
\end{equation} 
where $\mathbf{f} = [f(\mathbf{x}_1),\dots,f(\mathbf{x}_N)]^T$ as well as 
\begin{equation}\label{eq:Phi_P}
	\Phi = 
	\begin{bmatrix} 
		\varphi( \varepsilon_{1} \| \mathbf{x}_1 - \mathbf{x}_1 \|_2 ) & \dots & \varphi( \varepsilon_{N} \| \mathbf{x}_1 - \mathbf{x}_N \|_2 ) \\ 
		\vdots & & \vdots \\ 
		\varphi( \varepsilon_{1} \| \mathbf{x}_N - \mathbf{x}_1 \|_2 ) & \dots & \varphi( \varepsilon_{N} \| \mathbf{x}_N - \mathbf{x}_N \|_2 )
	\end{bmatrix}, 
	\quad 
	P = 
	\begin{bmatrix} 
		p_1(\mathbf{x}_1) & \dots & p_K(\mathbf{x}_1) \\ 
		\vdots & & \vdots \\  
		p_1(\mathbf{x}_N) & \dots & p_K(\mathbf{x}_N) 
	\end{bmatrix}.
\end{equation} 
It is well-known that \cref{eq:system} is ensured to have a unique solution---corresponding to existence and uniqueness of the RBF interpolant---if the kernel $\varphi$ is positive definite of order $m$ and the set of data points is $\mathbb{P}_{m}(\Omega)$-unisolvent. 
See, for instance, \cite[Chapter 7]{fasshauer2007meshfree} and \cite[Chapter 3.1]{glaubitz2020shock} or references therein. 
The set of all RBF interpolants \cref{eq:RBF-interpol} forms an $N$-dimensional linear space, denote by $\mathcal{S}_{N,d}$. 
This space is spanned by the basis elements 
\begin{equation}\label{eq:cardinal}
  	c_m(\boldsymbol{x}) 
    = \sum_{n=1}^N \alpha_n^{(m)} \varphi( \varepsilon_{n} \| \boldsymbol{x} - \mathbf{x}_n \|_2 ) + \sum_{k=1}^K \beta^{(m)}_k p_k(\boldsymbol{x}), 
    \quad m=1,\dots,N,
\end{equation}
that are uniquely determined by 
\begin{equation}\label{eq:cond_cardinal}
  c_m(\mathbf{x}_n) = \delta_{mn} := 
  \begin{cases} 
    1 & \text{if } m=n, \\ 
    0 & \text{otherwise}, 
  \end{cases} 
  \quad m,n=1,\dots,N,
\end{equation}
and condition \cref{eq:cond2}. 
The functions $c_m$ are the so-called \emph{cardinal functions}. 
They provide us with the following representation of the RBF interpolant \cref{eq:RBF-interpol}: 
\begin{equation}
	(s_{N,d})f(\boldsymbol{x}) = \sum_{n=1}^N f(\mathbf{x}_n) c_n(\boldsymbol{x})
\end{equation} 
This representation is convenient to subsequently derive cubature weights based on RBFs that are independent of the function $f$.

\subsection{Cubature Formulas Based on Radial Basis Functions} 
\label{sub:prelim_CFs} 

A fundamental idea behind many CFs is to first approximate the (unknown) functions $f: \Omega \to \R$ based on the given data pairs $\{\mathbf{x}_n,f_n\}_{n=1}^N \subset \Omega \times \R$ and to exactly integrate this approximation. 
In the case of RBF-CFs this approximation is chosen as the RBF interpolant \cref{eq:RBF-interpol}. 
Hence, the corresponding RBF-CF is defined as 
\begin{equation}\label{eq:RBF-CRs_def}
	C_N[f] := I[s_{N,d}f] = \int_{\Omega} (s_{N,d}f)(\boldsymbol{x}) \omega(\boldsymbol{x}) \intd \boldsymbol{x}. 
\end{equation}
When formulated w.\,r.\,t.\ the cardinal functions $c_n$, $n=1,\dots,N$, we get 
\begin{equation}\label{eq:RBF-CRs}
	C_N[f] = \sum_{n=1}^N w_n f(x_n) 
	\quad \text{with} \quad w_n = I[c_n]. 
\end{equation} 
That is, the RBF cubature weights $\mathbf{w}$ are given by the moments corresponding to the cardinal functions. 
This formulation is often preferred over \cref{eq:RBF-CRs_def} since the cubature weights $\mathbf{w}$ do not have to be recomputed when another function is considered. 
In our implementation, we compute the RBF cubature weights by solving the linear system 
\begin{equation}\label{eq:LS_weights}
	\underbrace{\begin{bmatrix} \Phi & P \\ P^T & 0 \end{bmatrix}}_{= A}  
	\begin{bmatrix} \mathbf{w} \\ \mathbf{v} \end{bmatrix} 
	= 
	\begin{bmatrix} \mathbf{m}^{\text{RBF}} \\ \mathbf{m}^{\text{poly}} \end{bmatrix},
\end{equation} 
where $\mathbf{v} \in \R^K$ is an auxiliary vector. 
Furthermore, the vectors ${\mathbf{m}^{\text{RBF}} \in \R^N}$ and ${\mathbf{m}^{\text{poly}} \in \R^K}$ contain the moments of the translated kernels and polynomial basis functions, respectively. 
That is, 
\begin{equation} 
\begin{aligned}
	\mathbf{m}^{\text{RBF}} & = \left[ I[\varphi_1], \dots, I[\varphi_N] \right]^T, \\ 
	\mathbf{m}^{\text{poly}} & = \left[ I[p_1], \dots, I[p_K] \right]^T,
\end{aligned}
\end{equation} 
with $\varphi_n(\boldsymbol{x}) = \varphi( \varepsilon_{n} \| \boldsymbol{x} - \mathbf{x}_n \|_2 )$. 
The moments of different RBFs can be found in \cref{sec:app_moments} and references listed there. 
The moments of polynomials for different domains $\Omega$ can be found in the literature, e.\,g., \cite[Appendix A]{glaubitz2020stableCFs} and \cite{folland2001integrate,lasserre2021simple}. 
\section{Stability and the Lebesgue Constant} 
\label{sec:stability}

In this section, we address stability of RBF interpolants and the corresponding RBF-CFs. 
In particular, we show that both can be estimated in terms of the famous Lebesgue constant. 
That said, we also demonstrate that RBF-CFs often come with improved stability compared to RBF interpolation.

\subsection{Stability and Accuracy of Cubature Formulas} 
\label{sub:stability_CFs}

We start by addressing stability and accuracy of RBF-CFs. 
To this end, let us denote the best approximation of $f$ from $\mathcal{S}_{N,d}$ in the $L^\infty$-norm by $\hat{s}$. 
That is, 
\begin{equation}\label{eq:Lebesgue}
	\hat{s} = \argmin_{s \in \mathcal{S}_{N,d}} \norm{ f - s }_{L^{\infty}(\Omega)} 
	\quad \text{with} \quad 
	\norm{ f - s }_{L^{\infty}(\Omega)} = \sup_{\mathbf{x} \in \Omega} | f(\mathbf{x}) - s(\mathbf{x}) |. 
\end{equation} 
Note that this best approximation w.\,r.\,t.\ the $L^\infty$-norm is not necessarily equal to the RBF interpolant. 
Still, the following error bound holds for the RBF-CF \cref{eq:RBF-CRs}, that corresponds to exactly integrating the RBF interpolant from $\mathcal{S}_{N,d}$: 
\begin{equation}\label{eq:L-inequality} 
\begin{aligned}
	| C_N[f] - I[f] | 
		\leq \left( \| I \|_{\infty} + \| C_N \|_{\infty} \right) \inf_{ s \in \mathcal{S}_{N,d} } \norm{ f - s }_{L^{\infty}(\Omega)}
\end{aligned}
\end{equation} 
Inequality \cref{eq:L-inequality} is commonly known as the Lebesgue inequality; see, e.\,g., \cite{van2020adaptive} or \cite[Theorem 3.1.1]{brass2011quadrature}. 
It is most often encountered in the context of polynomial interpolation \cite{brutman1996lebesgue,ibrahimoglu2016lebesgue}, but straightforwardly carries over to numerical integration.
In this context, the operator norms $\| I \|_{\infty}$ and $\|C_N\|_{\infty}$ are respectively given by $\| I \|_{\infty} = I[1]$ and 
\begin{equation}\label{eq:stab_measure}
	\| C_N \|_{\infty} 
		= \sum_{n=1}^N |w_n| 
		= \sum_{n=1}^N | I[c_n] |.
\end{equation}
Recall that the $c_n$'s are the cardinal functions (see \cref{sub:prelim_RBFs}).

In fact, $\| C_N \|_{\infty}$ is a common stability measure for CFs. 
This is because the propagation of input errors, e.\,g., due to noise or rounding errors, can be bounded as follows: 
\begin{equation}
	| C_N[f] - C_N[\tilde{f}] | 
		\leq \| C_N \|_{\infty} \| f - \tilde{f} \|_{L^\infty}
\end{equation} 
That is, input errors are amplified at most by a factor that is equal to the operator norm $\| C_N \|_{\infty}$. 
At the same time, we have a lower bound for $\| C_N \|_{\infty}$ given by 
\begin{equation}
	\| C_N \|_{\infty} 
		\geq C_N[1],
\end{equation} 
where equality holds if and only if all cubature weights are nonnegative. 
This is the reason for which the construction of CFs is mainly devoted to nonnegative CFs. 

\begin{definition}[Stability]
	We call the RBF-CF $C_N$ \emph{stable} if $\| C_N \|_{\infty} = C_N[1]$ holds. 
	This is the case if and only if $I[c_n] \geq 0$ for all cardinal functions $c_n$, $n=1,\dots,N$.
\end{definition}

It is also worth noting that $C[1] = \| I \|_{\infty}$ if the CF is exact for constants. 
For RBF-CFs, this is the case if at least constants are included in the underlying RBF interpolant ($d \geq 0$). 

Summarizing the above discussion originating from the Lebesgue inequality \cref{eq:L-inequality}, we have a two-fold goal when using RBF-CFs. 
On the one hand, the data points, the kernel, the shape parameter, and the basis of polynomials should be chosen such that $\mathcal{S}_{N,d}$ provides a best approximation to $f$ in the $L^\infty$-norm that is as accurate as possible. 
On the other hand, to ensure stability, $\|C_N\|_{\infty}$ should be as small as possible. 
That is, $I[c_n] \geq 0$ for all cardinal functions $c_n \in \mathcal{S}_{N,d}$.

\subsection{Stability of RBF Approximations} 
\label{sub:stability_RBFs}

We now demonstrate how the stability of RBF-CFs can be connected to the stability of the corresponding RBF interpolant. 
Indeed, the stability measure $\| C_N \|_{\infty}$ can be bounded from above by 
\begin{equation}
	\| C_N \|_{\infty} 
		\leq \| I \|_{\infty} \Lambda_N, 
		\quad \text{with} \quad 
		\Lambda_N := \sup_{\mathbf{x} \in \Omega} \sum_{n=1}^N | c_n(\mathbf{x}) |.
\end{equation}
Here, $\Lambda_N$ is the Lebesgue constant corresponding to the recovery process $f \mapsto s_{N,d}f$ (RBF interpolation). 
Obviously, $\Lambda_N \geq 1$. 
Note that if $1 \in \mathcal{S}_{N,d}$ (the RBF-CF is exact for constants), we therefore have 
\begin{equation}\label{eq:stab_eq1}
	\| I \|_{\infty} 
		\leq \| C_N \|_{\infty} 
		\leq \| I \|_{\infty} \Lambda_N.
\end{equation} 
Hence, the RBF-CF is stable ($\| C_N \|_{\infty} = \| I \|_{\infty}$) if $\Lambda_N$ is minimal ($\Lambda_N=1$).
We briefly note that the inequality $\| C_N \|_{\infty} \leq \| I \|_{\infty} \Lambda_N$ is sharp by considering the following \cref{ex:sharp}. 

\begin{example}[$\|C_N\|_{\infty} = \Lambda_N$]\label{ex:sharp} 
	Let us consider the one-dimensional domain $\Omega = [0,1]$ with $\omega \equiv 1$, which immediately implies $\| I \|_{\infty} = 1$.
	In \cite{bos2008univariate} it was shown that for the linear PHS $\varphi(r) = r$ and data points $0 = x_1 < x_2 < \dots < x_N = 1$ the corresponding cardinal functions $c_m$ are simple hat functions. 
	In particular, $c_m$ is the ordinary ``connect the dots'' piecewise linear interpolant of the data pairs $(x_n,\delta_{nm})$, $n=1,\dots,N$. 
	Thus, $\Lambda_N = 1$. 
	At the same time, this yields $\|C_N\|_{\infty} = 1$ and therefore  
	$\|C_N\|_{\infty} = \Lambda_N$. 
\end{example} 

Looking for minimal Lebesgue constants is a classical problem in recovery theory. 
For instance, it is well known that for polynomial interpolation even near-optimal sets of data points yield a Lebesgue constant that grows as $\mathcal{O}(\log N)$ in one dimension and as $\mathcal{O}(\log^2 N)$ in two dimensions; see \cite{brutman1996lebesgue,bos2006bivariate,bos2007bivariate,ibrahimoglu2016lebesgue} and references therein. 
In the case of RBF interpolation, the Lebesgue constant and appropriate data point distributions were studied in \cite{iske2003approximation,de2003optimal,mehri2007lebesgue,de2010stability} and many more works. 
That said, the second inequality in \cref{eq:stab_eq1} also tells us that in some cases we can expect the RBF-CF to have superior stability properties compared to the underlying RBF interpolant.
In fact, this might not come as a surprise since integration is well-known to have a smoothing (stabilizing) effect in a variety of different contexts. 
Finally, it should be stressed that \cref{eq:stab_eq1} only holds if $1 \in \mathcal{S}_{N,d}$. 
In general, 
\begin{equation}\label{eq:stab_eq2}
	C_N[1]  
		\leq \| C_N \|_{\infty} 
		\leq \| I \|_{\infty} \Lambda_N.
\end{equation} 
Still, this indicates that a recovery space $\mathcal{S}_{N,d}$ is desired that yields a small Lebesgue constant as well as the RBF-CF potentially having superior stability compared to RBF interpolation. 
\section{Theoretical Stability, Numerical Conditioning, and Robustness}
\label{sec:initial}

In this section, we report on two important observations. 
The first being that in many cases we find RBF-CFs to have superior stability properties compared to the corresponding RBF interpolants. 
That is, we show that most often a strict inequality, $\| C_N \|_{\infty} < \| I \|_{\infty} \Lambda_N$, holds for the second inequality in \cref{eq:stab_eq1}. 
Second, we emphasize the importance to distinguish between \emph{theoretical stability} (the CF having nonnegative weights only) and overall \emph{robustness} of the CF. 
The latter one is not just influenced by the theoretical stability---assuming infinite arithmetics---but also incorporates the effect of numerical conditioning. 
In particular, the cubature weights $\mathbf{w}$ are computed by numerically solving the linear system \cref{eq:LS_weights}. 
On a computer, this is always done in some finite arithmetic which inevitably results in rounding errors.
Such rounding errors can also propagate into the cubature weights $\mathbf{w}$ and, depending on the conditioning of the coefficient matrix $A$, might cause the RBF-CF to decrease in robustness. 
That said, our findings below indicate that despite the matrix $A$ often having potentially prohibitively high condition numbers, the numerical computation of the cubature weights $\mathbf{w}$ still yields accurate results for these. 
Henceforth, for sake of simplicity, we assume $\omega \equiv 1$.

\begin{figure}[t]
	\centering
  	\begin{subfigure}[b]{0.4\textwidth}
		\includegraphics[width=\textwidth]{%
      	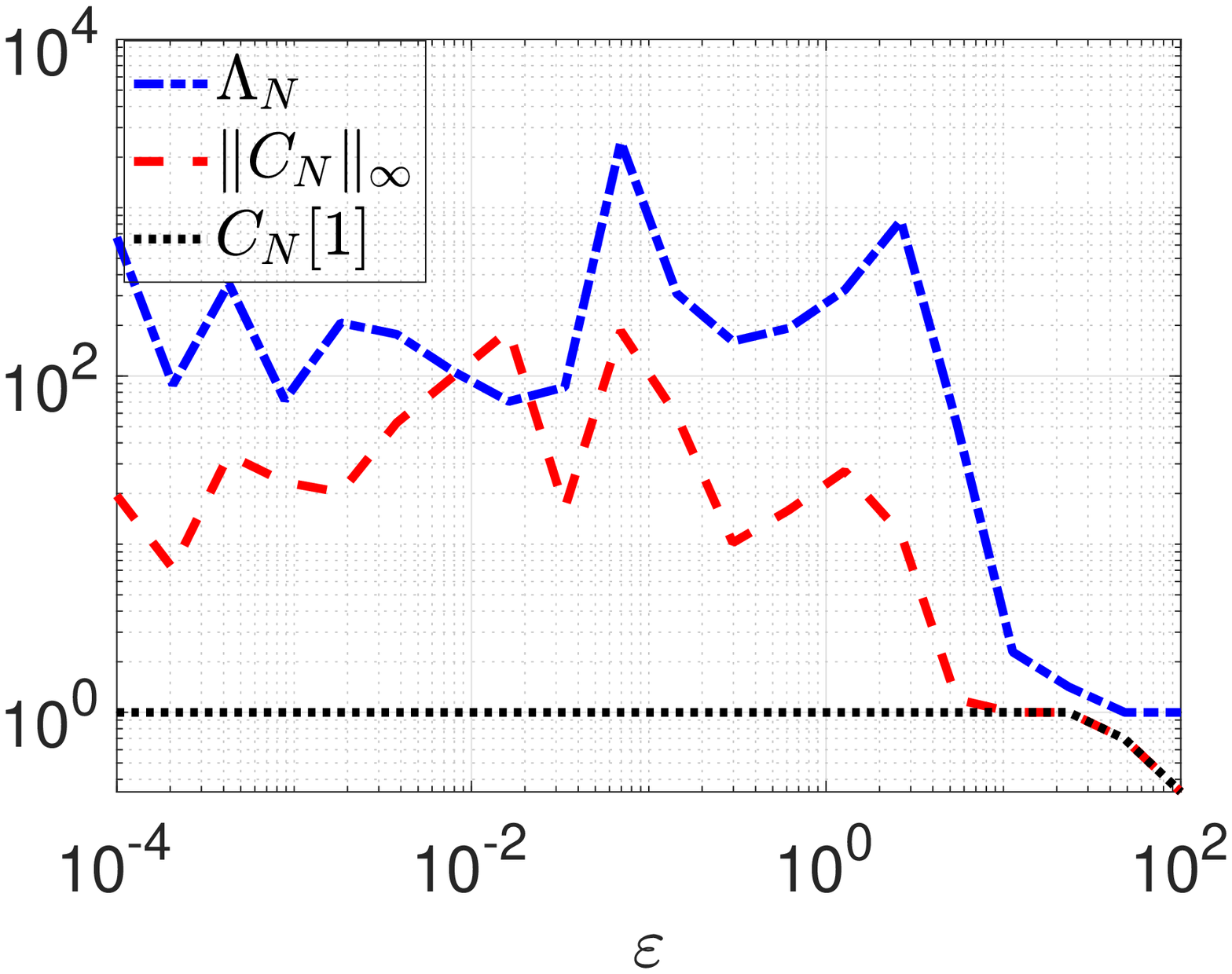} 
    \caption{Pure RBF interpolant/CF ($d=-1$)}
    \label{fig:stab_G_noPol}
  	\end{subfigure}%
  	~
  	\begin{subfigure}[b]{0.4\textwidth}
		\includegraphics[width=\textwidth]{%
      	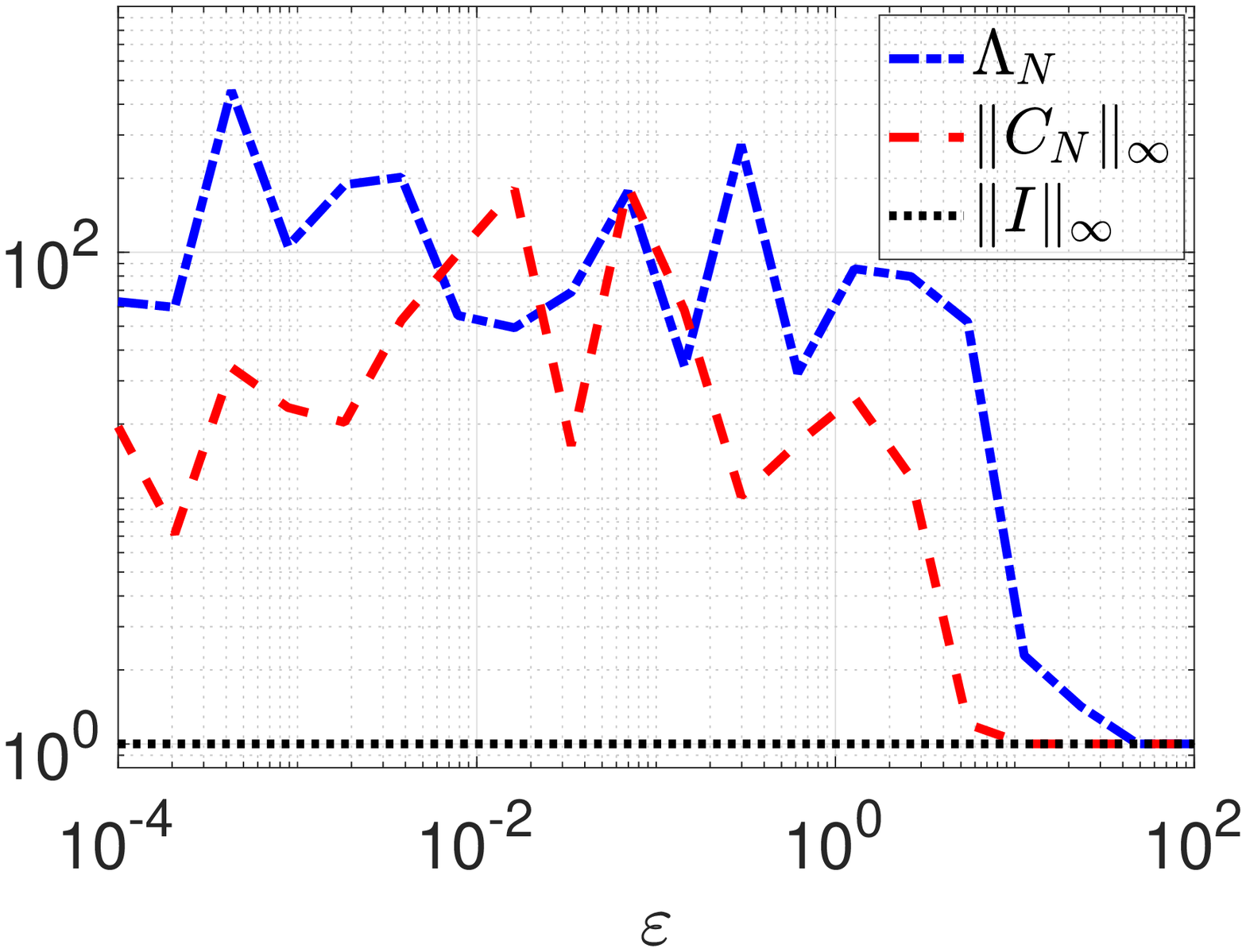} 
    \caption{Augmented by a constant ($d=0$)}
    \label{fig:stab_G_d0}
  	\end{subfigure}%
  	\\ 
  	\begin{subfigure}[b]{0.4\textwidth}
		\includegraphics[width=\textwidth]{%
      	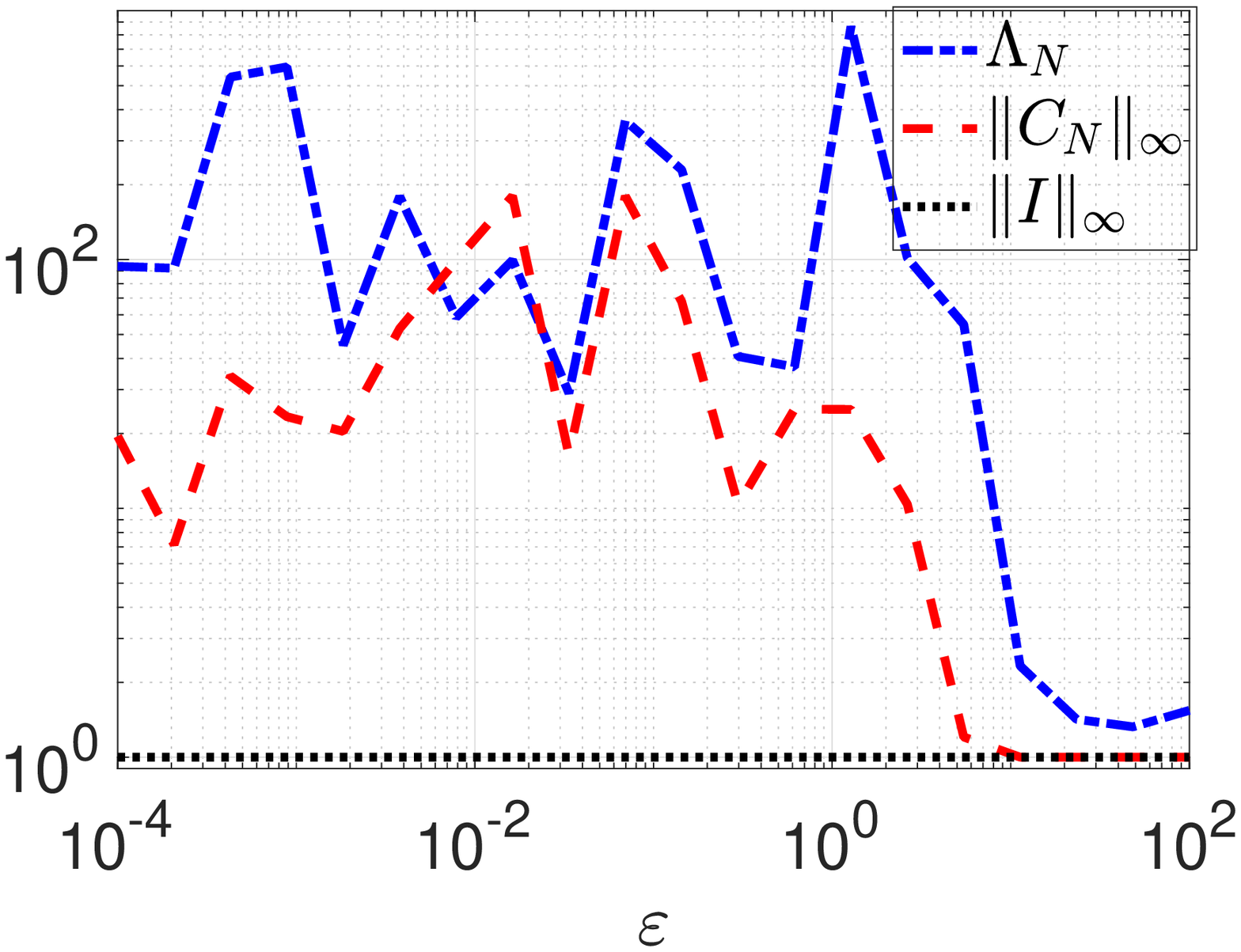} 
    \caption{Augmented by a linear term ($d=1$)}
    \label{fig:stab_G_d1}
  	\end{subfigure}%
	~
	\begin{subfigure}[b]{0.4\textwidth}
		\includegraphics[width=\textwidth]{%
      	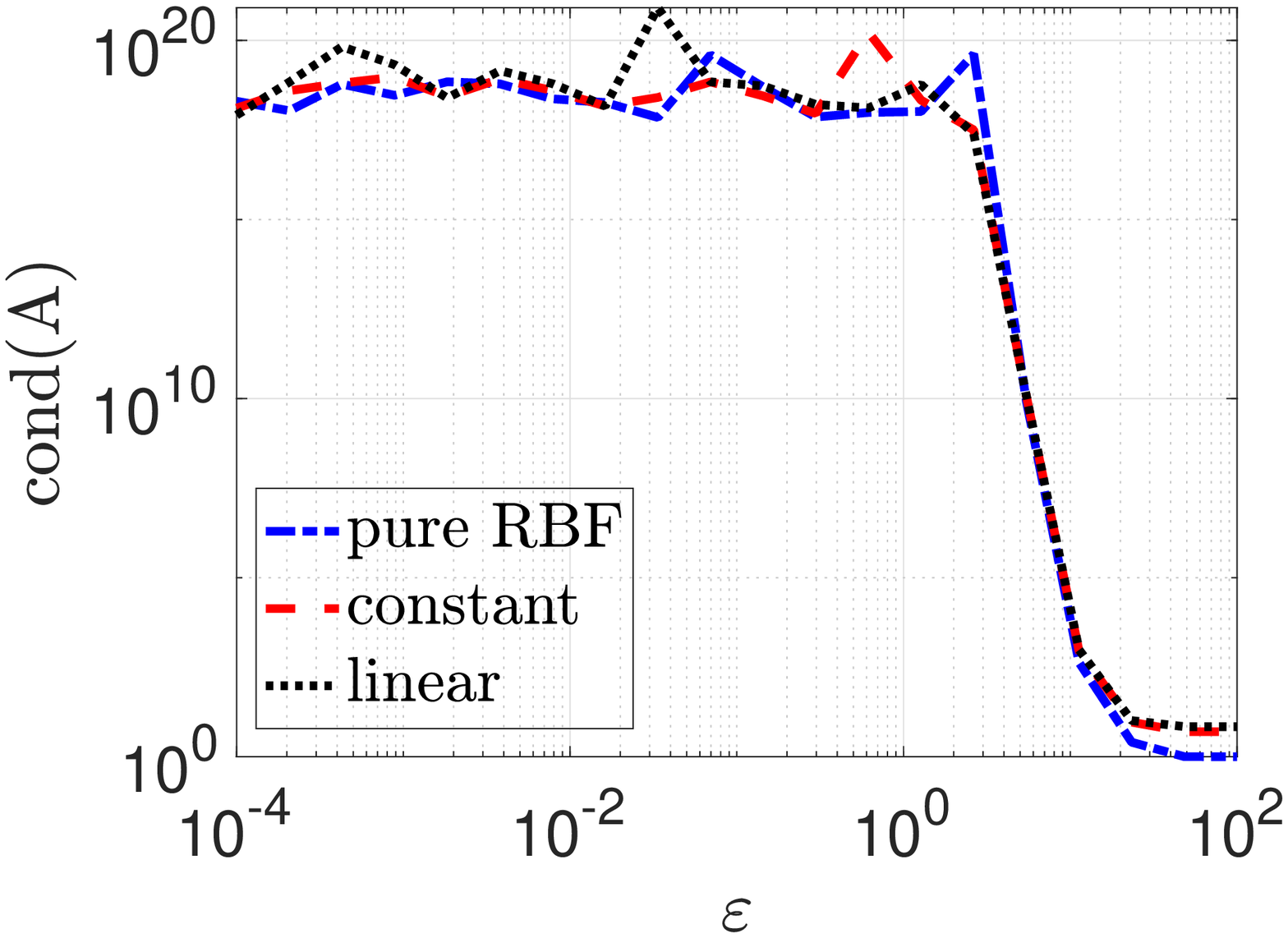} 
    \caption{Condition numbers}
    \label{fig:stab_G_cond}
  	\end{subfigure}%
  	\caption{
  	A comparison of the stability measure $\|C_N\|_{\infty}$, the Lebesgue constant $\Lambda_N$, and the condition number $\cond(A)$ for the Gaussian kernel. 
  	$N=20$ equidistant data points were considered, while the shape parameter $\varepsilon$ was allowed to vary. 
	Note that for the pure RBF interpolant/CF ($d=-1$), the optimal stability measure is $C_N[1]$ rather than $\|I\|_{\infty} = 1$. 
  	}
  	\label{fig:stab_G}
\end{figure} 

We start by demonstrating that RBF-CFs in many cases can have superior stability properties compared to RBF interpolants. 
This is demonstrated in \cref{fig:stab_G} for $\Omega = [0,1]$ and a Gaussian kernel $\varphi(r) = \exp( - \varepsilon^2 r^2 )$. 
The corresponding RBF approximation was either augmented with no polynomial terms (\cref{fig:stab_G_noPol}), a constant term (\cref{fig:stab_G_d0}), or a linear term (\cref{fig:stab_G_d1}).
See the caption of \cref{fig:stab_G} for more details.
The following observations can be made based on the results presented in \cref{fig:stab_G}:  
(1) RBF-based integration can be distinctly more stable than RBF-based interpolation. 
This is indicated by the stability measure $\| C_N \|_{\infty}$ often being smaller than the Lebesgue constant $\Lambda_N$. 
(2) Finding stable (nonnegative) RBF-CFs is a nontrivial task. 
Even though, in the tests presented here, we can observe certain regions of stability w.\,r.\,t.\ the shape parameter $\varepsilon$, it is not clear how to theoretically quantify the boundary of this region. 
A first step towards such an analysis is presented in \cref{sec:compact} for compactly supported RBFs. 
Further results in this direction would be of great interest. 
(3) There are two potential sources for negative weights, causing $\| C_N \|_{\infty} > C_N[1]$ and the RBF-CF to become sensitive towards input errors. 
On one hand, this can be caused by one (or multiple) of the cardinal functions having a negative moment. 
This is what we previously referred to as ``theoretical instability".
On the other hand, negative weights might also be caused by numerical ill-conditioning by the coefficient matrix $A$ in the linear system \cref{eq:LS_weights} that is numerically solved to compute the cubature weights. 
In fact, we can observe such numerical ill-conditioning in \cref{fig:stab_G_noPol} and \cref{fig:stab_G_d0}. 
In these figures, we have $\| C_N \|_{\infty} > \|I\|_{\infty} \Lambda_N$ (note that $\|I\|_{\infty} = 1$) for $\varepsilon \approx 10^{-2}$. 
Theoretically---assuming error-free computations---this should not happen. 
In accordance with this, \cref{fig:stab_G_cond} illustrates that in all cases ($d=-1,0,1$) the condition number of the matrix $A$, $\cond(A)$, reaches the upper bound of (decimal) double precision arithmetics ($\approx 10^{16}$) for $\varepsilon$ close to $10^0$. 

\begin{remark}[The Uncertainty Principle for Direct RBF Methods]
	Severe ill-conditioning of $A$ for flat RBFs (small shape parameters $\varepsilon$) is a well-known phenomenon in the RBF community. 
	At the same time, one often finds that the best accuracy for an RBF interpolant is achieved when $\varepsilon$ is small. 
	This so-called \emph{uncertainty} or \emph{trade-off principle} of (direct) RBF methods was first formulated in \cite{schaback1995error}. 
	Unfortunately, it has contributed to a widespread misconception that numerical ill-conditioning is unavoidable for flat RBFs. 
	It should be stressed that the uncertainty principle is specific to the direct RBF approach \cite{driscoll2002interpolation,fornberg2004some,larsson2005theoretical,schaback2005multivariate}. 
	That is, when $A$ is formulated w.\,r.\,t.\ the basis consisting of the translated RBFs, as described in \cref{eq:Phi_P}.
	Indeed, by now, numerous works have demonstrated that severe ill-conditioning of $A$ for flat RBFs can be remedied by formulating $A$ and the linear system \cref{eq:LS_weights} w.\,r.\,t\ to certain more stable bases spanning the RBF space $\mathcal{S}_{N,d}$. 
	See \cite{muller2009newton,pazouki2011bases,fornberg2011stable,fasshauer2012stable,de2013new,fornberg2013stable,wright2017stable} and references therein. 
	However, it should be noted that the linear system \cref{eq:LS_weights} used to determine the cubature weights of the RBF-CF requires knowledge of the moments of the basis that is used to formulate $A$. 
	This might be a potential bottleneck for some of the above-listed approaches. 
	A detailed discussion of how the moments of stable bases of $\mathcal{S}_{N,d}$ can be determined would therefore be of interest. 
\end{remark}

\begin{figure}[t]
  \centering
  	\begin{subfigure}[b]{0.33\textwidth}
		\includegraphics[width=\textwidth]{%
      	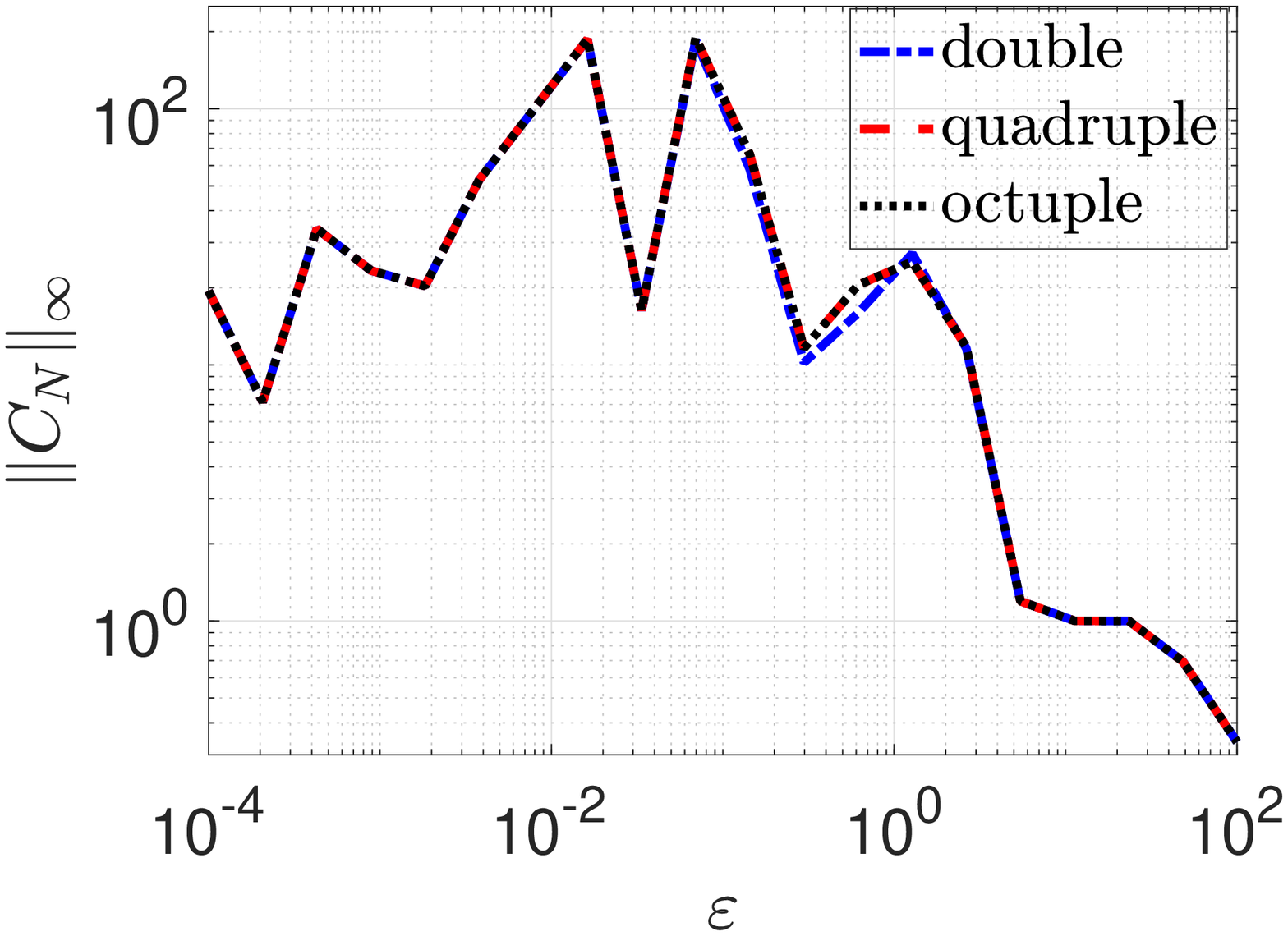} 
    \caption{Pure RBF ($d=-1$)}
    \label{fig:precision_G_noPol}
  \end{subfigure}%
  ~
  	\begin{subfigure}[b]{0.33\textwidth}
		\includegraphics[width=\textwidth]{%
      	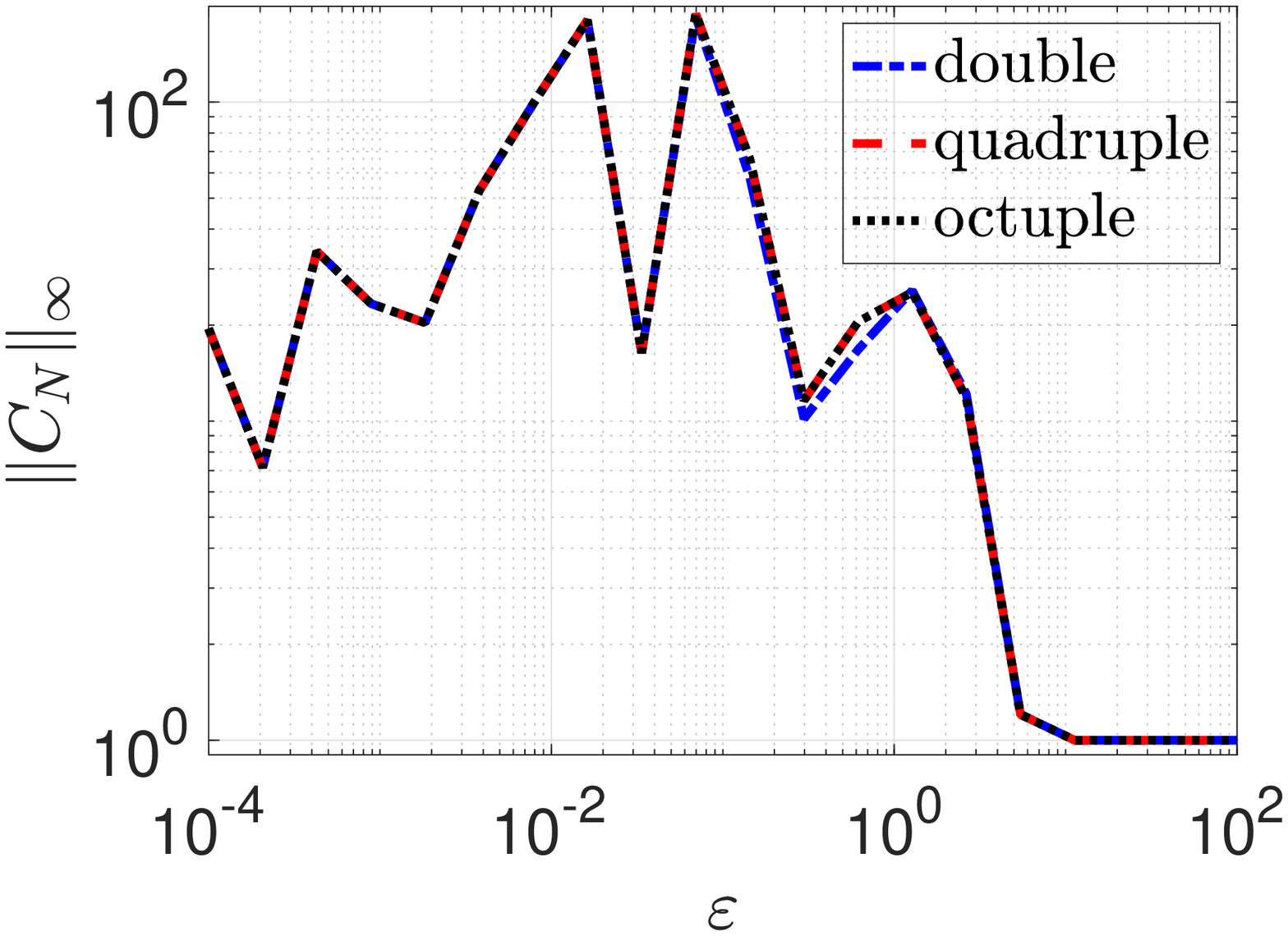} 
    \caption{Constant ($d=0$)}
    \label{fig:precision_G_d0}
  \end{subfigure}%
  ~
    	\begin{subfigure}[b]{0.33\textwidth}
		\includegraphics[width=\textwidth]{%
      	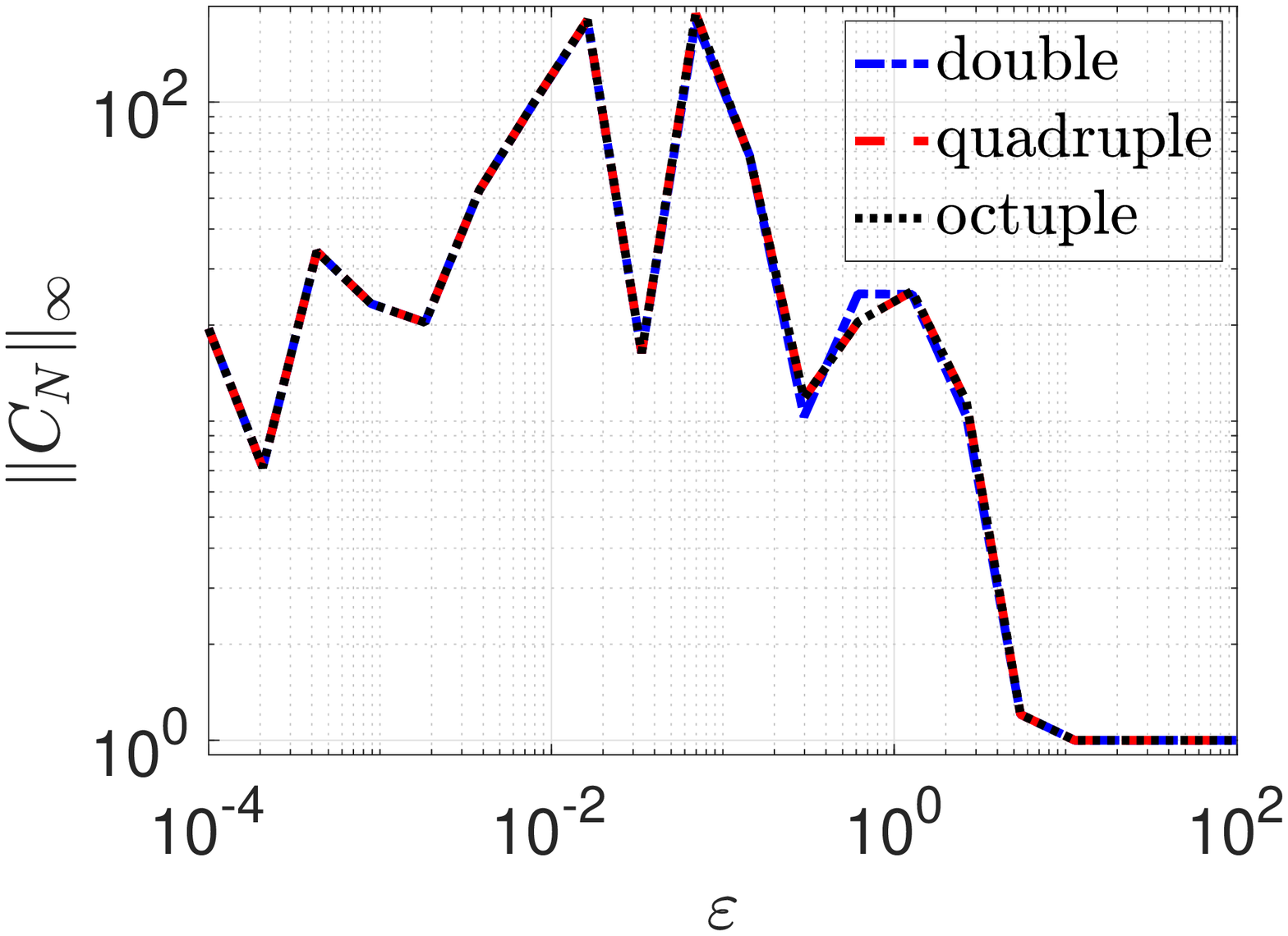} 
    \caption{Linear term ($d=1$)}
    \label{fig:precision_G_d1}
  \end{subfigure}%
  \caption{
  	Comparison of the stability measure $\|C_N\|_{\infty}$ for different computational precisions. 
	Considered are double (32 bits), quadruple (64 bits) and octuple (128 bits) precision. 
	In all cases $N=20$ equidistant data points and the Gaussian kernel were used. 
	The corresponding RBF interpolant either included no polynomial terms ($d=-1$), a constant ($d=0$) or a linear ($d=1$) term. 
	}
  \label{fig:precision_G}
\end{figure} 

The results presented in \cref{fig:stab_G} were obtained by the direct RBF method. 
One may therefore wonder to which extent the observed instabilities are influenced by numerical ill-conditioning. 
To address this question, we have repeated the same test with an increased computational precision using the function \emph{vpa} in MATLAB. 
\cref{fig:precision_G} provides a comparison of the stability measure $\|C_N\|_{\infty}$ computed by double (32 bits), quadruple (64 bits) and octuple (128 bits) precision. 
Despite $A$ being highly ill-conditioned, the results for quadruple precision might be considered as ``close" to the ones for usual double precision. 
In addition, further increasing the precision from quadruple to octuple precision does not seem to change the results---at least not by the naked eye. 
These results agree with the often reported observation that using stable solvers leads to useful results and well-behaved RBF interpolants even in the case of unreasonably large condition numbers. 
Indeed, we observe that the observed instabilities for RBF-CFs cannot be explained by numerical ill-conditioning alone. 
Rather, our results indicate that numerical ill-conditioning only amplifies already existing (theoretical) instabilities in the RBF-CF. 
\section{Compactly Supported Radial Basis Functions}
\label{sec:compact}

There is a rich body of literature on stability results for CFs based on (algebraic and trigonometric) polynomials, including \cite{haber1970numerical,stroud1971approximate,brass1977quadraturverfahren,engels1980numerical,cools1997constructing,krommer1998computational,krylov2006approximate,davis2007methods,brass2011quadrature} and the many references therein.
In comparison, provable results on the stability of RBF-CFs are rarely encountered in the literature, despite their increased use in applications. 
Here, our goal is to pave the way towards a more mature stability theory for these. 
As a first step in this direction, we next prove stability of RBF-CFs for compactly supported kernels with nonoverlapping supports. 
To be more precise, we subsequently consider RBFs $\varphi: \R_0^+ \to \R$ satisfying the following restrictions: 
\begin{enumerate}[label=(R\arabic*)] 
	\item \label{item:R1}
	$\varphi$ is nonnegative, i.\,e., $\varphi \geq 0$.
	
	\item \label{item:R2} 
	$\varphi$ is uniformly bounded. 
	W.\,l.\,o.\,g.\ we assume $\max_{r \in \R_0^+} |\varphi(r)| = 1$.
	
	\item \label{item:R3} 
	$\varphi$ is compactly supported. 
	W.\,l.\,o.\,g.\ we assume $\operatorname{supp} \varphi = [0,1]$.
	
\end{enumerate}
Already note that \ref{item:R3} implies $\operatorname{supp} \varphi_n = B_{\varepsilon_{n}^{-1}}(\mathbf{x}_n)$, where 
\begin{equation}
	B_{\varepsilon_{n}^{-1}}(\mathbf{x}_n) := \{ \, \mathbf{x} \in \Omega \mid \| \mathbf{x}_n - \mathbf{x} \|_2 \leq \varepsilon_{n}^{-1} \, \}, 
	\quad  
	\varphi_n(\boldsymbol{x}) := \varphi( \varepsilon_{n} \| \mathbf{x}_n - \boldsymbol{x} \|_2 ).
\end{equation}
Clearly, the $\varphi_n$'s will have nonoverlapping support if the shape parameters $\varepsilon_{n}$ are sufficiently large. 
This can be ensured by the following condition: 
\begin{equation}\label{eq:R4}
	\varepsilon_{n}^{-1} \leq h_{n} 
		:= \min\left\{ \, \| \mathbf{x}_n - \mathbf{x}_m \|_2 \mid \mathbf{x}_m \in X \setminus \{\mathbf{x}_n\} \, \right\}, 
		\quad n=1,\dots,N
\end{equation} 
Here, $X$ denotes the set of data points. 
The different basis functions having nonoverlapping support might seem to be a fairly restrictive sufficient condition. 
However, our numerical tests presented in \cref{sec:numerical} indicate that this condition does not seem to be ``far away" from being necessary as well. 
This might be considered as a discouraging result for the utility of compactly supported RBFs in the context of numerical integration. 
Finally, it should be pointed out that throughout this section, we assume $\omega \equiv 1$. 
This assumption is made for the main result, \cref{thm:main}, to hold. 
Its role will become clearer after consulting the proof of \cref{thm:main} and is revisited in \cref{rem:omega}.

\subsection{Main Results}
\label{sub:compact_main}

Our main result is the following \cref{thm:main}.
It states that RBF-CFs are conditionally stable for any polynomial degree $d \in \N$ if the number of (equidistributed) data points, $N$, is sufficiently larger than $d$. 

\begin{theorem}[Conditional Stability of RBF-CFs]\label{thm:main}
	Let $(\mathbf{x}_n)_{n \in \N}$ be an equidistributed sequence in $\Omega$ and $X_N = \{ \mathbf{x}_n \}_{n=1}^N$. 
	Furthermore, let $\omega \equiv 1$, let $\varphi: \R_0^+ \to \R$ be a RBF satisfying \ref{item:R1} to \ref{item:R3}, and choose the shape parameters $\varepsilon_n$ such that the corresponding functions $\varphi_n$ have nonoverlapping support and equal moments ($I[\varphi_n] = I[\varphi_m]$ for all $n,m=1,\dots,N$). 
	For every polynomial degree $d \in \N$ there exists an $N_0 \in \N$ such that for all $N \geq N_0$ the corresponding RBF-CF \cref{eq:RBF-CRs} is stable. 
	That is, $I[c_m] \geq 0$ for all $m=1,\dots,N$. 
\end{theorem}

The proof of \cref{thm:main} is given in \cref{sub:compact_proof} after collecting a few preliminarily results. 

Note that a sequence $(\mathbf{x}_n)_{n \in \N}$ is \emph{equidistributed in $\Omega$} if and only if 
\begin{equation}
	\lim_{N \to \infty} \frac{|\Omega|}{N} \sum_{n=1}^N g(\mathbf{x}_n) 
		= \int_{\Omega} g(\boldsymbol{x}) \intd \boldsymbol{x}
\end{equation}
holds for all measurable bounded functions $g: \Omega \to \R$ that are continuous almost everywhere (in the sense of Lebesgue), see \cite{weyl1916gleichverteilung}. 
For details on equidistributed sequences, we refer to the monograph \cite{kuipers2012uniform}.
Still, it should be noted that equidistributed sequences are dense sequences with a special ordering. 
In particular, if $(\mathbf{x}_n)_{n \in \N} \subset \Omega$ is equidistributed, then for every $d \in \N$ there exists an $N_0 \in \N$ such that $X_N$ is $\mathbb{P}_d(\Omega)$-unisolvent for all $N \geq N_0$; see \cite{glaubitz2021construction}. 
This ensures that the corresponding RBF interpolant is well-defined. 

It should also be noted that if $\Omega \subset \R^D$ is bounded and has a boundary of measure zero (again in the sense of Lebesgue), then an equidistributed sequence in $\Omega$ is induced by every equidistributed sequence in the $D$-dimensional hypercube. 
Since $\Omega$ is bounded, we can find an $R > 0$ such that $\Omega \subset [-R,R]^D$. 
Let $(\mathbf{y}_n)_{n \in \N}$ be an equidistributed sequence in $[-R,R]^D$.\footnote{Examples for such sequences include certain equidistant, (scaled and translated) Halton \cite{halton1960efficiency} or some other low-discrepancy points \cite{hlawka1961funktionen,niederreiter1992random,caflisch1998monte,dick2013high}.}  
Next, define $(\mathbf{x}_n)_{n \in \N}$ as the subsequence of $(\mathbf{y}_n)_{n \in \N} \subset [-R,R]^D$ that only contains the points inside of $\Omega$. 
It was shown in \cite{glaubitz2021construction} that this results in $(\mathbf{x}_n)_{n \in \N}$ being equidistributed in $\Omega$ if $\partial \Omega$ is of measure zero.

\subsection{Explicit Representation of the Cardinal Functions}
\label{sub:compact_explicit}

In preparation of proving \cref{thm:main} we derive an explicit representation for the cardinal functions $c_n$ under the restrictions \ref{item:R1} to \ref{item:R3} and \cref{eq:R4}. 
In particular, we make use of the concept of discrete orthogonal polynomials. 
Let us define the following discrete inner product corresponding to the data points $X_N = \{\mathbf{x}_n\}_{n=1}^N$: 
\begin{equation}\label{eq:discrete_scp}
	[u,v]_{X_N} = \frac{|\Omega|}{N} \sum_{n=1}^N u(\mathbf{x}_n) v(\mathbf{x}_n)
\end{equation} 
Recall that the data points $X_N$ are coming from an equidistributed sequence and are therefore ensured to be $\mathbb{P}_d(\Omega)$-unisolvent for any degree $d \in \N$ if a sufficiently large number of data points is used. 
In this case, \cref{eq:discrete_scp} is therefore ensured to be positive definite on $\mathbb{P}_d(\Omega)$. 
We say that the basis $\{p_k\}_{k=1}^K$ of $\mathbb{P}_d(\Omega)$, where $K = \dim \mathbb{P}_d(\Omega)$, consists of \emph{discrete orthogonal polynomials (DOPs)} if they satisfy 
\begin{equation}
	[p_k,p_l]_{X_N} = \delta_{kl} := 
	\begin{cases} 
		1 & \text{ if } k=l, \\ 
		0 & \text{ otherwise}, 
	\end{cases} 
	\quad k,l=1,\dots,K.
\end{equation} 
We now come to the desired explicit representation for the cardinal functions $c_m$. 

\begin{lemma}[Explicit Representation for $c_m$]\label{lem:rep_cm}
	Let the RBF $\varphi: \R_0^+ \to \R$ satisfy \ref{item:R2} and \ref{item:R3}. 
	Furthermore, choose the shape parameters $\varepsilon_n$ such that the corresponding functions $\varphi_n$ have nonoverlapping support and let the basis $\{p_k\}_{k=1}^K$ consists of DOPs. 
	Then, the cardinal function $c_m$, $m=1,\dots,N$, is given by 
	\begin{equation}\label{eq:rep_cm}
	\begin{aligned}
		c_m(\boldsymbol{x}) 
			= \varphi_m(\boldsymbol{x}) 
			- \frac{|\Omega|}{N} \sum_{n=1}^N \left( \sum_{k=1}^K p_k(\mathbf{x}_m) p_k(\mathbf{x}_n) \right) \varphi_n(\boldsymbol{x}) 
			+ \frac{|\Omega|}{N} \sum_{k=1}^K p_k(\mathbf{x}_m) p_k(\boldsymbol{x}). 
	\end{aligned}
	\end{equation} 
\end{lemma}

\begin{proof} 
	Let $m,n \in \{1,\dots,N\}$. 
	The restrictions \ref{item:R2}, \ref{item:R3} together with the assumption of the $\varphi_n$'s having nonoverlapping support yields $\varphi_n(\mathbf{x}_m) = \delta_{mn}$.
	Hence, \cref{eq:cardinal} and \cref{eq:cond_cardinal} imply 
	\begin{equation}\label{eq:alpha}
		\alpha_n^{(m)} = \delta_{mn} - \sum_{k=1}^K \beta^{(m)}_k p_k(\mathbf{x}_n).
	\end{equation} 
	If we substitute \cref{eq:alpha} into \cref{eq:cond2}, we get 
	\begin{equation} 
		p_l(\mathbf{x}_m) - \frac{N}{|\Omega|} \sum_{k=1}^K \beta^{(m)}_k [p_k,p_l]_{X_N} = 0, 
		\quad l=1,\dots,K. 
	\end{equation} 
	Thus, if $\{p_k\}_{k=1}^K$ consists of DOPs, this gives us 
	\begin{equation}\label{eq:beta}
		\beta^{(m)}_l = \frac{N}{|\Omega|} p_l(\mathbf{x}_m), \quad l=1,\dots,K.
	\end{equation} 
	Finally, substituting \cref{eq:beta} into \cref{eq:alpha} yields 
	\begin{equation} 
		\alpha_n^{(m)} = \delta_{mn} - \frac{N}{|\Omega|} \sum_{k=1}^K p_k(\mathbf{x}_m) p_k(\mathbf{x}_n) 
	\end{equation} 
	and therefore the assertion. 
\end{proof}

We already remarked that using a basis consisting of DOPs is not necessary for the implementation of RBF-CFs. 
In fact, the cubature weights are, ignoring computational considerations, independent of the polynomial basis elements w.\,r.\,t.\ which the matrix $P$ and the corresponding moments $\mathbf{m}^{\text{poly}}$ are formulated. 
We only use DOPs as a theoretical tool---a convenient perspective on the problem at hand\footnote{For example, many properties of interpolation polynomials are shown by representing these w.\,r.\,t.\ the Lagrange basis, while this representation is often not recommended for actual computations.}---to show stability of RBF-CFs.

\subsection{Some Low Hanging Fruits}
\label{sub:compact_low}

Using the explicit representation \cref{eq:rep_cm} it is trivial to prove stability of RBF-CFs when no polynomial term or only a constant is included in the RBF interpolant. 

\begin{lemma}[No Polynomials]
	Let the RBF $\varphi: \R_0^+ \to \R$ satisfy \ref{item:R1} to \ref{item:R3} and choose the shape parameters $\varepsilon_n$ such that the corresponding functions $\varphi_n$ have nonoverlapping support.
	Assume that no polynomials are included in the corresponding RBF interpolant ($K=0$). 
	Then, the associated RBF-CF is stable, i.e., $I[c_m] \geq 0$ for all $m=1,\dots,N$.
\end{lemma}

\begin{proof} 
	It is obvious that $c_m(\boldsymbol{x}) = \varphi_m(\boldsymbol{x})$. 
	Thus, by restriction \ref{item:R1}, $c_m$ is nonnegative and therefore $I[c_m] \geq 0$.
\end{proof}

\begin{lemma}[Only a Constant]
	Let the RBF $\varphi: \R_0^+ \to \R$ satisfy \ref{item:R1} to \ref{item:R3} and choose the shape parameters $\varepsilon_n$ such that the corresponding functions $\varphi_n$ have nonoverlapping support.
	Assume that only a constant is included in the corresponding RBF interpolant ($d=0$ or $K=1$). 
	Then, the associated RBF-CF is stable, i.e., $I[c_m] \geq 0$ for all $m=1,\dots,N$.
\end{lemma}

\begin{proof} 
	Let $m \in \{1,\dots,N\}$. 
	If we choose $p_1 \equiv |\Omega|^{-1/2}$, \cref{lem:rep_cm} yields 
	\begin{equation} 
		c_m(\boldsymbol{x}) 
			= \varphi_m(\boldsymbol{x}) 
			+ \frac{1}{N} \left( 1 - \sum_{n=1}^N \varphi_n(\boldsymbol{x}) \right). 
	\end{equation} 
	Note that by \ref{item:R2}, \ref{item:R3}, and \cref{eq:R4}, we therefore have $c_m(\boldsymbol{x}) \geq \varphi_m(\boldsymbol{x})$. 
	Hence, \ref{item:R1} implies the assertion. 
\end{proof}

\subsection{Proof of the Main Results}
\label{sub:compact_proof} 

The following technical lemma will be convenient to the proof of \cref{thm:main}. 

\begin{lemma}\label{lem:technical}
	Let $(\mathbf{x}_n)_{n \in \N}$ be equidistributed in $\Omega$, $X_N = \{ \mathbf{x}_n \}_{n=1}^N$, and let $[\cdot,\cdot]_{X_N}$ be the discrete inner product \cref{eq:discrete_scp}. 
	Furthermore, let $\{ p_k^{(N)} \}_{k=1}^K$ be a basis of $\mathbb{P}_d(\Omega)$ consisting of DOPs w.\,r.\,t.\ $[\cdot,\cdot]_{X_N}$. 
	Then, for all $k=1,\dots,K$, 
	\begin{equation}
		p_k^{(N)} \to p_k \quad \text{in } L^{\infty}(\Omega), \quad N \to \infty,
	\end{equation} 
	where $\{ p_k \}_{k=1}^K$ is a basis of $\mathbb{P}_d(\Omega)$ consisting of continuous orthogonal polynomials satisfying 
	\begin{equation} 
		\int_{\Omega} p_k(\boldsymbol{x}) p_l(\boldsymbol{x}) \intd \boldsymbol{x} 
			= \delta_{kl}, \quad k,l=1,\dots,K.
	\end{equation}
	Moreover, it holds that 
	\begin{equation} 
		\lim_{N \to \infty} \int_{\Omega} p_k^{(N)}(\boldsymbol{x}) p_l^{(N)}(\boldsymbol{x}) \intd \boldsymbol{x} = \delta_{kl}, \quad k,l=1,\dots,K.
	\end{equation}
\end{lemma}

\begin{proof}
	The assertion is a direct consequence of the results from \cite{glaubitz2020stableCFs}. 
\end{proof}

Essentially, \cref{lem:technical} states that if a sequence of discrete inner product converges to a continuous one, then also the corresponding DOPs---assuming that the ordering of the elements does not change---converges to a basis of continuous orthogonal polynomials. 
Furthermore, this convergence also holds in a uniform sense. 
We are now able to provide a proof for \cref{thm:main}. 

\begin{proof}[Proof of \cref{thm:main}]
	Let $d \in \N$ and $m \in \{1,\dots,N\}$. 
	Under the assumptions of \cref{thm:main}, we have $I[\varphi_n] = I[\varphi_m]$ for all $n=1,\dots,N$. 
	Thus, \cref{lem:rep_cm} implies  
	\begin{equation}
		I[c_m] 
			= I[\varphi_m] \left[ 1 - \frac{|\Omega|}{N} \sum_{n=1}^N \sum_{k=1}^K p^{(N)}_k(\mathbf{x}_m) p^{(N)}_k(\mathbf{x}_{n}) \right] 
			+ \frac{|\Omega|}{N} \sum_{k=1}^K p^{(N)}_k(\mathbf{x}_m) I[p_k]. 
	\end{equation} 
	Let $\{ p_k^{(N)} \}_{k=1}^K$ be a basis of $\mathbb{P}_d(\Omega)$ consisting of DOPs. 
	That is, $[p_k^{(N)},p_l^{(N)}]_{X_N} = \delta_{kl}$. 
	In particular, $p_1^{(N)} \equiv |\Omega|^{-1/2}$. 
	With this in mind, it is easy to verify that 
	\begin{equation}\label{eq:omega_proof1} 
	\begin{aligned} 
		\frac{|\Omega|}{N} \sum_{n=1}^N \sum_{k=1}^K p^{(N)}_k(\mathbf{x}_m) p^{(N)}_k(\mathbf{x}_n) 
		= \sum_{k=1}^K p^{(N)}_k(\mathbf{x}_m) |\Omega|^{1/2} [p^{(N)}_k,p^{(N)}_1]_{X_N}
		= 1. 
	\end{aligned}
	\end{equation} 
	Thus, we have 
	\begin{equation} 
		I[c_m] \geq 0 \iff 
			\sum_{k=1}^K p_k^{(N)}(\mathbf{x}_m) I[p_k^{(N)}] \geq 0.
	\end{equation} 
	Finally, observe that 
	\begin{equation}
		\sum_{k=1}^K p_k^{(N)}(\mathbf{x}_m) I[p_k^{(N)}] 
			= |\Omega|^{1/2} \sum_{k=1}^K p_k^{(N)}(\mathbf{x}_m) \int_{\Omega} p_k^{(N)}(\boldsymbol{x}) p_1^{(N)}(\boldsymbol{x}) \intd \boldsymbol{x}, 
	\end{equation} 
	under the assumption that $\omega \equiv 1$. 
	\cref{lem:technical} therefore implies 
	\begin{equation}\label{eq:omega_proof2} 
		\lim_{N \to \infty} \sum_{k=1}^K p_k^{(N)}(\mathbf{x}_m) I[p_k^{(N)}] = 1, 
	\end{equation}
	which completes the proof. 
\end{proof} 

\begin{remark}[On the Assumption that $\omega \equiv 1$]\label{rem:omega}
	The assumption that $\omega \equiv 1$ in \cref{thm:main} is necessary for \cref{eq:omega_proof1} and \cref{eq:omega_proof2} to both hold true. 
	On the one hand, \cref{eq:omega_proof1} is ensured by the the DOPs being orthogonal w.\,r.\,t.\ the discrete inner product \cref{eq:discrete_scp}. 
	This discrete inner product can be considered as an approximation to the continuous inner product ${\scp{u}{v} = \int_{\Omega} u(\boldsymbol{x}) v(\boldsymbol{x}) \intd \boldsymbol{x}}$. 
	This also results in \cref{lem:technical}. 
	On the other hand, in general, \cref{eq:omega_proof2} only holds if the DOPs converge to a basis of polynomials that is orthogonal w.\,r.\,t.\ the weighted continuous inner product ${\scp{u}{v}_{\omega} = \int_{\Omega} u(\boldsymbol{x}) v(\boldsymbol{x}) \omega(\boldsymbol{x}) \intd \boldsymbol{x}}$. 
	Hence, for \cref{eq:omega_proof1} and \cref{eq:omega_proof2} to both hold true at the same time, we have to assume that $\omega \equiv 1$. 
	In this case, the two continuous inner products are the same.
\end{remark}

\section{On the Connection Between RBF-CFs With and Without Polynomials} 
\label{sec:connection} 

A natural question in the context of RBFs is which influence the polynomial terms have on the quality of the RBF interpolation and the RBF-CF, beyond ensuring existence of the RBF interpolant. 
In particular, in the context of the present work, one might ask ``how are polynomial terms influencing stability of the RBF-CF?".
In what follows, we address this question by showing that---under certain assumptions that are to be specified yet---at least asymptotic stability of RBF-CFs is independent of polynomial terms. 
We hope this result to be another step forward towards a more mature stability theory for RBF-CFs. 

Recently, the following explicit formula for the cardinal functions was derived in \cite{bayona2019insight,bayona2019comparison}. 
Let us denote ${\mathbf{c}(\boldsymbol{x}) = [c_1(\boldsymbol{x}),\dots,c_N(\boldsymbol{x})]^T}$, where $c_1,\dots,c_N$ are the cardinal functions spanning $\mathcal{S}_{N,d}$; see \cref{eq:cardinal} and \cref{eq:cond_cardinal}. 
Provided that $\Phi$ and $P$ in \cref{eq:Phi_P} have full rank\footnote{
$P$ having full rank means that $P$ has full column rank, i.\,e., the columns of $P$ are linearly independent. 
This is equivalent to the set of data points being $\mathbb{P}_d(\Omega)$-unisolvent. 
}, 
\begin{equation}\label{eq:formula_Bayona}
	\mathbf{c}(\boldsymbol{x}) 
		= \hat{\mathbf{c}}(\boldsymbol{x}) 
		- B \boldsymbol{\tau}(\boldsymbol{x})
\end{equation} 
holds. 
Here, $\hat{\mathbf{c}}(\boldsymbol{x}) = [\hat{c}_1(\boldsymbol{x}),\dots,\hat{c}_N(\boldsymbol{x})]^T$ are the cardinal functions corresponding to the pure RBF interpolation without polynomials. 
That is, they span $\mathcal{S}_{N,-1}$. 
At the same time, $B$ and $\boldsymbol{\tau}$ are defined as 
\begin{equation} 
	B := \Phi^{-1} P \left( P^T \Phi^{-1} P \right)^{-1}, 
	\quad 
	\boldsymbol{\tau}(\boldsymbol{x}) := P^T \hat{\mathbf{c}}(\boldsymbol{x}) - \mathbf{p}(\boldsymbol{x})
\end{equation} 
with ${\mathbf{p}(\boldsymbol{x}) = [p_1(\boldsymbol{x}),\dots,p_K(\boldsymbol{x})]^T}$. 
Note that $\boldsymbol{\tau}$ can be interpreted as a residual measuring how well pure RBFs can approximate polynomials up to degree $d$. 
Obviously, \cref{eq:formula_Bayona} implies 
\begin{equation}\label{eq:relation_weights}
	\mathbf{w} 
		= \hat{\mathbf{w}} - B I[\boldsymbol{\tau}],
\end{equation} 
where $\mathbf{w}$ is the vector of cubature weights of the RBF-CF with polynomials ($d \geq 0$). 
At the same time, $\hat{\mathbf{w}}$ is the vector of weights corresponding to the pure RBF-CF without polynomial augmentation ($d=-1$). 
Moreover, $I[\boldsymbol{\tau}]$ denotes the componentwise application of the integral operator $I$.
It was numerically demonstrated in \cite{bayona2019insight} that for fixed $d \in \N$ 
\begin{equation}\label{eq:observation_Bayona}
	\max_{\boldsymbol{x} \in \Omega} \| B \boldsymbol{\tau}(\boldsymbol{x}) \|_{\ell^\infty} \to 0 
	\quad \text{as} \quad N \to \infty
\end{equation}
if PHS are used. 
Note that, for fixed $\boldsymbol{x} \in \Omega$, $B \boldsymbol{\tau}(\boldsymbol{x})$ is an $N$-dimensional vector and $\| B \boldsymbol{\tau}(\boldsymbol{x}) \|_{\ell^\infty}$ denotes its $\ell^{\infty}$-norm. 
That is, the maximum absolute value of the $N$ components. 
It should be pointed out that while \cref{eq:observation_Bayona} was numerically demonstrated only for PHS the relations \cref{eq:formula_Bayona} and \cref{eq:relation_weights} hold for general RBFs, assuming that $\Phi$ and $P$ have full rank.
Please see \cite[Section 4]{bayona2019insight} for more details. 
We also remark that \cref{eq:observation_Bayona} implies the weaker statement 
\begin{equation}\label{eq:cond}
	\| B \boldsymbol{\tau}(\cdot) \|_{\ell^1} \to 0 
	\ \ \text{in } L^1(\Omega) \quad \text{as} \quad N \to \infty.
\end{equation} 
Here, $B \boldsymbol{\tau}(\cdot)$ denotes a vector-valued function, $B \boldsymbol{\tau}: \Omega \to \R^N$. 
That is, for a fixed argument $\boldsymbol{x} \in \Omega$, $B \boldsymbol{\tau}(\boldsymbol{x})$ is an $N$-dimensional vector in $\R^N$ and $\| B \boldsymbol{\tau}(\boldsymbol{x}) \|_{\ell^1}$ denotes the usual $\ell^1$-norm of this vector.  
Thus, \cref{eq:cond} means that the integral of the $\ell^1$-norm of the vector-valued function $B \boldsymbol{\tau}(\cdot)$ converges to zero as $N \to \infty$.
The above condition is not just weaker than \cref{eq:observation_Bayona} (see \cref{rem:assumption2}), but also more convenient to investigate stability of CFs. 
Indeed, we have the following results.

\begin{lemma}\label{lem:connection} 
	Let $\omega \in L^{\infty}(\Omega)$.
	Assume $\Phi$ and $P$ in \cref{eq:Phi_P} have full rank and assume \cref{eq:cond} to hold. 
	Then the two following statements are equivalent: 
	\begin{enumerate} 
		\item[(a)] 
		$\| \hat{\mathbf{w}} \|_{\ell^1} \to \|I\|_{\infty}$ for $N \to \infty$ 
		
		\item[(b)]
		$\| \mathbf{w} \|_{\ell^1} \to \|I\|_{\infty}$ for $N \to \infty$
	\end{enumerate} 
	That is, either both the pure and polynomial augmented RBF-CF are asymptotically stable or none is. 
\end{lemma}

A short discussion on the term ``asymptotically stable" is subsequently provided in \cref{rem:asymptotic_stable}. 

\begin{proof} 
	Assume $\Phi$ and $P$ in \cref{eq:Phi_P} have full rank and assume \cref{eq:cond} to hold. 
	Then \cref{eq:relation_weights} follows and therefore 
	\begin{equation}\label{eq:connection_proof} 
	\begin{aligned}
		\| \mathbf{w} \|_{\ell^1}
			& \leq \| \hat{\mathbf{w}} \|_{\ell^1} + \| BI[\boldsymbol{\tau}] \|_{\ell^1}, \\ 
		\| \hat{\mathbf{w}} \|_{\ell^1}
			& \leq \| \mathbf{w} \|_{\ell^1} + \| BI[\boldsymbol{\tau}] \|_{\ell^1}.
	\end{aligned}
	\end{equation} 
	Next, note that $BI[\boldsymbol{\tau}] = I[ B \boldsymbol{\tau}]$ and thus 
	\begin{equation} 
	\begin{aligned} 
		\| BI[\boldsymbol{\tau}] \|_{\ell^1} 
			= \sum_{n=1}^N \left| I[ (B \boldsymbol{\tau})_n ] \right|  
			\leq I \left[ \sum_{n=1}^N | (B \boldsymbol{\tau})_n | \right] 
			= I \left[ \| B \boldsymbol{\tau} \|_{\ell^1} \right]. 
	\end{aligned}
	\end{equation} 
	Since $\omega \in L^{\infty}(\Omega)$ it follows that 
	\begin{equation} 
		\| BI[\boldsymbol{\tau}] \|_{\ell^1} 
			\leq \| \omega \|_{L^{\infty}(\Omega)} \int_{\Omega} \| B \boldsymbol{\tau}(\boldsymbol{x}) \|_{\ell^1} \intd \boldsymbol{x}.
	\end{equation} 
	Thus, by assuming that \cref{eq:cond} holds, we get $\| BI[\boldsymbol{\tau}] \|_{\ell^1} \to 0$ for fixed $d \in \N$ and $N \to \infty$. 
	Finally, substituting this into \cref{eq:connection_proof} yields the assertion. 
\end{proof}

Essentially, \cref{lem:connection} states that--under the listed assumptions---it is sufficient to consider asymptotic stability of the pure RBF-CF. 
Once asymptotic (in)stability is established for the pure RBF-CF, by \cref{lem:connection}, it also carries over to all corresponding augmented RBF-CFs. 
Interestingly, this is following our findings for compactly supported RBFs reported in \cref{thm:main}. 
There, conditional stability was ensured independently of the degree of the augmented polynomials.

\begin{remark}[Asymptotic Stability]\label{rem:asymptotic_stable}
	We call a sequence of CFs with weights $\mathbf{w}_N \in \R^N$ for $N \in \N$ asymptotically stable if $\| \mathbf{w}_N \|_{\ell^1} \to \| I \|_{\infty}$ for $N \to \infty$. 
	Recall that $\| \mathbf{w}_N \|_{\ell^1} = \|C_N\|_{\infty}$ if the weights $\mathbf{w}_N$ correspond to the $N$-point CF $C_N$. 
	It is easy to note that this is a weaker property than every single CF being stable, i.\,e., $\| \mathbf{w}_N \|_{\ell^1} = \| I \|_{\infty}$ for all $N \in \N$. 
	That said, consulting \cref{eq:L-inequality}, asymptotic stability is sufficient for the CF to converge for all functions that can be approximated arbitrarily accurate by RBFs w.\,r.\,t.\ the $L^{\infty}(\Omega)$-norm. 
	Of course, the propagation of input errors might be suboptimal for every single CF. 
\end{remark}

\cref{lem:connection} essentially makes two assumptions. 
(1) $A$ and $P$ are full rank matrices on the data set of data points; and 
(2) the condition \cref{eq:observation_Bayona} holds. 
In the two following remarks, we comment on these assumptions. 

\begin{remark}[On the First Assumption of \cref{lem:connection}]\label{rem:assumption1}
	Although it might seem restrictive to require $A$ and $P$ to have full rank, there are often even more restrictive constraints in practical problems. 
	For instance, when solving partial differential equations, the data points are usually required to be smoothly scattered in such a way that the distance between data points is kept roughly constant. 
	For such data points, it seems unlikely to find $A$ and $P$ (for $N$ being sufficiently larger than $d$) to be singular. 
	See \cite{bayona2019insight} for more details. 
\end{remark}

\begin{remark}[On the Second Assumption of \cref{lem:connection}]\label{rem:assumption2} 
	The second assumption for \cref{lem:connection} to hold is that \cref{eq:cond} is satisfied.
	That is, the integral of $\| B \boldsymbol{\tau}(\cdot) \|_{\ell^1}: \Omega \to \R_0^+$ converges to zero as $N \to \infty$. 
	This is a weaker condition than the maximum value of $\| B \boldsymbol{\tau}(\cdot) \|_{\ell^1}$ converging to zero, which was numerically observed to hold for PHS in \cite{bayona2019insight}. 
	The relation between these conditions can be observed by applying H\"older's inequality (see, for instance, \cite[Chapter 3]{rudin1987real}). 
	Let $1 \leq p,q \leq \infty$ with $1/p + 1/q = 1$ and assume that $\omega \in L^q(\Omega)$. 
	Then we have 
	\begin{equation} 
		\int_{\Omega} \| B \boldsymbol{\tau}(\boldsymbol{x}) \|_{\ell^1} \omega(\boldsymbol{x}) \intd \boldsymbol{x} 
			\leq \left( \int_{\Omega} \| B \boldsymbol{\tau}(\boldsymbol{x}) \|_{\ell^1}^p \intd \boldsymbol{x} \right)^{1/p} 
			\left( \int_{\Omega} \omega(\boldsymbol{x})^q \intd \boldsymbol{x} \right)^{1/q}.
	\end{equation}  
	Hence, $\| B \boldsymbol{\tau}\|_{\ell^1}$ converging to zero in $L^p(\Omega)$ as $N \to \infty$ for some $p \geq 1$ immediately implies \cref{eq:relation_weights}. 
	The special case of $p = \infty$ corresponds to \cref{eq:observation_Bayona}. 
\end{remark}
\section{Numerical Results} 
\label{sec:numerical} 

We present a variety of numerical tests in one and two dimensions to demonstrate our theoretical findings. 
In particular, a stability and error analysis for CFs based on different RBFs is presented. 
Thereby, compactly supported RBFs are discussed in \cref{sub:num_compact}, Gaussian RBFs in \cref{sub:num_Gaussian}, and PHS in \cref{sub:num_PHS}. 
For sake of simplicity, a constant weight function $\omega \equiv 1$ is used in all test cases. 
All numerical tests presented here were generated by the open-access MATLAB code \cite{glaubitz2021stabilityCode}.

\subsection{Compactly Supported RBFs} 
\label{sub:num_compact}

Let us start with a demonstration of \cref{thm:main} in one dimension. 
To this end, we consider Wendland's compactly supported RBFs in $\Omega = [0,1]$.

\begin{figure}[tb]
	\centering 
	\begin{subfigure}[b]{0.33\textwidth}
		\includegraphics[width=\textwidth]{%
      	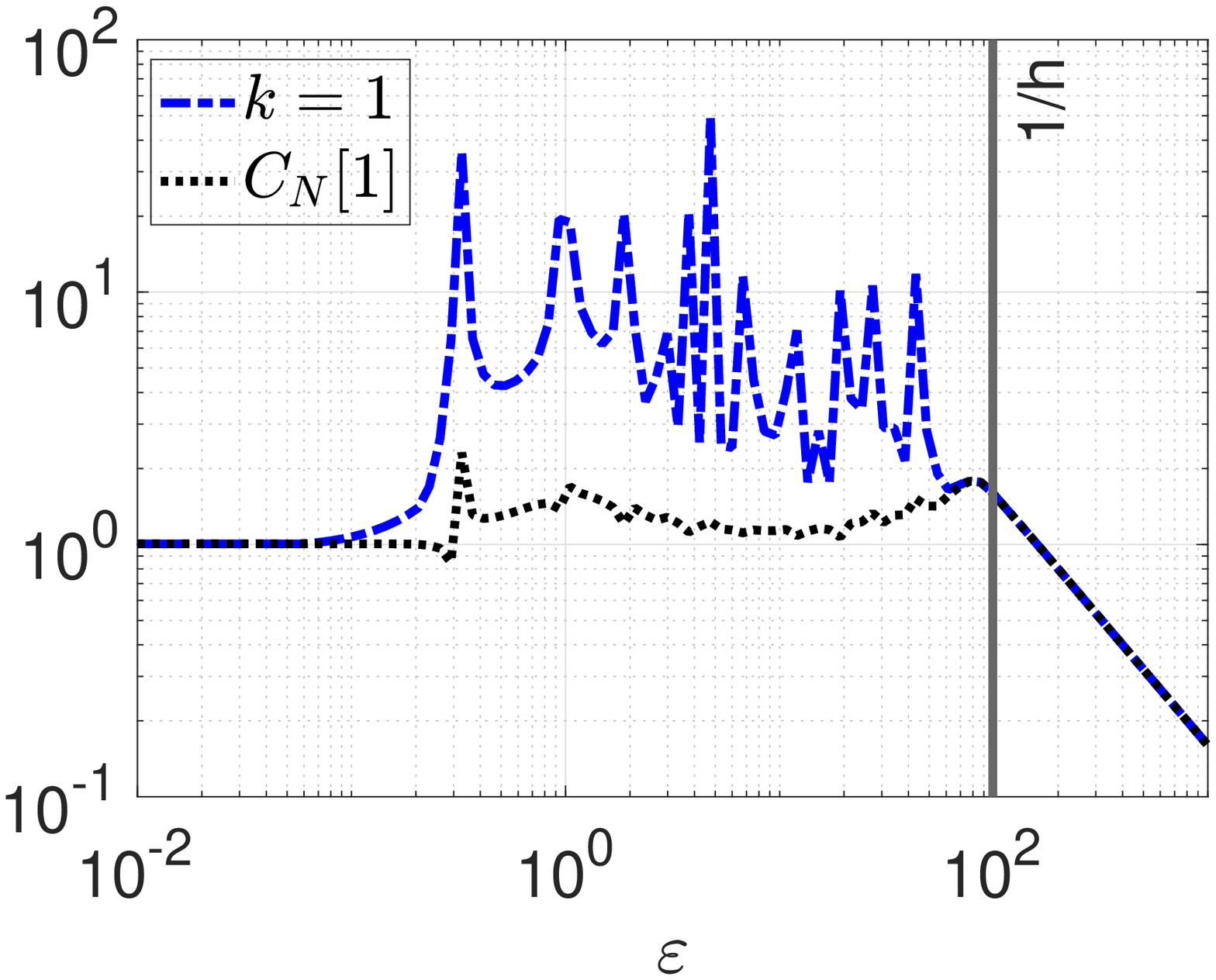} 
    \caption{$k=1$ and $d=-1$ (pure RBF)}
    \label{fig:demo_noPol}
  	\end{subfigure}%
	~
  	\begin{subfigure}[b]{0.33\textwidth}
		\includegraphics[width=\textwidth]{%
      	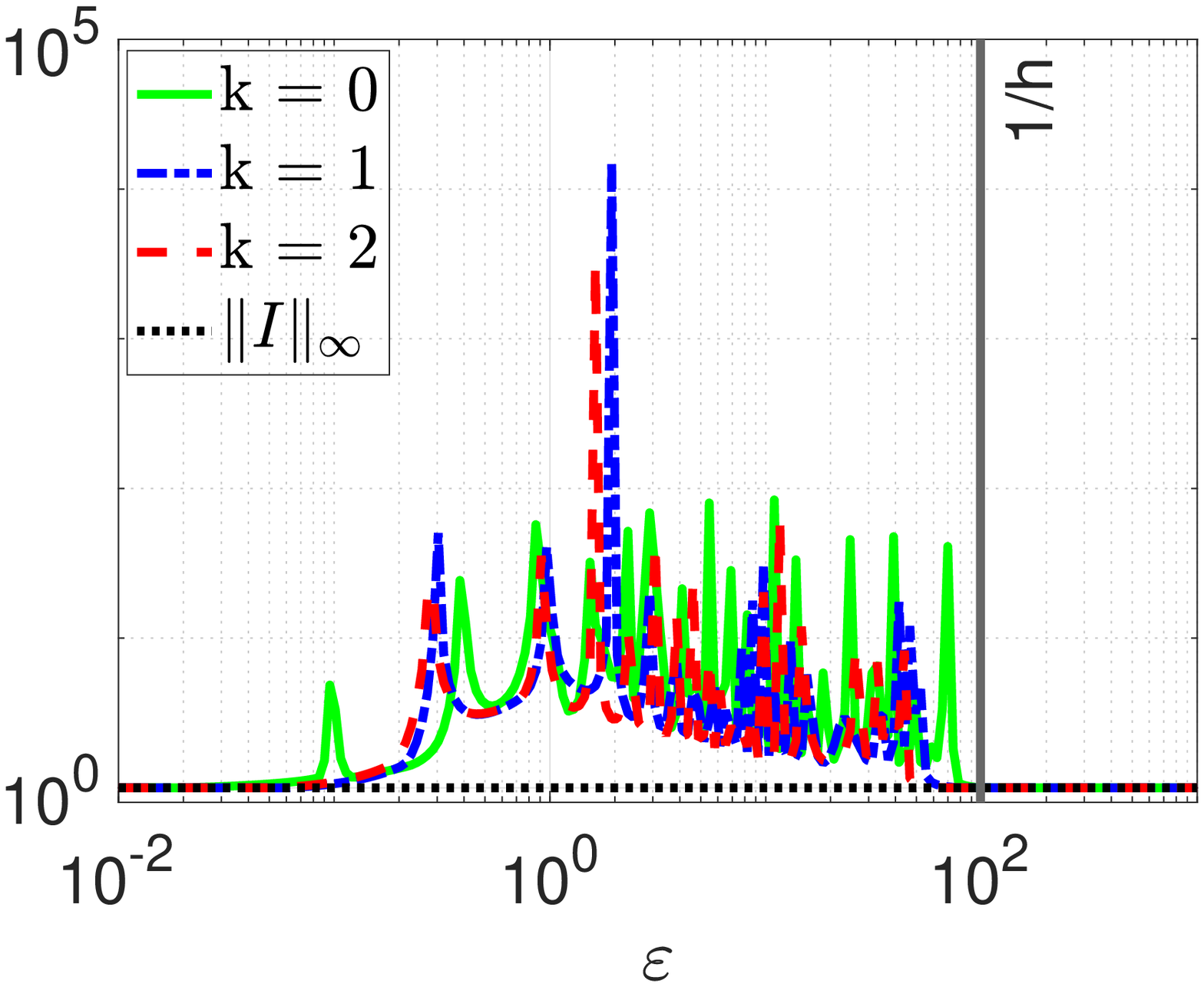} 
    \caption{$d=0$ (constant term)}
    \label{fig:demo_d0}
  	\end{subfigure}%
  	~ 
  	\begin{subfigure}[b]{0.33\textwidth}
		\includegraphics[width=\textwidth]{%
      	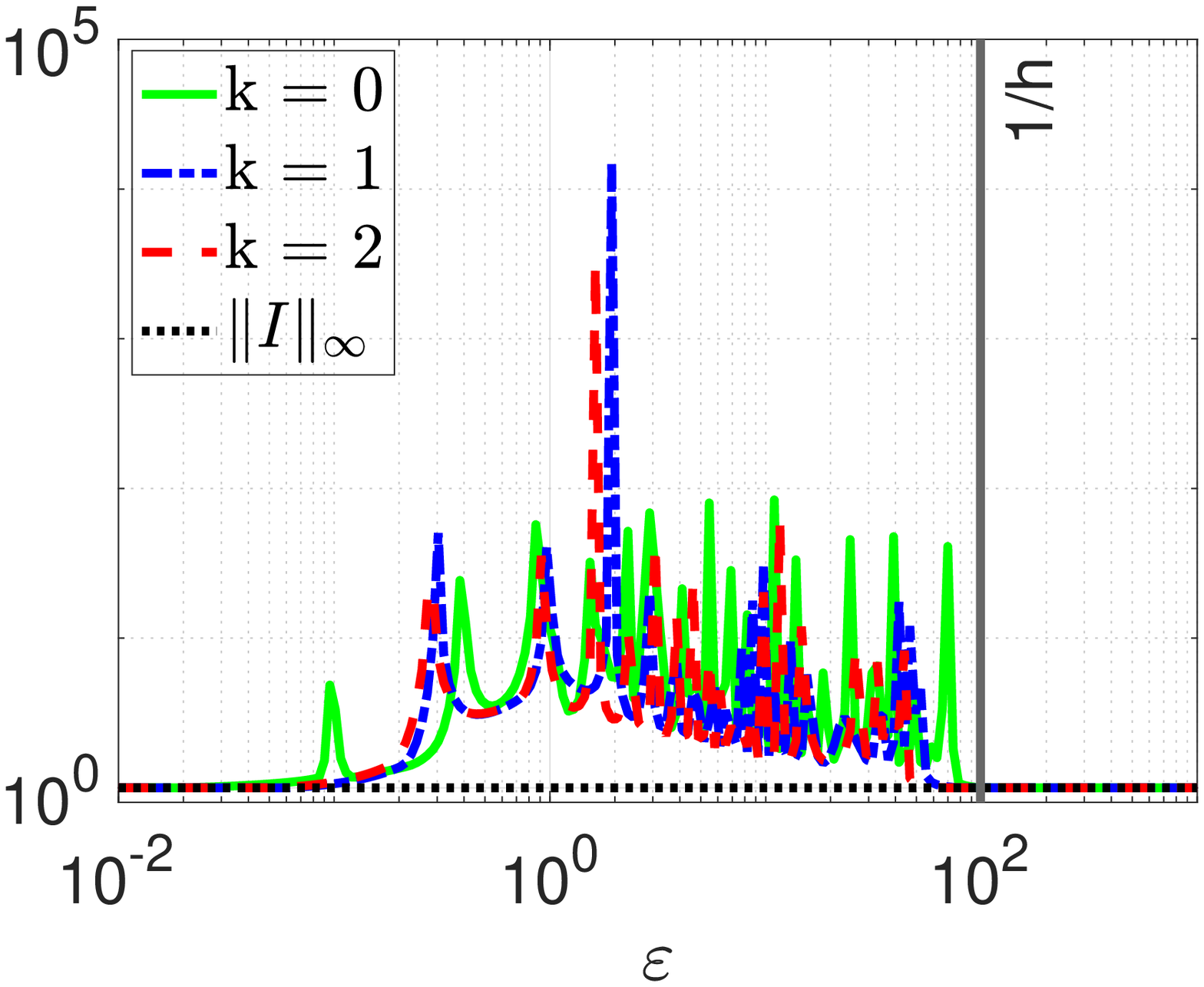} 
    \caption{$d=1$ (linear term)}
    \label{fig:demo_d1}
  	\end{subfigure}%
  	\caption{
  	The stability measure $\|C_N\|_{\infty}$ of Wendland's compactly supported RBF $\varphi_{1,k}$ with smoothness parameters $k=0,1,2$. 
	In all cases, $N=100$ equidistant data points were considered, while the reference shape parameter $\varepsilon$ was allowed to vary. 
	$1/h$ denotes the threshold above which the basis functions have nonoverlapping support. 
  	}
  	\label{fig:demo}
\end{figure} 

\cref{fig:demo} illustrates the stability measure $\|C_N\|_{\infty}$ of Wendland's compactly supported RBF $\varphi_{1,k}$ with smoothness parameters $k=0,1,2$ as well as the optimal stability measure. 
The latter is given by $C_N[1]$ if no constants are included and by $\|I\|_\infty = 1$ if constants are included in the RBF approximations space, meaning that the RBF-CF is exact for constants. 
Furthermore, $N=100$ equidistant data points in $\Omega = [0,1]$ were used, including the end points, $x_1 = 0$ and $x_N = 1$, and the (reference) shape parameter $\varepsilon$ was allowed to vary. 
Finally, $1/h$ denotes the threshold above which the compactly supported RBFs are all having nonoverlapping support. 

We start by noting the RBF-CFs are observed to be stable for sufficiently small shape parameters. 
This can be explained by all the basis functions, $\varphi_n$, converging to a constant function for $\varepsilon \to 0$. 
At the same time, we can also observe the RBF-CF to be stable for $\varepsilon \geq 1/h$. 
It can be argued that this is in accordance with \cref{thm:main}. 
Recall that \cref{thm:main} essentially states that for $\varepsilon \geq 1/h$, and assuming that all basis functions have equal moments ($I[\varphi_n] = I[\varphi_m]$ for all $n,m$), the corresponding RBF-CF (including polynomials of any degree) is stable if a sufficiently large number of equidistribiuted data points is used. 
Here, the equal moments condition was ensured by choosing the shape parameter as $\varepsilon_n = \varepsilon$ for the interior data points ($n=2,\dots,N-1$) and as $\varepsilon_1 = \varepsilon_N = \varepsilon/2$ for the boundary data points. 

\begin{figure}[tb]
	\centering 
  	\begin{subfigure}[b]{0.45\textwidth}
		\includegraphics[width=\textwidth]{%
      	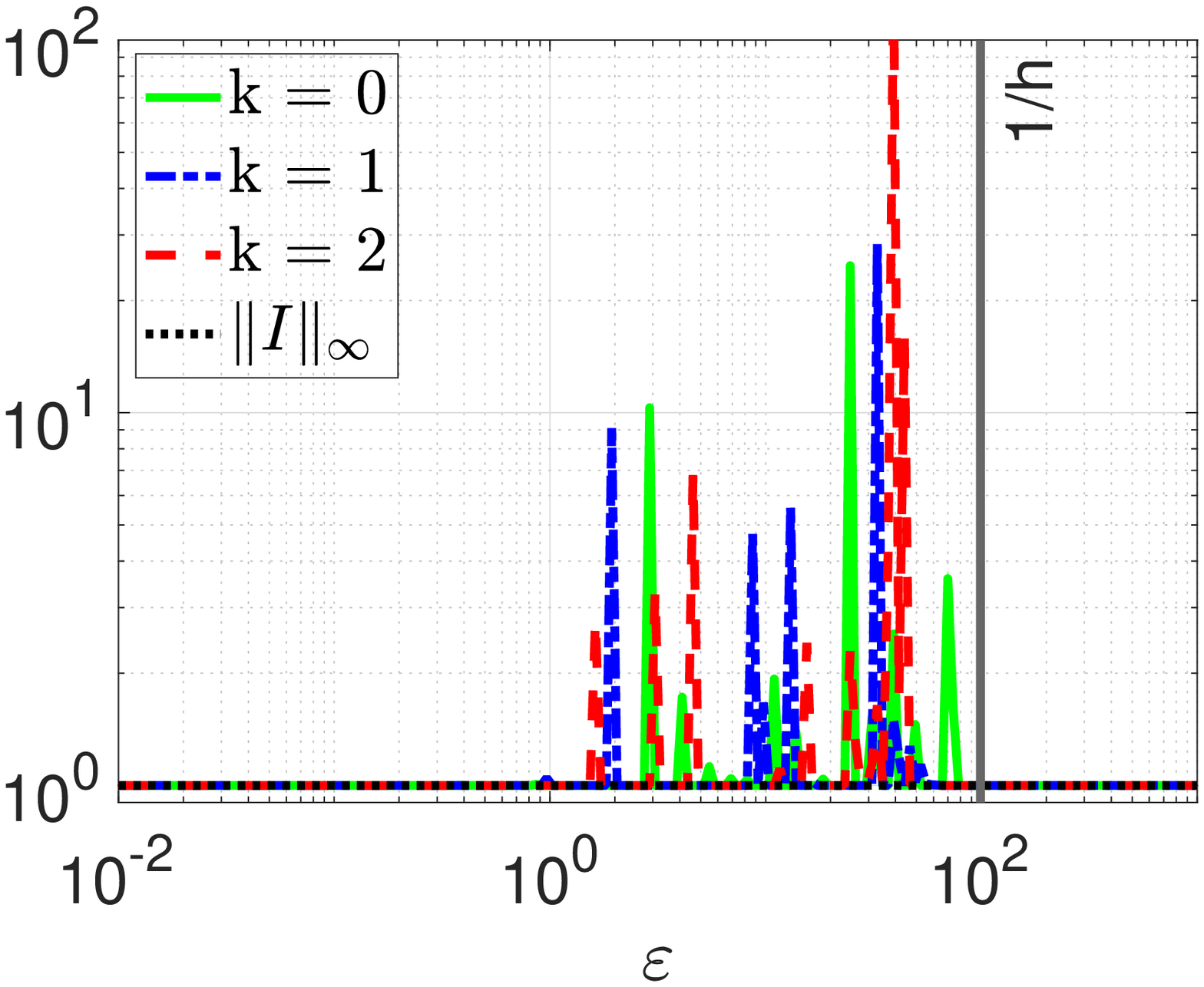} 
    \caption{$d=0$ (constant term)}
    \label{fig:momentCond_d0}
  	\end{subfigure}%
  	~ 
  	\begin{subfigure}[b]{0.45\textwidth}
		\includegraphics[width=\textwidth]{%
      	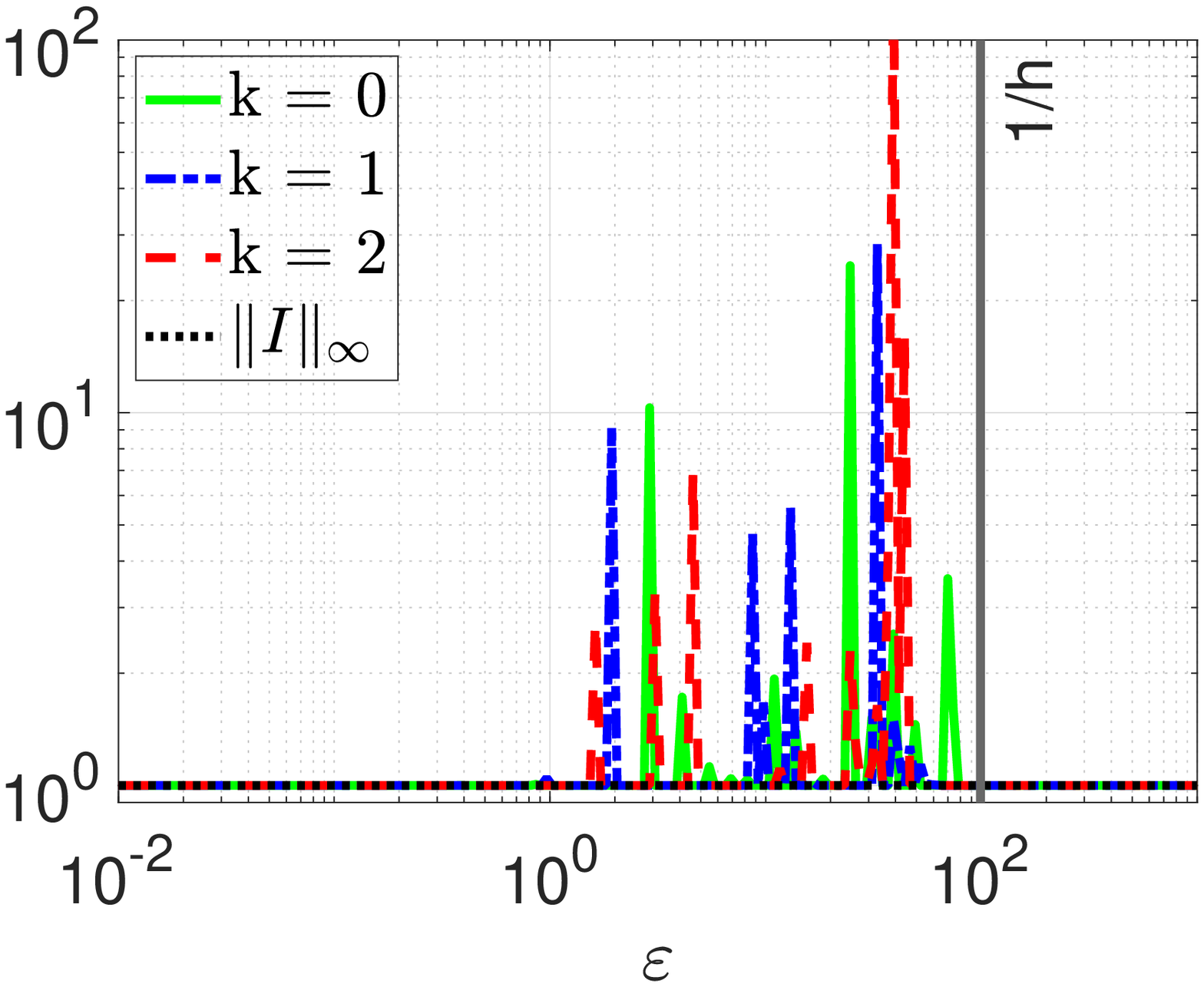} 
    \caption{$d=1$ (linear term)}
    \label{fig:momentCond_d1}
  	\end{subfigure}%
  	\caption{
  	The stability measure $\|C_N\|_{\infty}$ of Wendland's compactly supported RBF $\varphi_{1,k}$ with smoothness parameters $k=0,1,2$. 
	In all cases, $N=100$ equidistant data points were considered. 
	The same shape parameter $\varepsilon$ was used for all basis functions, yielding (at least) the moments corresponding to the boundary data points $x_1 = 0$ and $x_N=1$ to differ from the others. 
	$1/h$ denotes the threshold above which the basis functions have nonoverlapping support.
  	}
  	\label{fig:momentCond}
\end{figure}

\begin{figure}[tb]
	\centering 
  	\begin{subfigure}[b]{0.45\textwidth}
		\includegraphics[width=\textwidth]{%
      	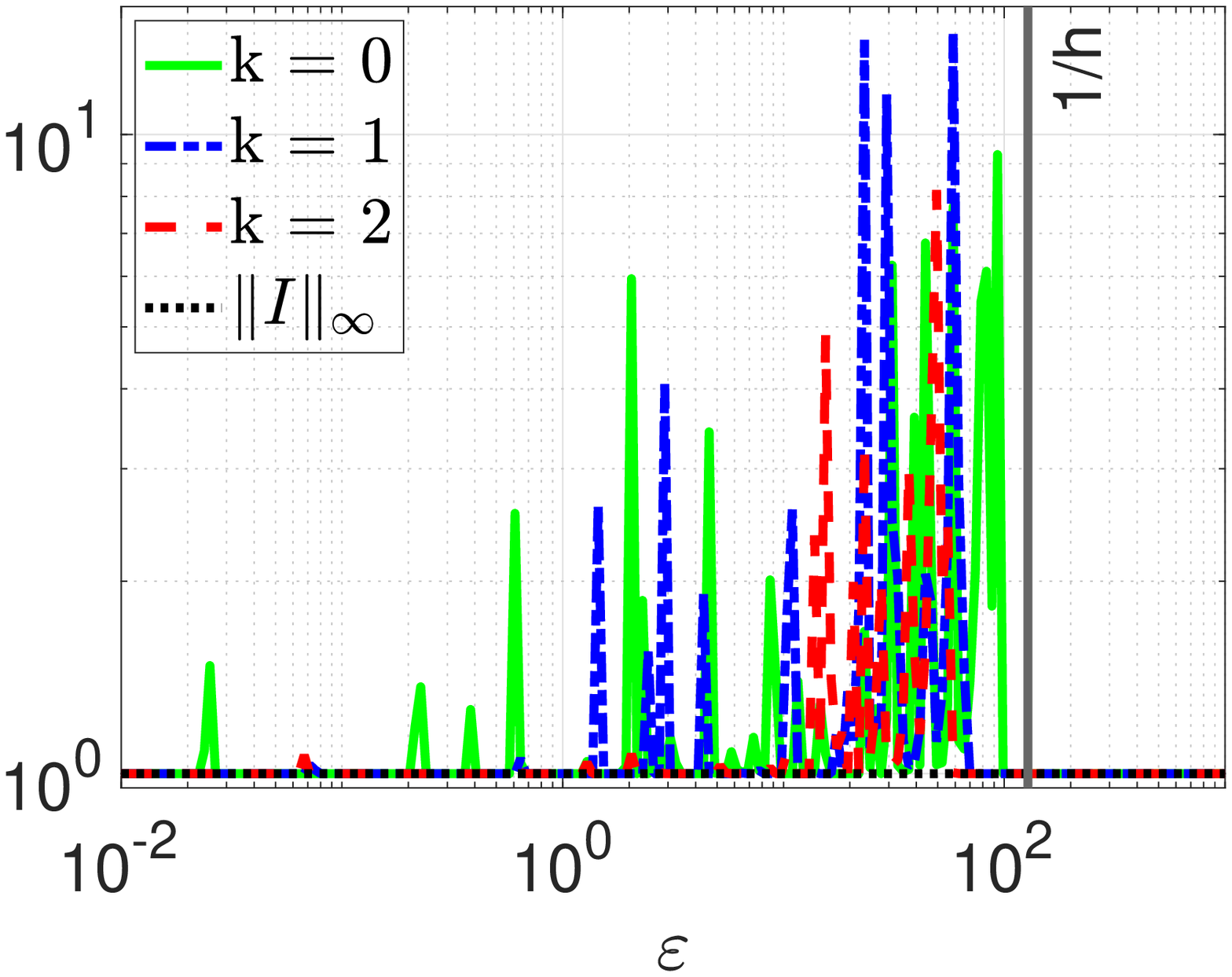} 
    \caption{$d=0$ (constant term)}
    \label{fig:nonequid_d0}
  	\end{subfigure}%
  	~ 
  	\begin{subfigure}[b]{0.45\textwidth}
		\includegraphics[width=\textwidth]{%
      	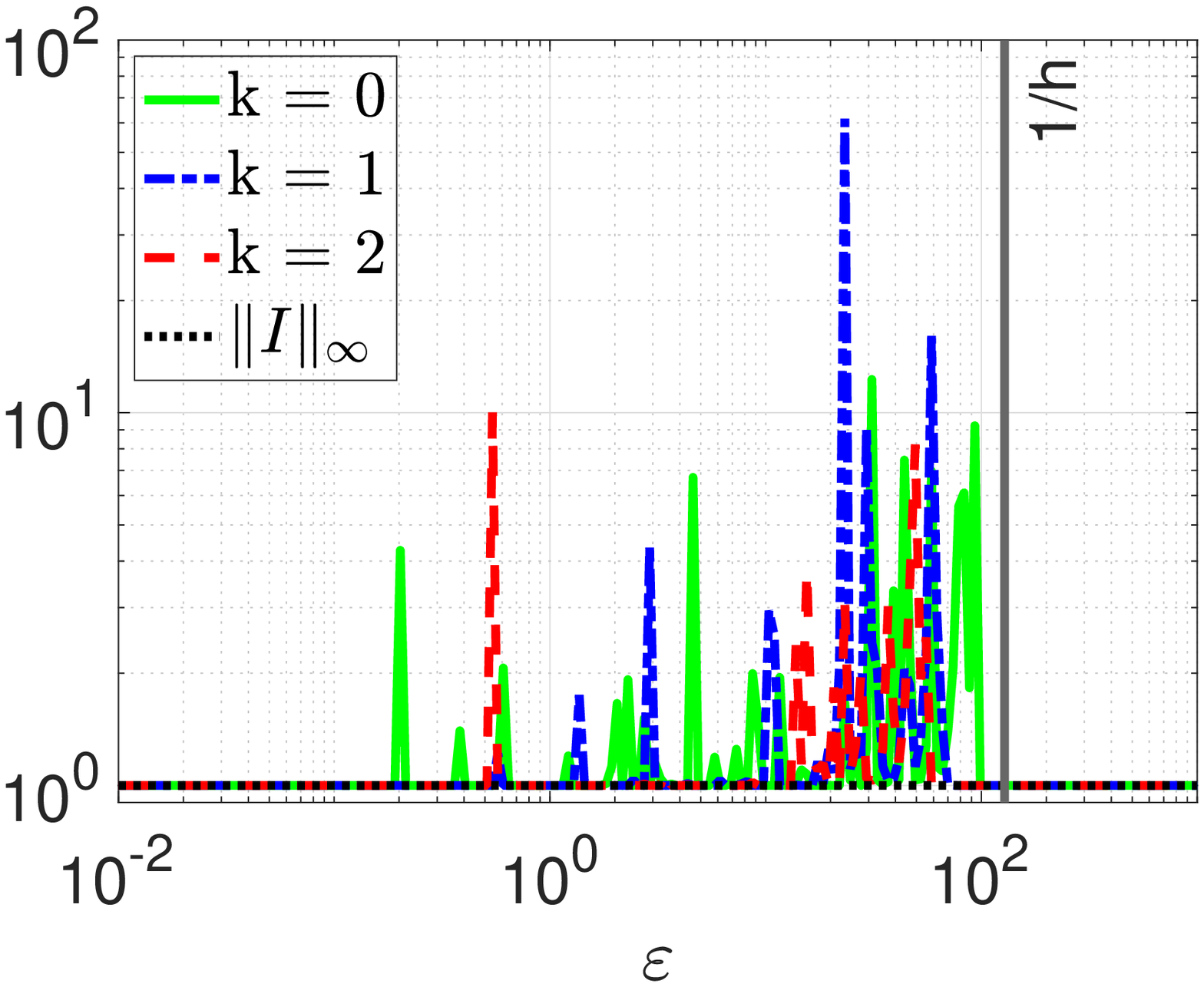} 
    \caption{$d=1$ (linear term)}
    \label{fig:nonequid_d1}
  	\end{subfigure}%
  	\caption{
  	The stability measure $\|C_N\|_{\infty}$ of Wendland's compactly supported RBF $\varphi_{1,k}$ with smoothness parameters $k=0,1,2$. 
	In all cases, $N=100$ Halton points and a constant shape parameter $\varepsilon$ were considered. 
	$1/h$ denotes the threshold above which the basis functions have nonoverlapping support.
  	}
  	\label{fig:nonequid}
\end{figure}

That said, at least numerically, we observe that it is possible to drop this equal moment condition. 
This is demonstrated by \cref{fig:momentCond}, where we perform the same test as in \cref{fig:demo} except choosing all the shape parameters to be equal ($\varepsilon_n = \varepsilon$, $n=1,\dots,N$). 
This results in the two basis functions corresponding to the boundary points $x_1 = 0$ and $x_N = 1$ having smaller moments than the basis functions corresponding to interior data points for all $\varepsilon$. 
Nevertheless, we can see in \cref{fig:momentCond} that for $\varepsilon \geq 1/h$ the RBF-CFs are still stable. 
Moreover, the same observation is also made in \cref{fig:nonequid} for the same test using Halton points.
Once more, we find the corresponding RBF-CFs to be stable for $\varepsilon \geq 1/h$ as well as for sufficiently small shape parameter $\varepsilon$. 

\begin{figure}[tb]
	\centering 
  	\begin{subfigure}[b]{0.33\textwidth}
		\includegraphics[width=\textwidth]{%
      	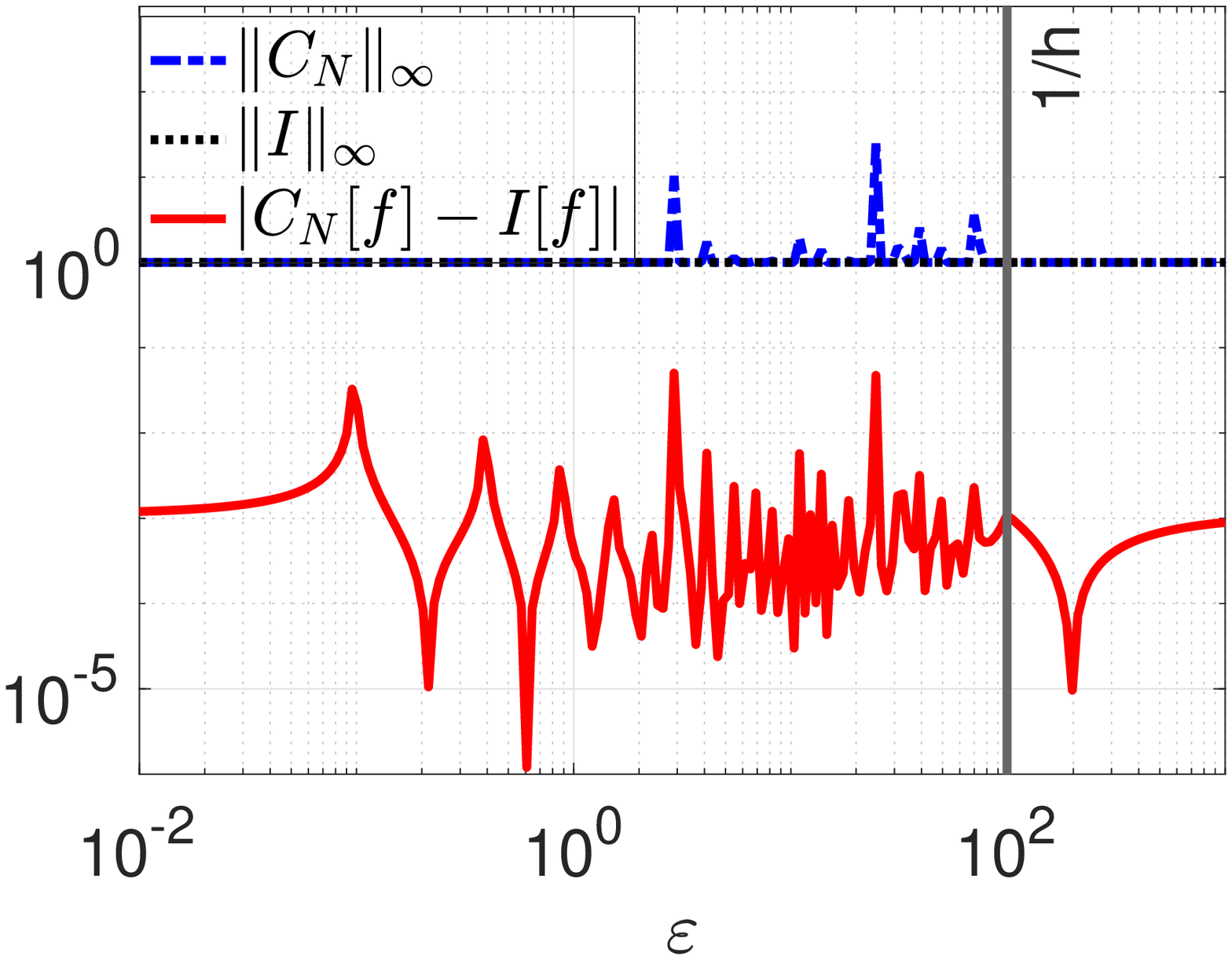} 
    \caption{$k=0$ and $d=0$ (constant term)}
    \label{fig:error_1d_k0_d0}
  	\end{subfigure}%
  	~ 
	\begin{subfigure}[b]{0.33\textwidth}
		\includegraphics[width=\textwidth]{%
      	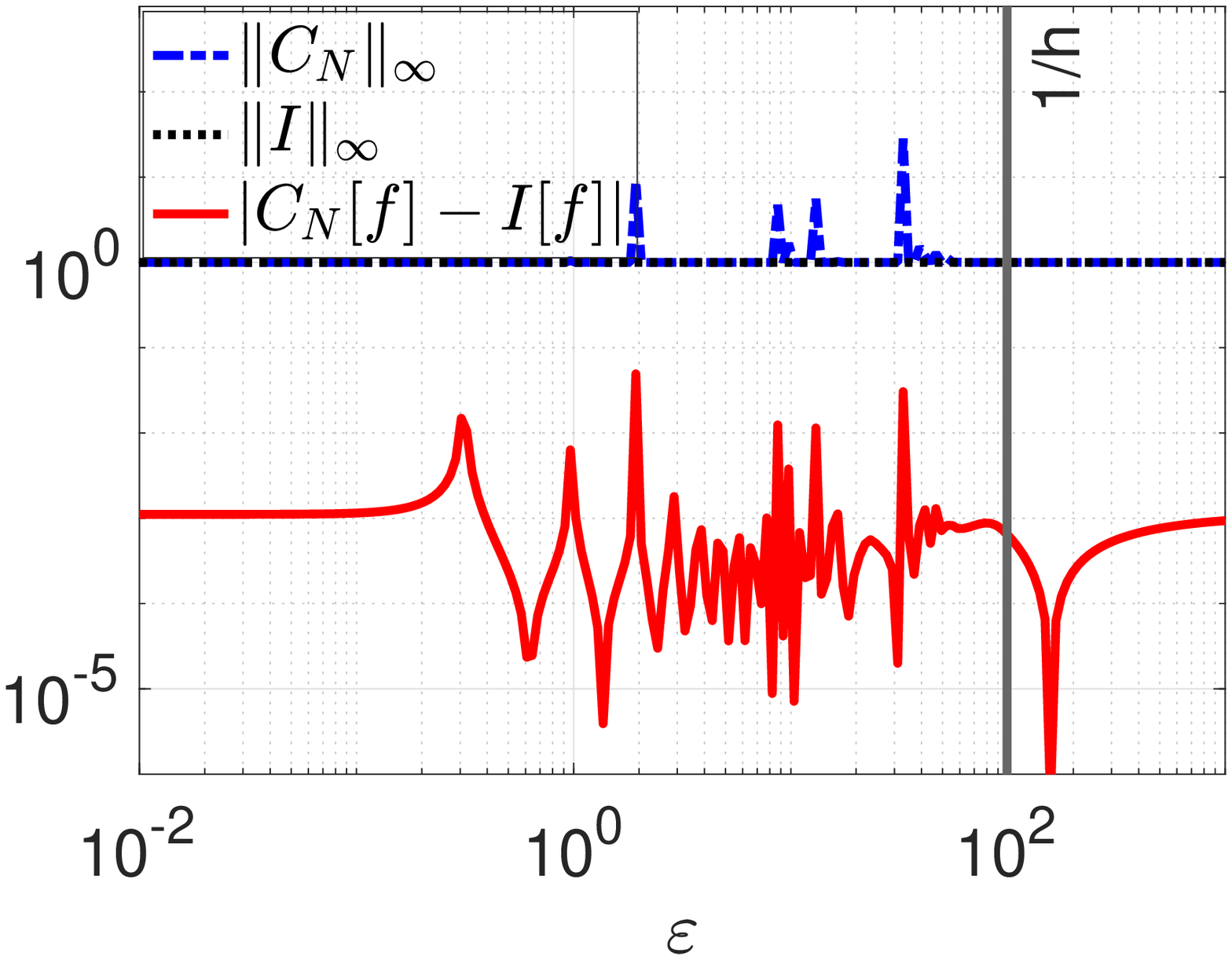} 
    \caption{$k=1$ and $d=0$ (constant term)}
    \label{fig:error_1d_k1_d0}
  	\end{subfigure}%
	~
	\begin{subfigure}[b]{0.33\textwidth}
		\includegraphics[width=\textwidth]{%
      	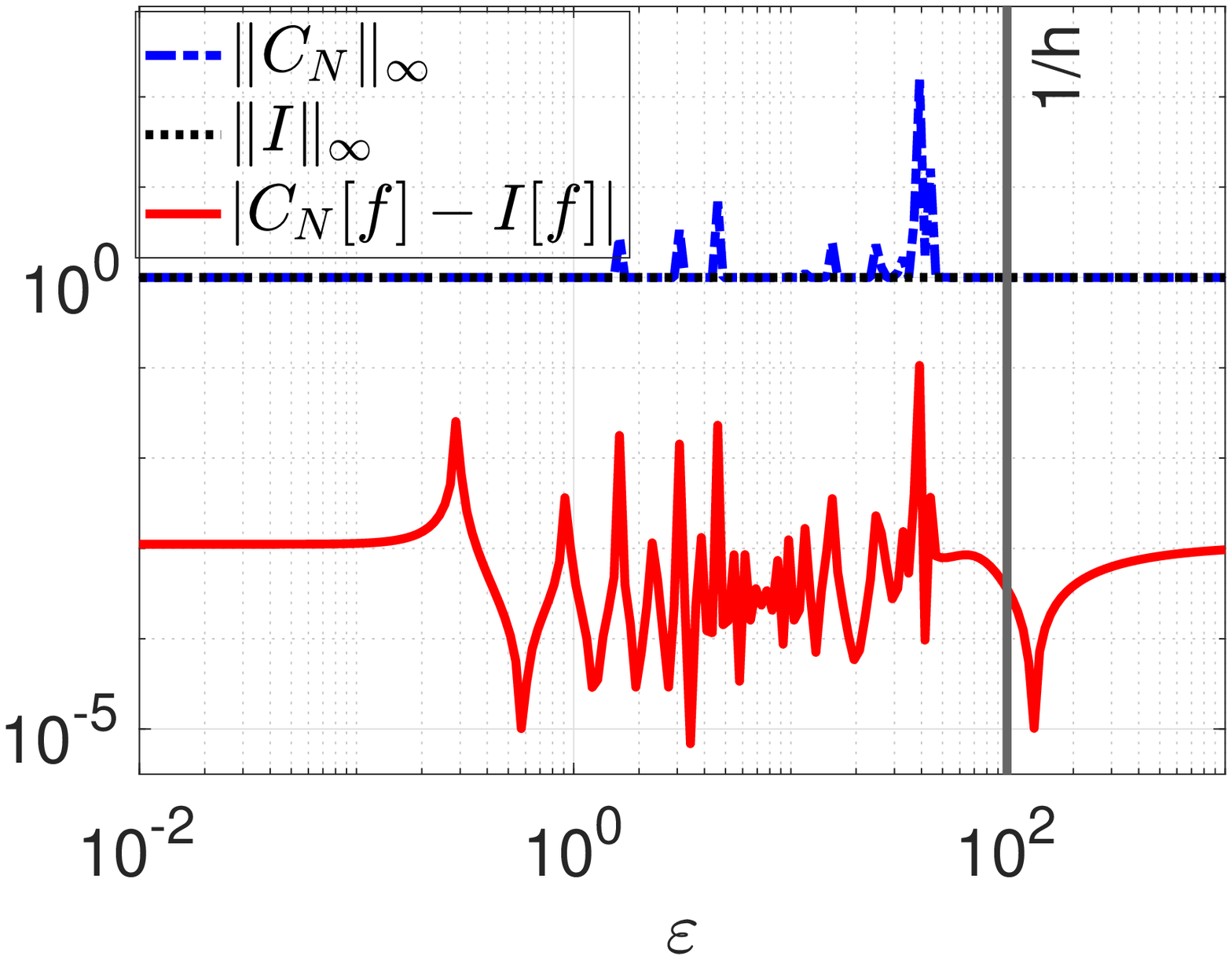} 
    \caption{$k=2$ and $d=0$ (constant term)}
    \label{fig:error_1d_k2_d0}
  	\end{subfigure}%
	\\
  	\begin{subfigure}[b]{0.33\textwidth}
		\includegraphics[width=\textwidth]{%
      	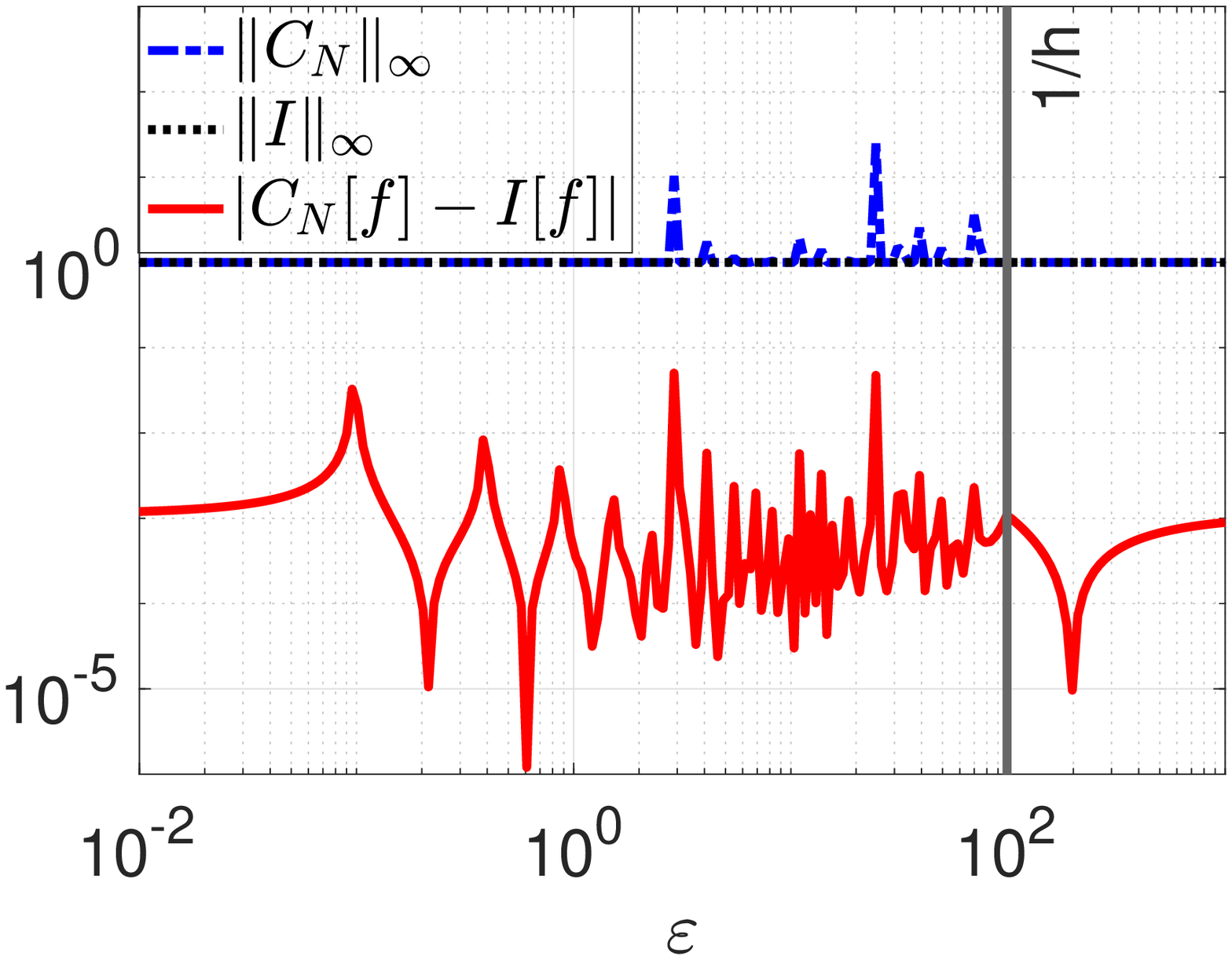} 
    \caption{$k=0$ and $d=1$ (linear term)}
    \label{fig:error_1d_k0_d1}
  	\end{subfigure}%
	~
	\begin{subfigure}[b]{0.33\textwidth}
		\includegraphics[width=\textwidth]{%
      	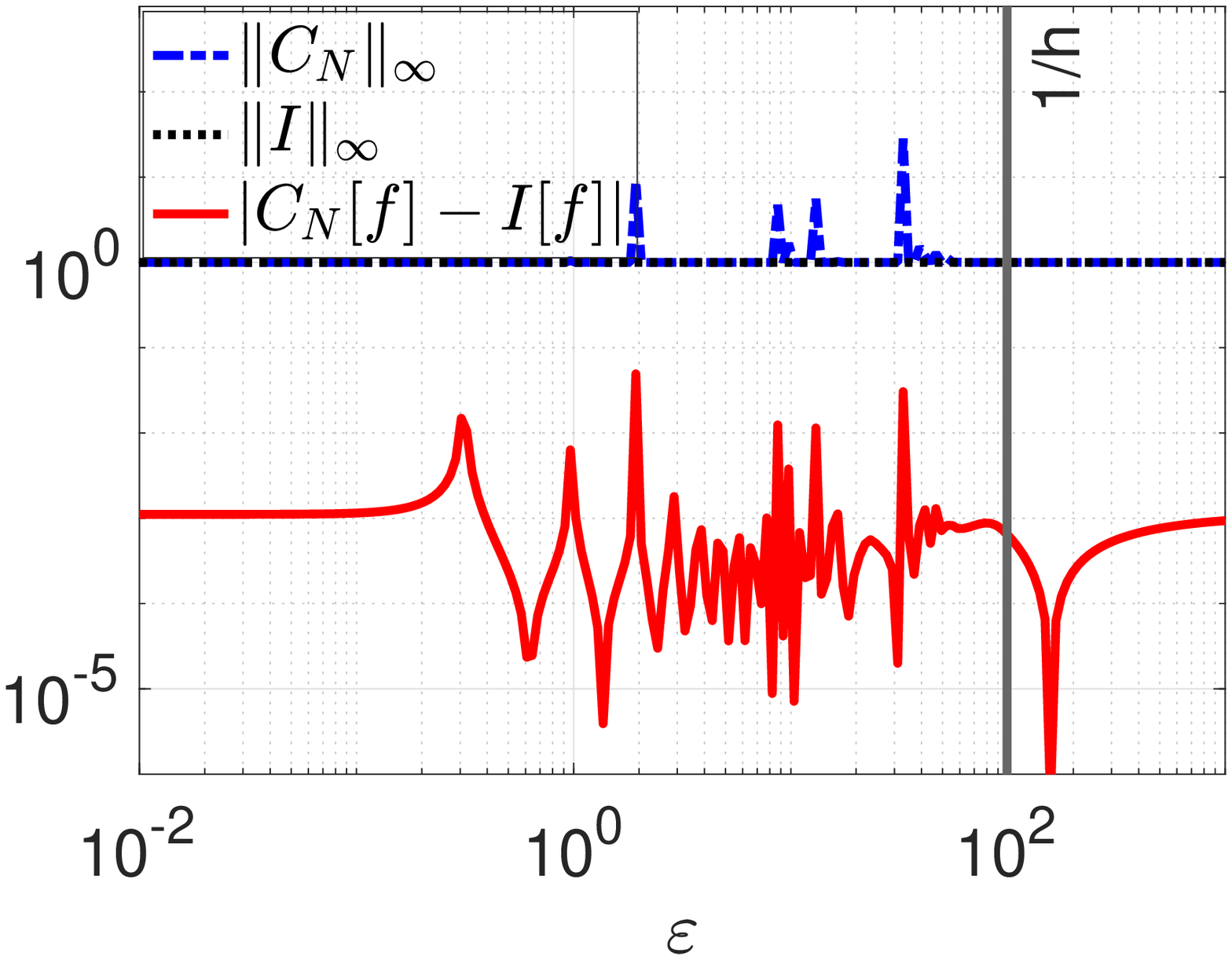} 
    \caption{$k=1$ and $d=1$ (linear term)}
    \label{fig:error_1d_k1_d1}
  	\end{subfigure}%
	~ 
	\begin{subfigure}[b]{0.33\textwidth}
		\includegraphics[width=\textwidth]{%
      	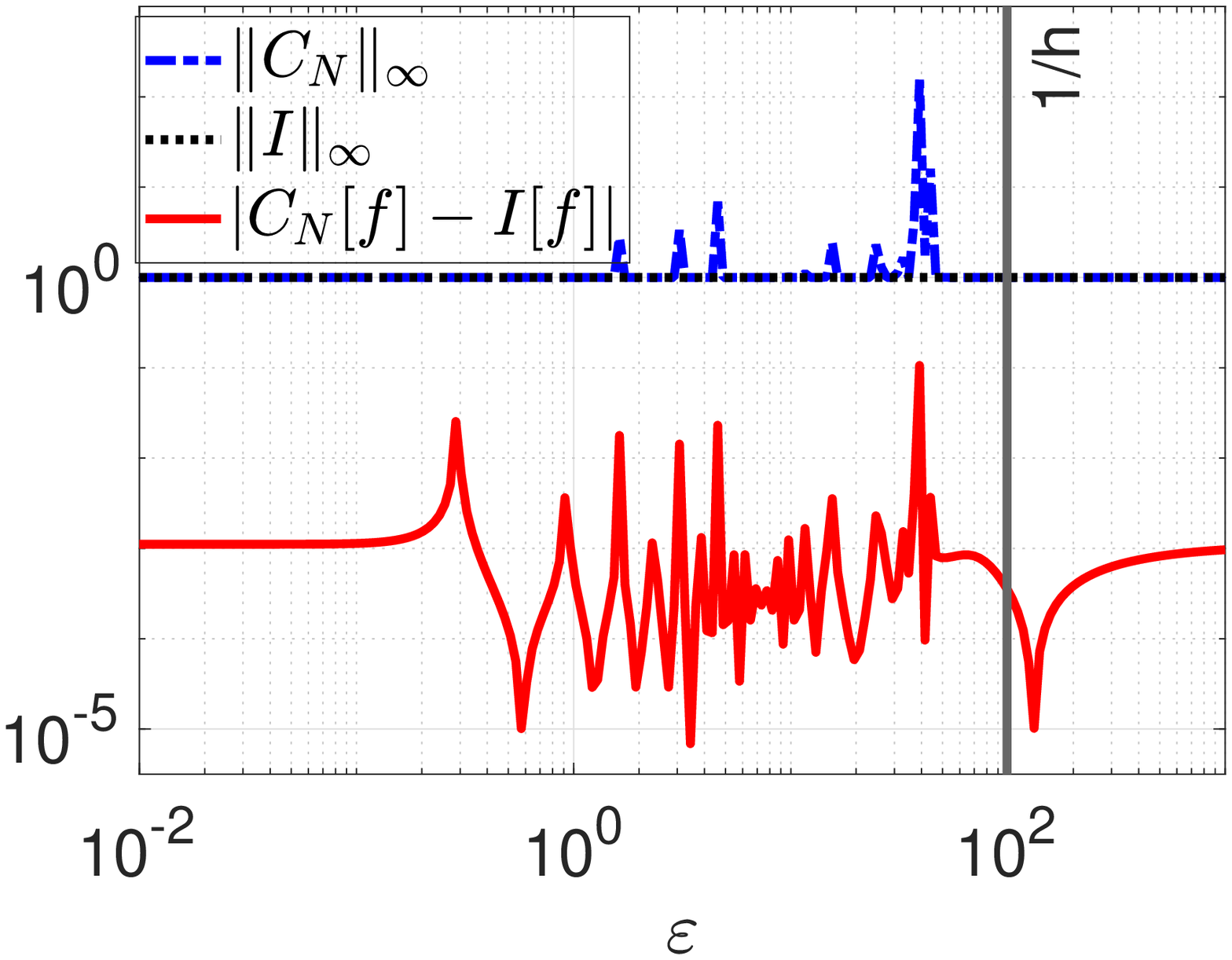} 
    \caption{$k=2$ and $d=1$ (linear term)}
    \label{fig:error_1d_k2_d1}
  	\end{subfigure}%
  	\caption{
	Error analysis for the one-dimensional test function $f(x) = c/(1 + (x-0.25)^2 )$ on $\Omega = [0,1]$, where $c$ is chosen such that $I[f] = 1$.  
	Illustrated are the error $| I[f] - C_N[f] |$ and the stability measure $\|C_N\|_{\infty}$ of Wendland's compactly supported RBF $\varphi_{1,k}$ with smoothness parameters $k=0,1,2$.
	In all cases, $N=100$ equidistant data points were considered, while the reference shape parameter $\varepsilon$ was allowed to vary. 
	$1/h$ denotes the threshold above which the basis functions have nonoverlapping support.
  	}
  	\label{fig:error_1d}
\end{figure} 

To also provide an error analysis, \cref{fig:error_1d} compares the stability measure $\|C_N\|_{\infty}$ with the error of the RBF-CF for the Runge-like test function $f(x) = c/(1 + (x-0.25)^2 )$ on $\Omega = [0,1]$, where $c$ is chosen such that $I[f] = 1$. 
Once more, we considered Wendland's compactly supported RBF $\varphi_{1,k}$ with smoothness parameters $k=0,1,2$, $N=100$ equidistant data points, and a varying shape parameter $\varepsilon$, which is the same for all basis functions. 
There are a few observations that can be made based on the results reported in \cref{fig:error_1d}. 
Arguably most importantly, the smallest error seems to be obtained for a shape parameter that yields the RBF-CF to be stable ($\|C_N\|_{\infty} = \|I\|_{\infty}$).

\begin{figure}[tb]
	\centering 
  	\begin{subfigure}[b]{0.33\textwidth}
		\includegraphics[width=\textwidth]{%
      	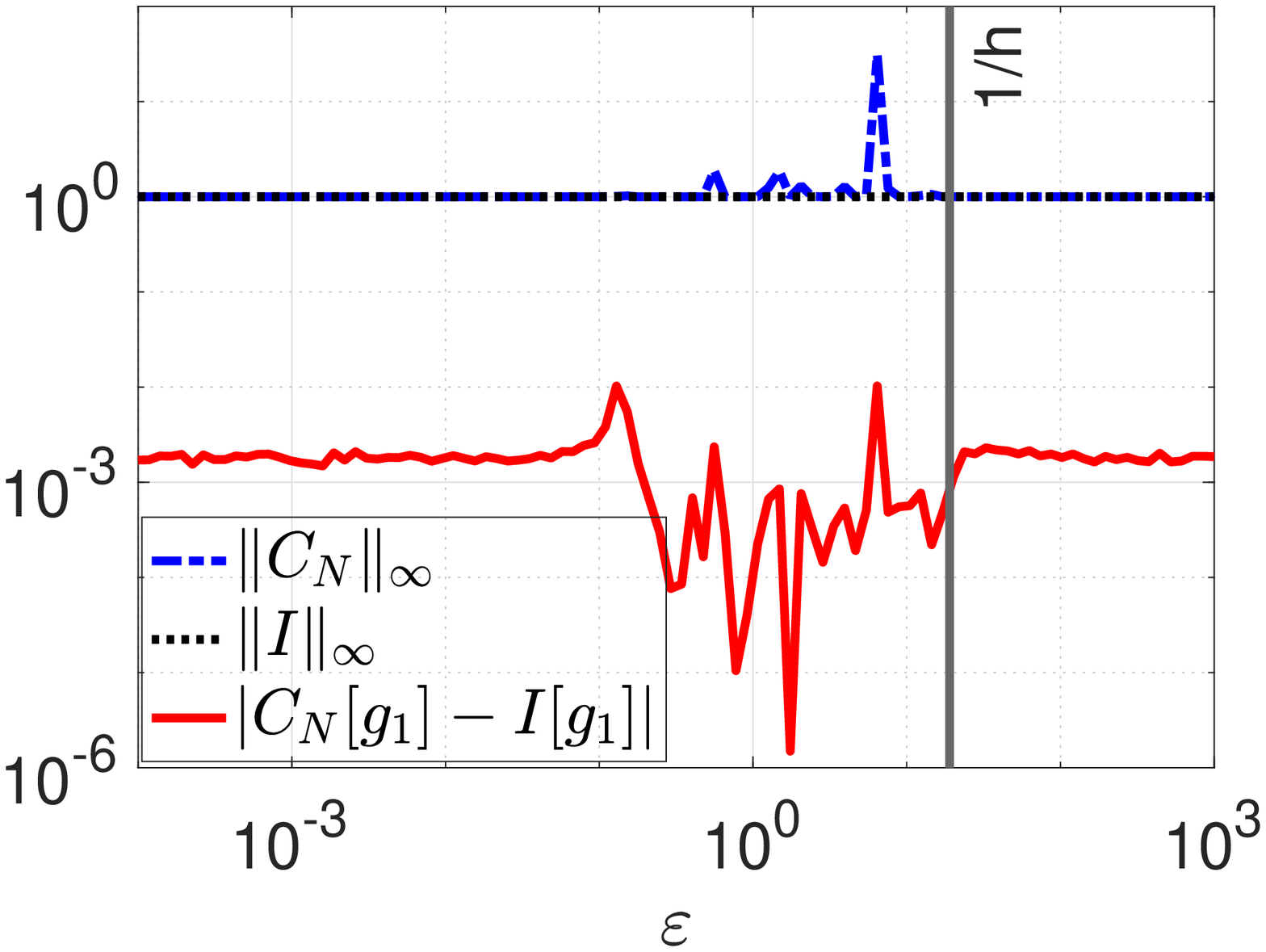} 
    \caption{Equidistant, $d=0$}
    \label{fig:error_2d_Genz1_N400_equid_k1_d0}
  	\end{subfigure}%
	~
	\begin{subfigure}[b]{0.33\textwidth}
		\includegraphics[width=\textwidth]{%
      	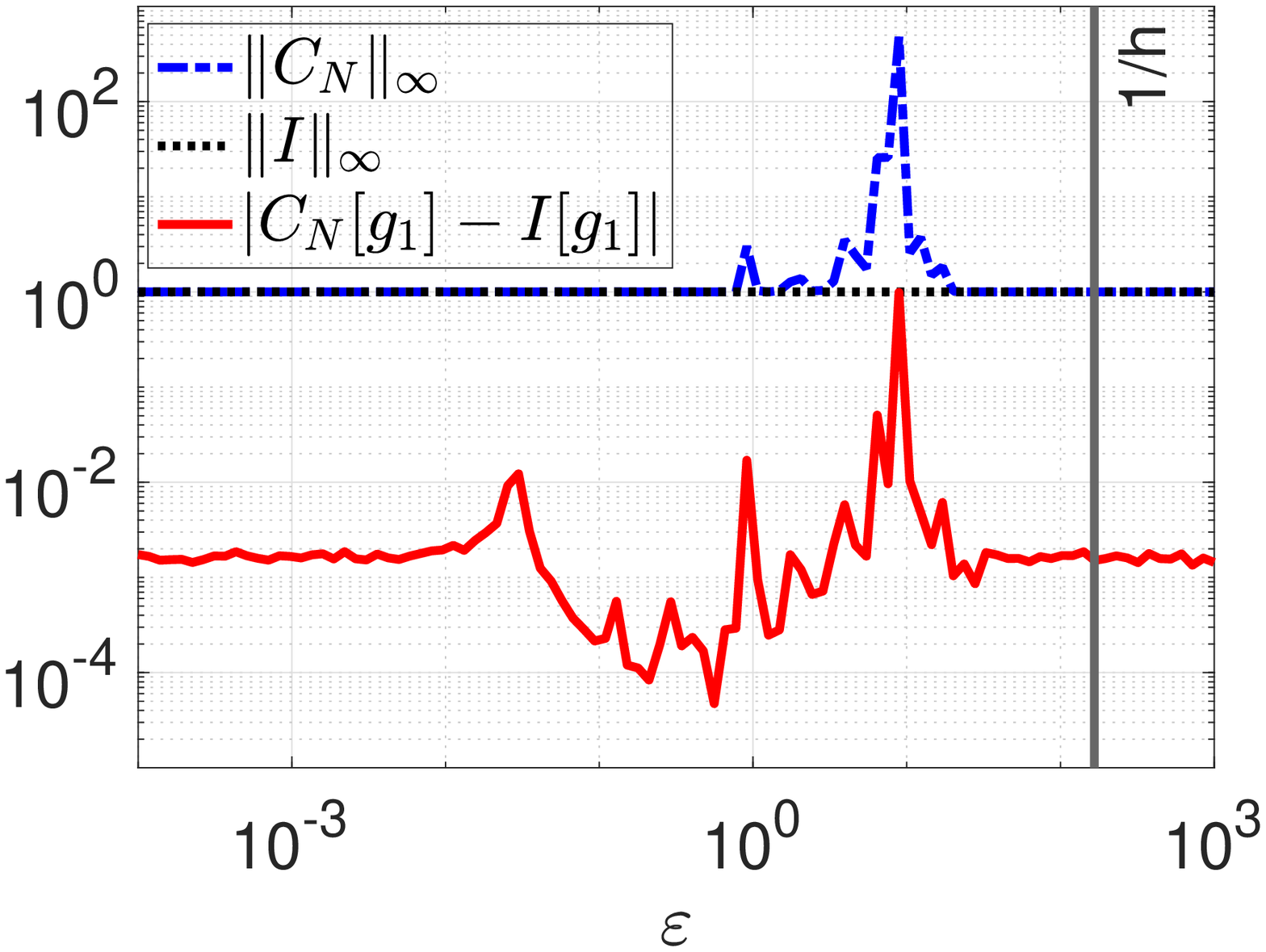} 
    \caption{Halton, $d=0$}
    \label{fig:error_2d_Genz1_N400_Halton_k1_d0}
  	\end{subfigure}%
	~ 
	\begin{subfigure}[b]{0.33\textwidth}
		\includegraphics[width=\textwidth]{%
      	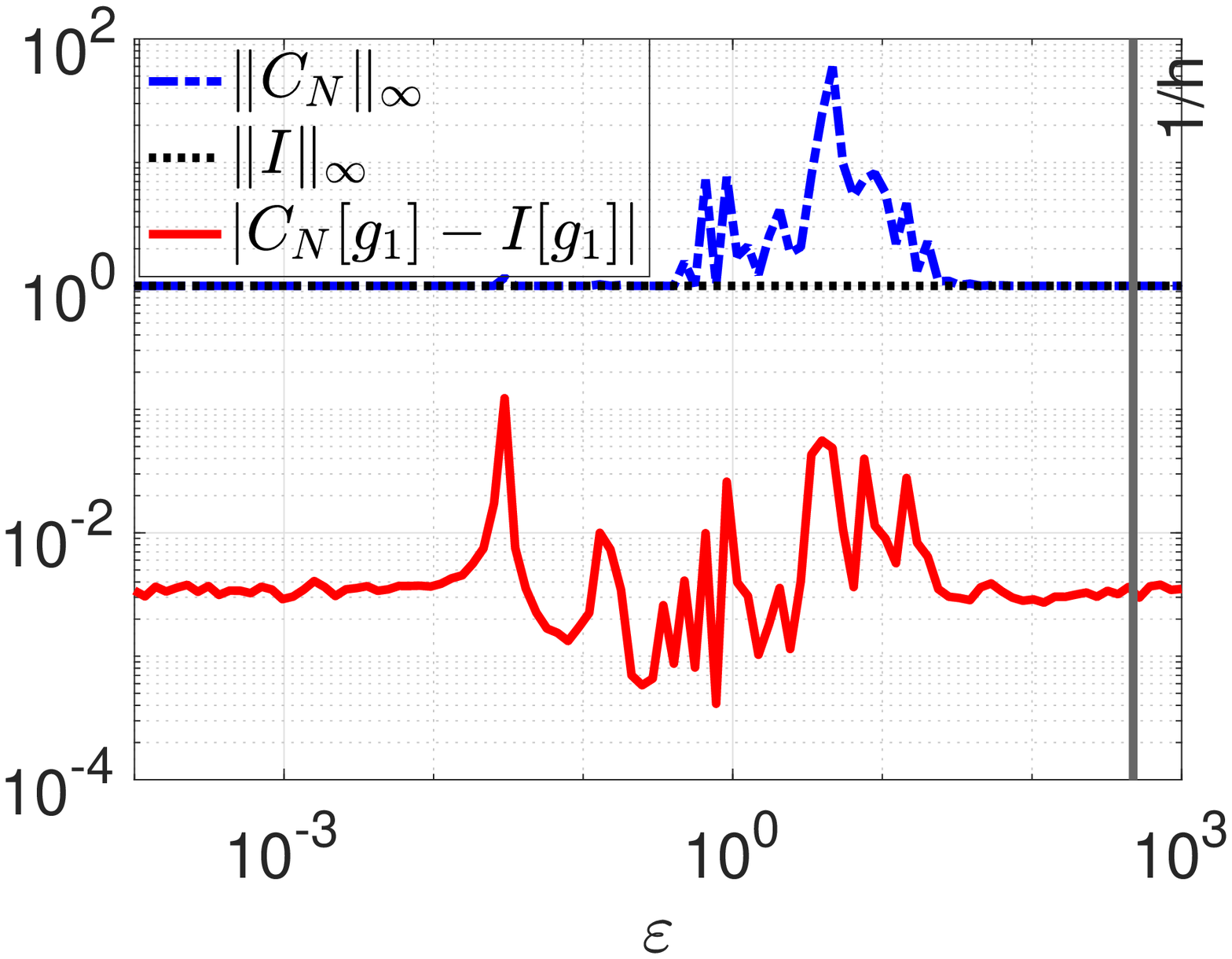} 
    \caption{Random, $d=0$}
    \label{fig:error_2d_Genz1_N400_random_k1_d0}
  	\end{subfigure}%
	\\
	\begin{subfigure}[b]{0.33\textwidth}
		\includegraphics[width=\textwidth]{%
      	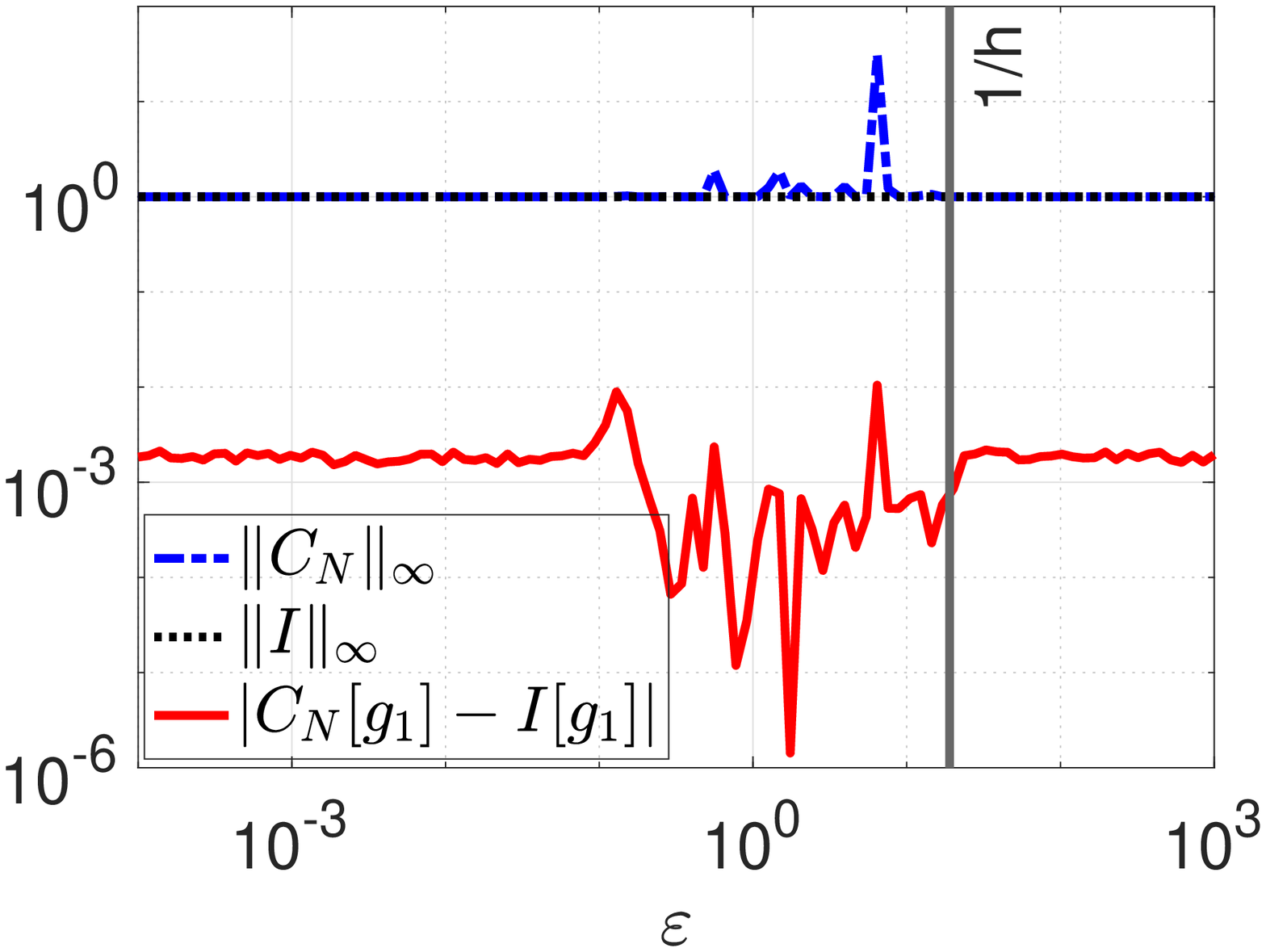} 
    \caption{Equidistant, $d=1$}
    \label{fig:error_2d_Genz1_N400_equid_k1_d1}
  	\end{subfigure}%
	~
	\begin{subfigure}[b]{0.33\textwidth}
		\includegraphics[width=\textwidth]{%
      	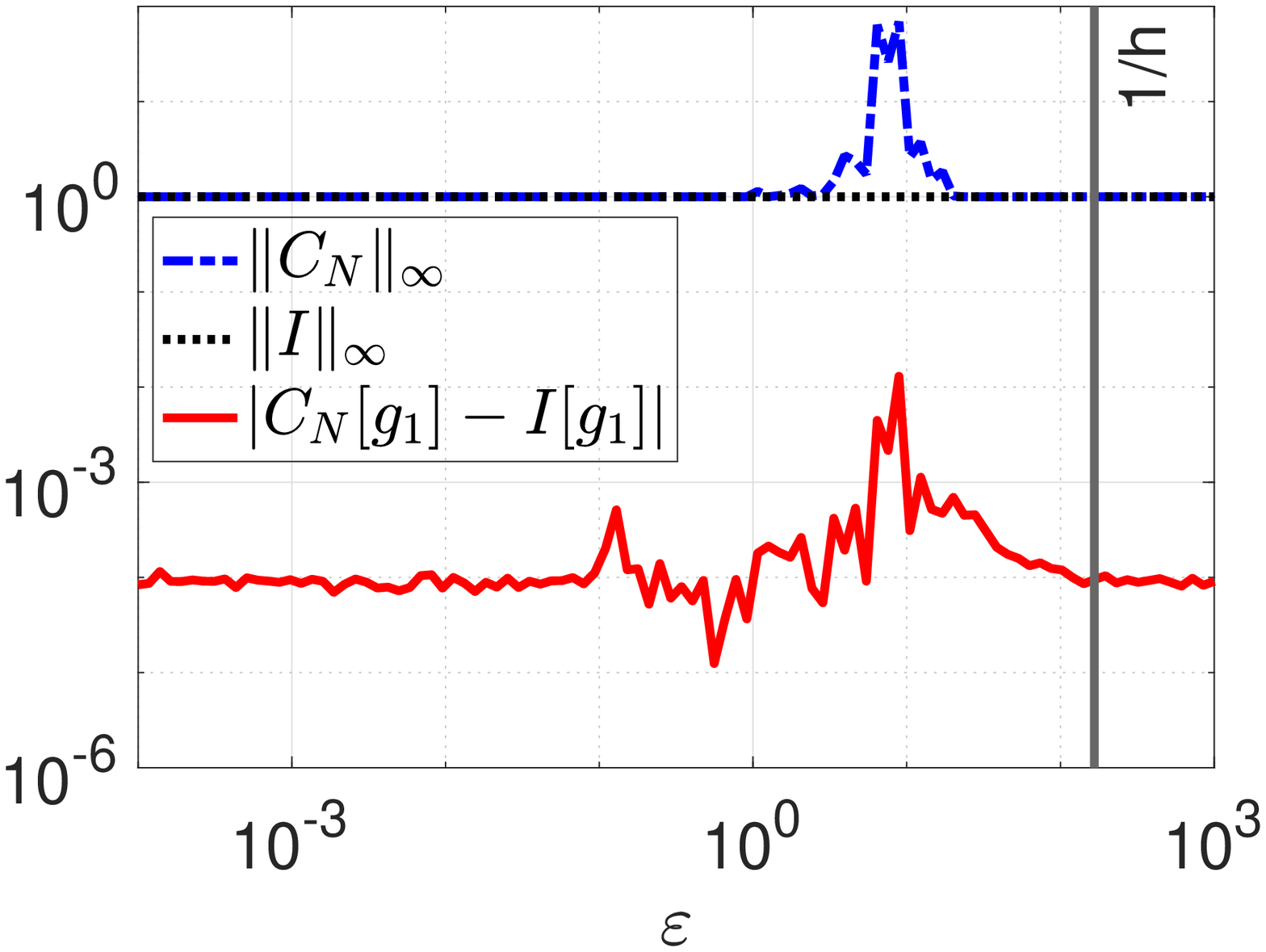} 
    \caption{Halton, $d=1$}
    \label{fig:error_2d_Genz1_N400_Halton_k1_d1}
  	\end{subfigure}%
	~ 
	\begin{subfigure}[b]{0.33\textwidth}
		\includegraphics[width=\textwidth]{%
      	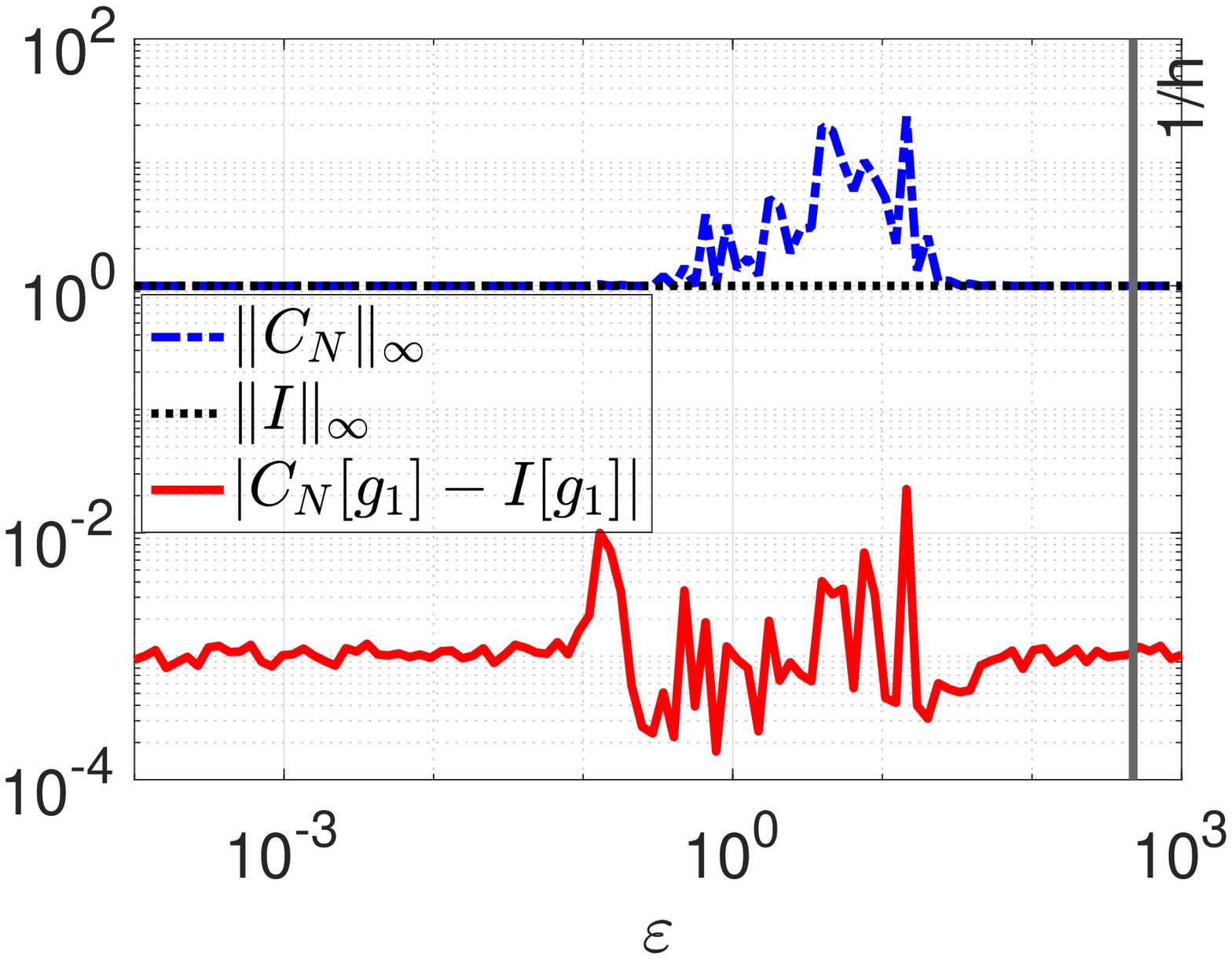} 
    \caption{Random, $d=1$}
    \label{fig:error_2d_Genz1_N400_random_k1_d1}
  	\end{subfigure}%
  	\caption{
	Error analysis for Wendland's compactly supported RBF $\varphi_{2,k}$ in two dimensions with smoothness parameter $k=1$.
	Considered is the first Genz test function $g_1$ on $\Omega = [0,1]^2$; see \cref{eq:Genz}.  
	In all cases, $N=400$ data points (equidistant, Halton, or random) were considered, while the reference shape parameter $\varepsilon$ was allowed to vary. 
	$1/h$ denotes the threshold above which the basis functions have nonoverlapping support.
  	}
  	\label{fig:error_2d_Genz1_k1}
\end{figure}

\begin{figure}[tb]
	\centering 
  	\begin{subfigure}[b]{0.33\textwidth}
		\includegraphics[width=\textwidth]{%
      	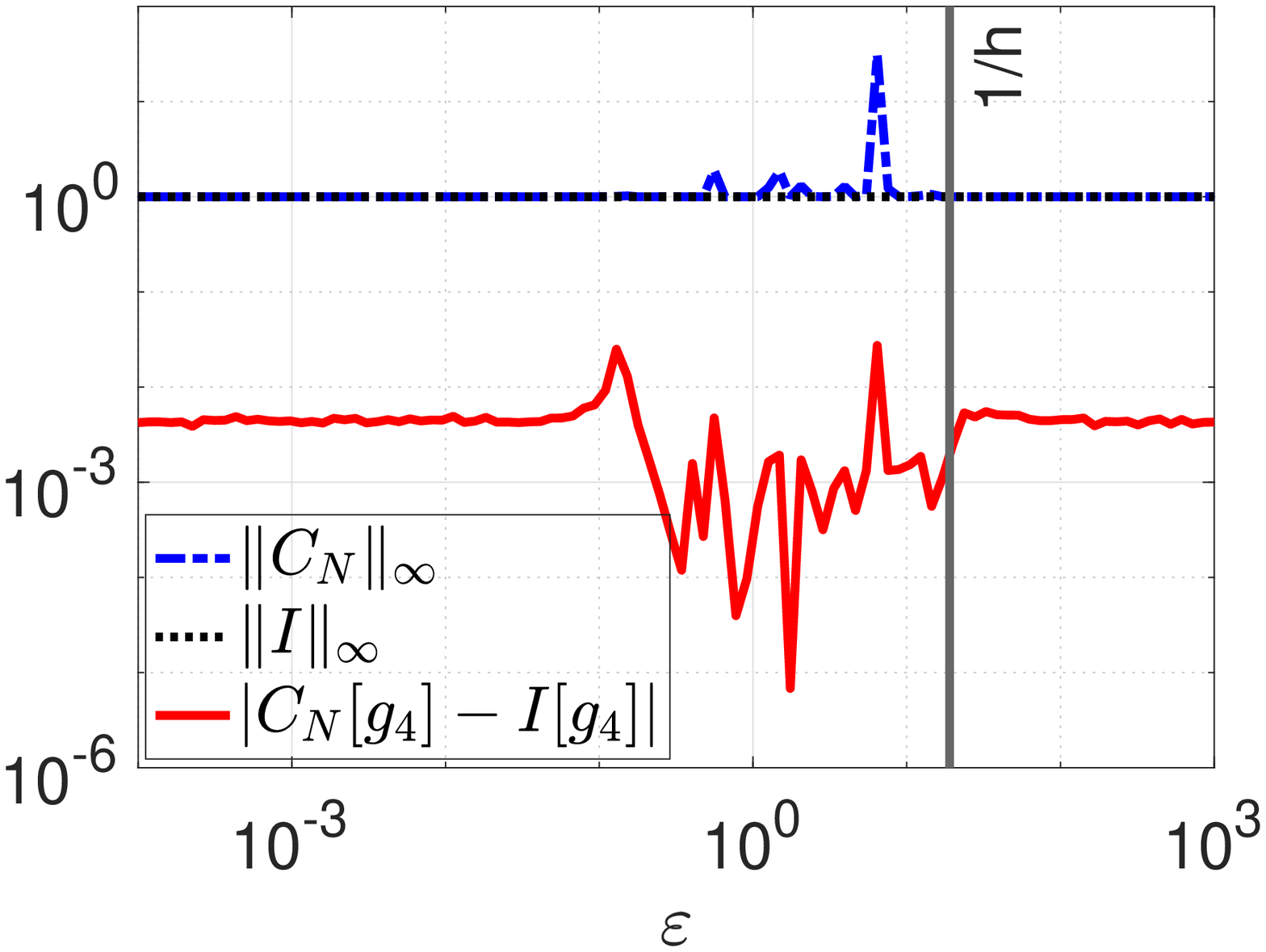} 
    \caption{Equidistant, $d=0$}
    \label{fig:error_2d_Genz4_N400_equid_k1_d0}
  	\end{subfigure}%
	~
	\begin{subfigure}[b]{0.33\textwidth}
		\includegraphics[width=\textwidth]{%
      	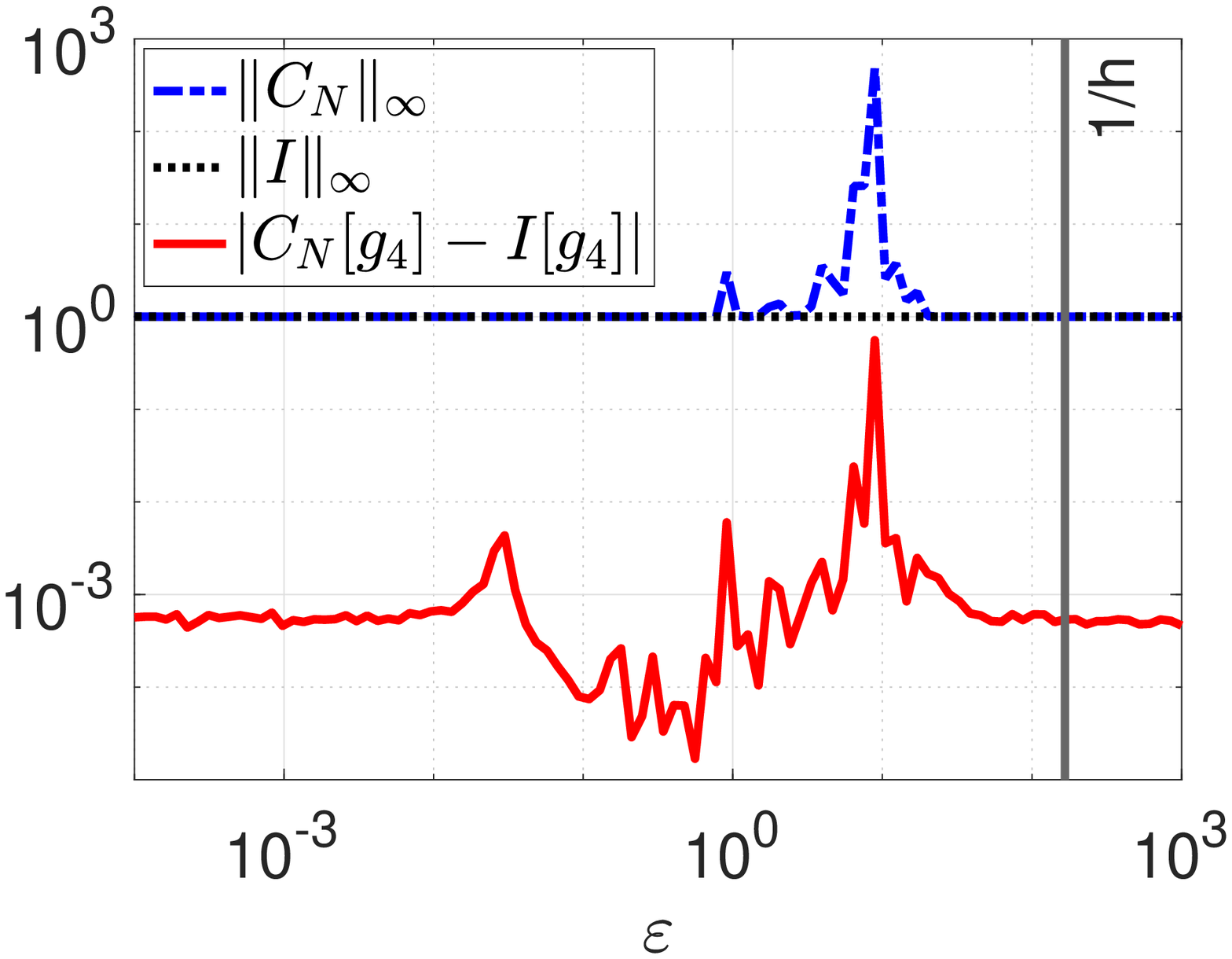} 
    \caption{Halton, $d=0$}
    \label{fig:error_2d_Genz4_N400_Halton_k1_d0}
  	\end{subfigure}%
	~ 
	\begin{subfigure}[b]{0.33\textwidth}
		\includegraphics[width=\textwidth]{%
      	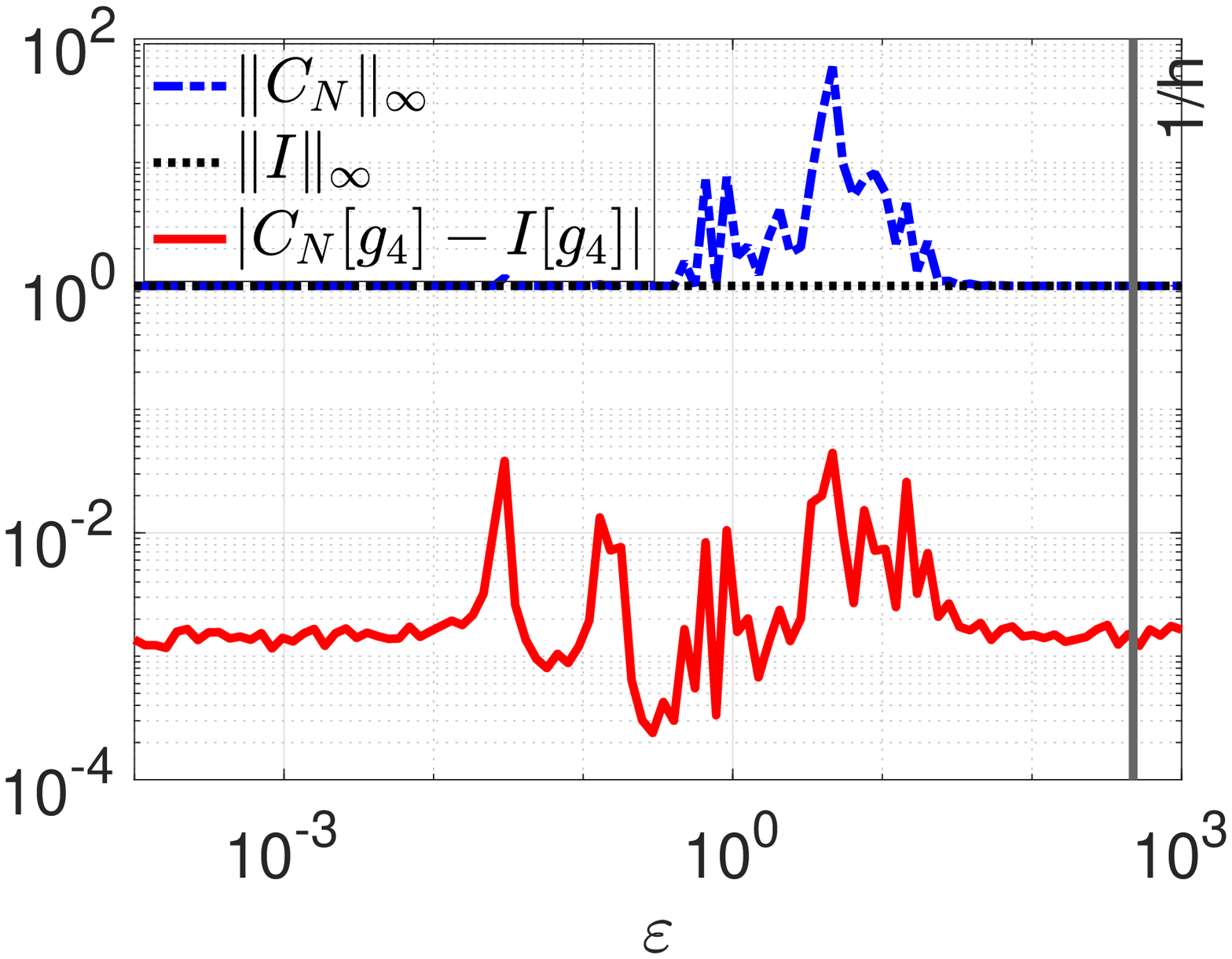} 
    \caption{Random, $d=0$}
    \label{fig:error_2d_Genz4_N400_random_k1_d0}
  	\end{subfigure}%
	\\
	\begin{subfigure}[b]{0.33\textwidth}
		\includegraphics[width=\textwidth]{%
      	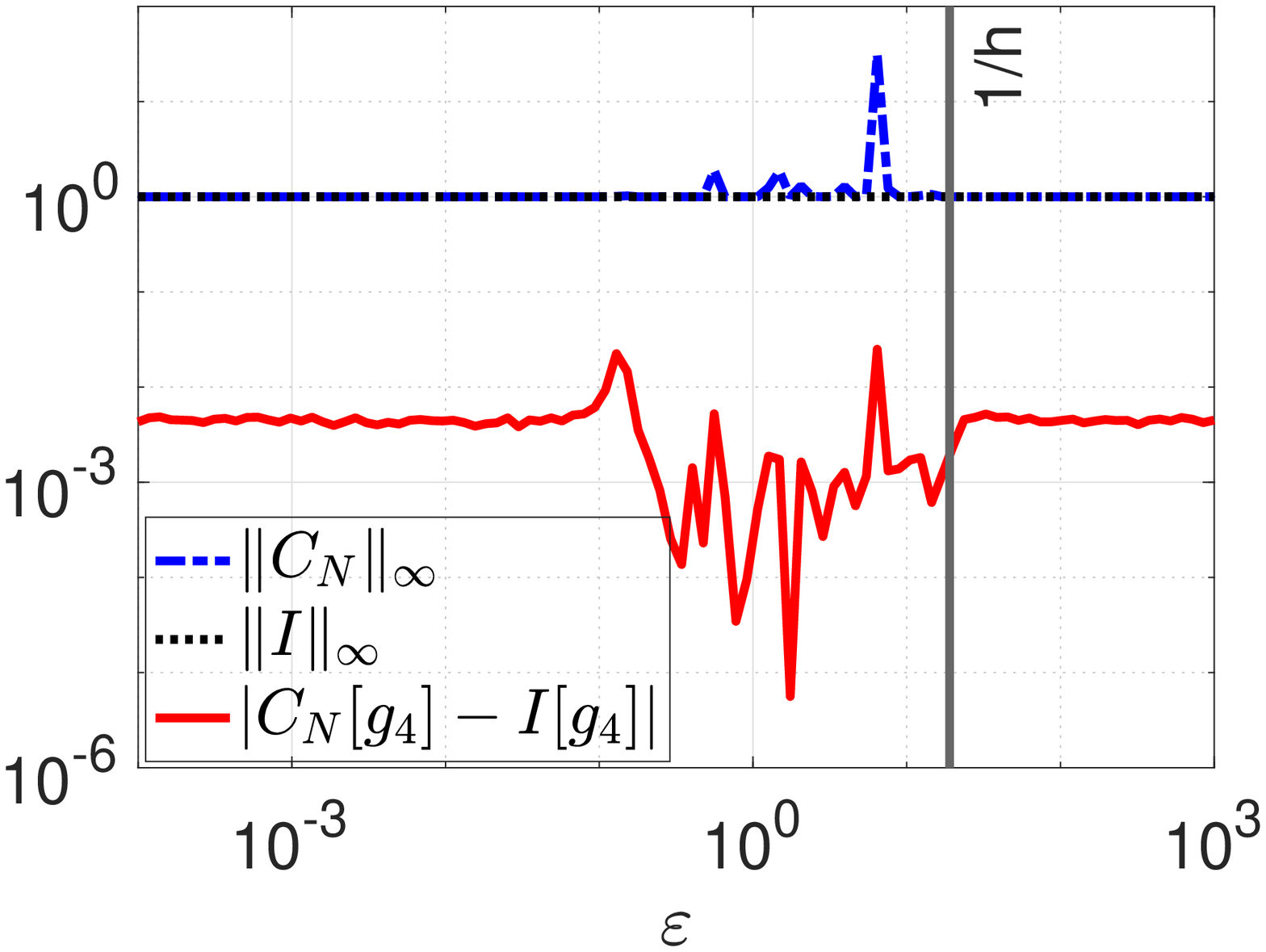} 
    \caption{Equidistant, $d=1$}
    \label{fig:error_2d_Genz4_N400_equid_k1_d1}
  	\end{subfigure}%
	~
	\begin{subfigure}[b]{0.33\textwidth}
		\includegraphics[width=\textwidth]{%
      	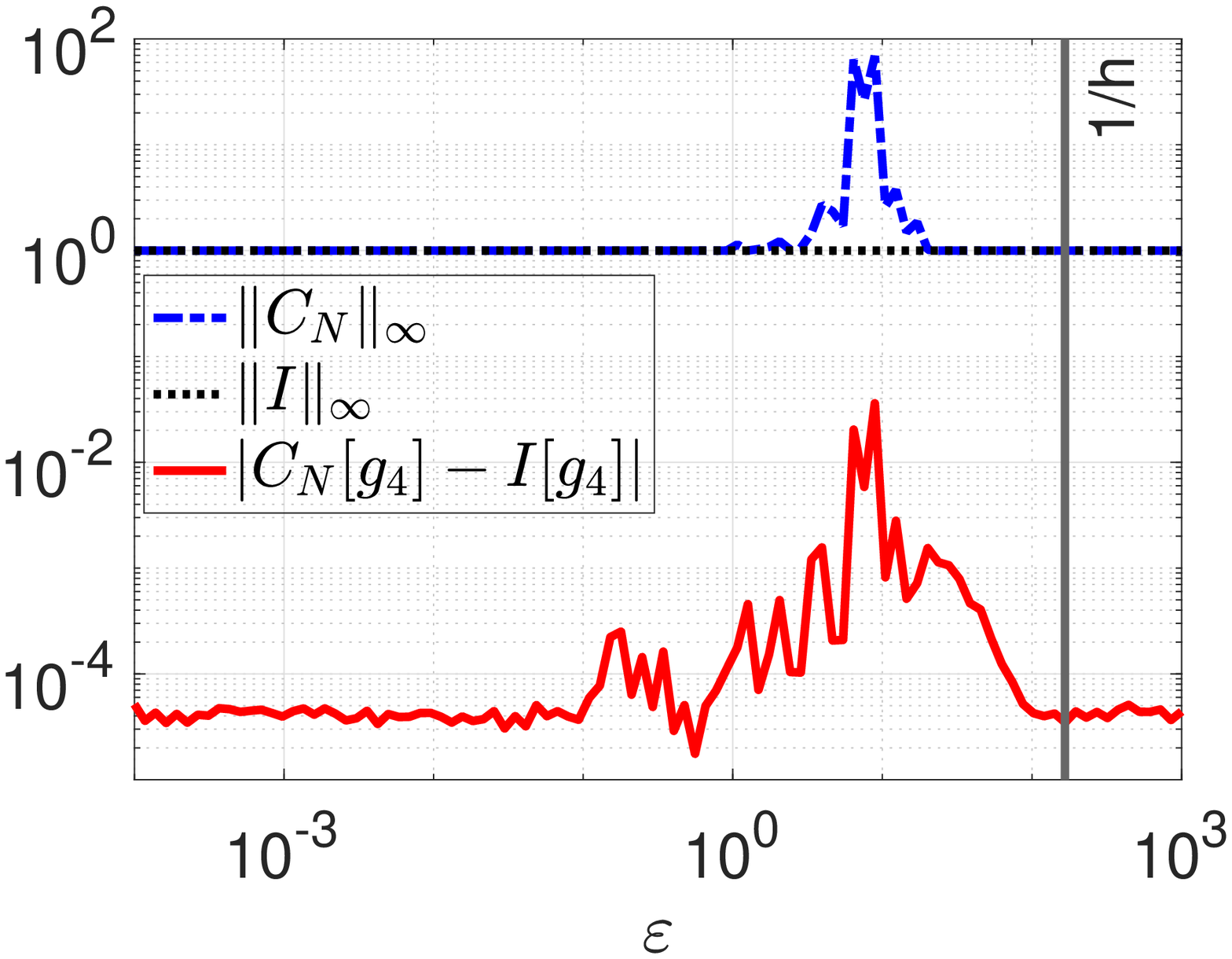} 
    \caption{Halton, $d=1$}
    \label{fig:error_2d_Genz4_N400_Halton_k1_d1}
  	\end{subfigure}%
	~ 
	\begin{subfigure}[b]{0.33\textwidth}
		\includegraphics[width=\textwidth]{%
      	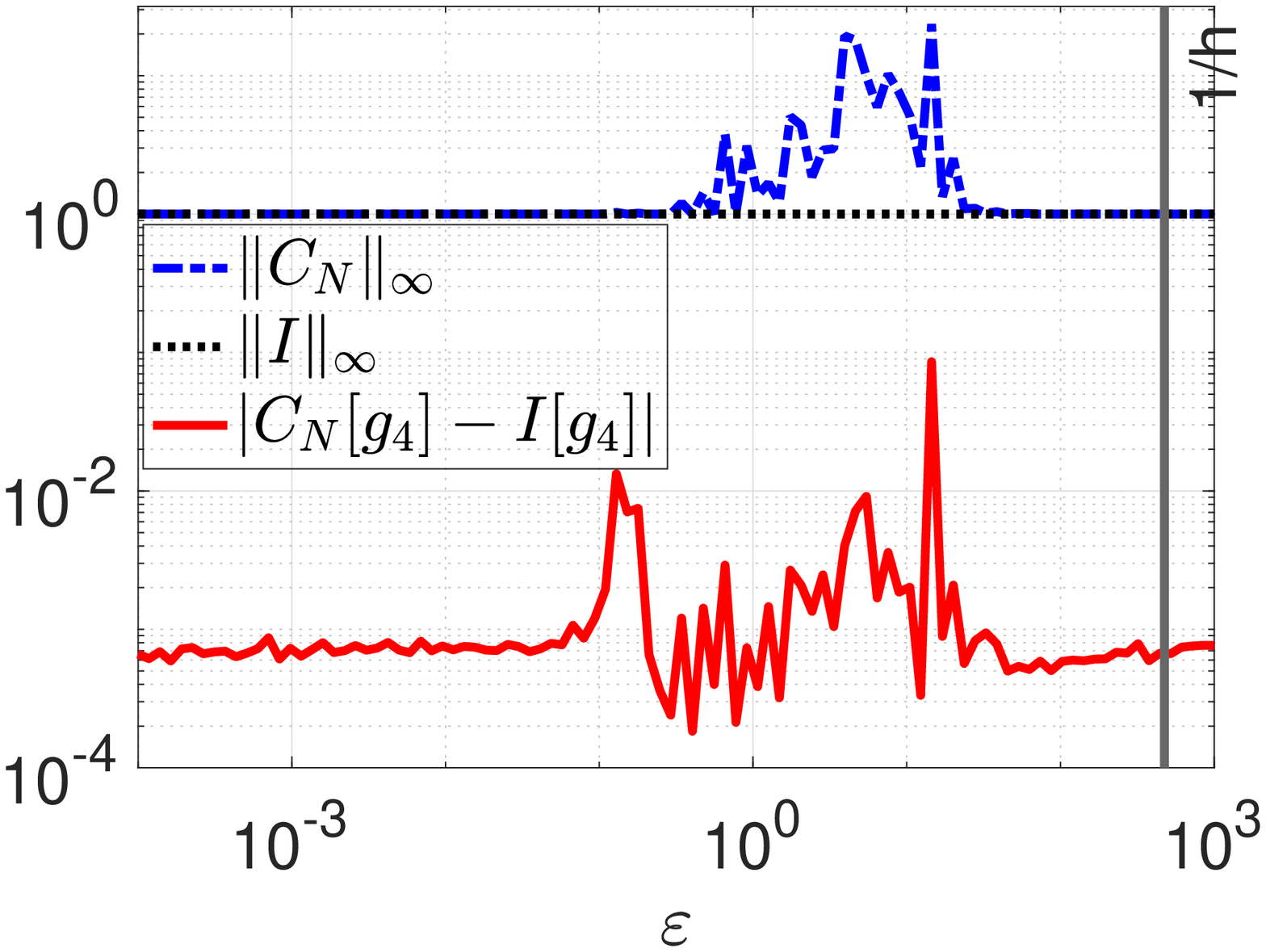} 
    \caption{Random, $d=1$}
    \label{fig:error_2d_Genz4_N400_random_k1_d1}
  	\end{subfigure}%
  	\caption{
	Error analysis for Wendland's compactly supported RBF $\varphi_{2,k}$ in two dimensions with smoothness parameter $k=1$.
	Considered is the fourth Genz test function $g_4$ on $\Omega = [0,1]^2$; see \cref{eq:Genz}.  
	In all cases, $N=400$ data points (equidistant, Halton, or random) were considered, while the reference shape parameter $\varepsilon$ was allowed to vary. 
	$1/h$ denotes the threshold above which the basis functions have nonoverlapping support.
  	}
  	\label{fig:error_2d_Genz4_k1}
\end{figure}

Next, we extend our numerical stability and error analysis to two dimensions, considering the domain $\Omega = [0,1]^2$ and the following Genz test functions \cite{genz1984testing} (also see \cite{van2020adaptive}): 
\begin{equation}\label{eq:Genz}
\begin{aligned}
	g_1(\boldsymbol{x}) 
		& = \cos\left( 2 \pi b_1 + \sum_{i=1}^q a_i x_i \right) \quad 
		&& \text{(oscillatory)}, \\
	g_2(\boldsymbol{x}) 
		& = \prod_{i=1}^q \left( a_i^{-2} + (x_i - b_i)^2 \right)^{-1} \quad 
		&& \text{(product peak)}, \\
	g_3(\boldsymbol{x}) 
		& = \left( 1 + \sum_{i=1}^q a_i x_i \right)^{-(q+1)} \quad 
		&& \text{(corner peak)}, \\
	g_4(\boldsymbol{x}) 
		& = \exp \left( - \sum_{i=1}^q a_i^2 ( x_i - b_i )^2 \right) \quad 
		&& \text{(Gaussian)}
\end{aligned}
\end{equation} 
Here, $q$ denotes the dimension under consideration and is henceforth chosen as $q=2$. 
These functions are designed to have different difficult characteristics for numerical integration routines.
The vectors $\mathbf{a} = (a_1,\dots,a_q)^T$ and $\mathbf{b} = (b_1,\dots,b_q)^T$ respectively contain (randomly chosen) shape and translation parameters. 
For each case, the experiment was repeated $100$ times. 
At the same time, for each experiment, the vectors $\mathbf{a}$ and $\mathbf{b}$ were drawn randomly from $[0,1]^2$. 
For reasons of space, we only report the results for $g_1$ and $g_4$ as well as $k=1$. 
These can be found in \cref{fig:error_2d_Genz1_k1} and \cref{fig:error_2d_Genz4_k1}, respectively. 
As before, the smallest errors are found for shape parameters that correspond to the RBF-CF being stable.  
The results for $g_2, g_3$ and $k=0,2$ are similar and can be found as part of the open-access MATLAB code \cite{glaubitz2021stabilityCode}.  

\begin{table}[tb]
\renewcommand{\arraystretch}{1.4}
\centering 
  	\begin{tabular}{c c c c c c c c c} 
	\toprule 
	& & \multicolumn{3}{c}{$g_1$} & & \multicolumn{3}{c}{$g_4$} \\ \hline 
	& & $e_{\text{min}}$ & $\varepsilon$ & $\|C_N\|_{\infty}$ 
	& & $e_{\text{min}}$ & $\varepsilon$ & $\|C_N\|_{\infty}$ \\ \hline 
	& & \multicolumn{7}{c}{Equidistant Points} \\ \hline 
	$d=0$ 	& & 1.4e-06 & 1.7e+00 & 1.0e+00 
				& & 5.6e-06 & 1.7e+00 & 1.0e+00 \\ 
	$d=1$ 	& & 1.7e-06 & 1.7e+00 & 1.0e+00 			
				& & 6.2e-06 & 1.7e+00 & 1.0e+00 \\ \hline 
	& & \multicolumn{7}{c}{Halton Points} \\ \hline 
	$d=0$ 	& & 5.0e-05 & 5.5e-01 & 1.0e+00 
				& & 2.0e-05 & 5.5e-01 & 1.0e+00 \\ 
	$d=1$ 	& & 1.1e-05 & 5.5e-01 & 1.0e+00 
				& & 1.4e-05 & 5.5e-01 & 1.0e+00 \\ \hline 
	& & \multicolumn{7}{c}{Random Points} \\ \hline 
	$d=0$ 	& & 4.1e-04 & 7.7e-01 & 1.0e+00 
				& & 1.6e-04 & 7.7e-01 & 1.0e+00 \\ \hline 
	$d=1$ 	& & 2.3e-04 & 2.9e-01 & 1.0e+00 
				& & 1.8e-04 & 4.0e-01 & 1.0e+00 \\
	\bottomrule
\end{tabular} 
\caption{Minimal errors, $e_{\text{min}}$, for the first and fourth Genz test function, $g_1$ and $g_4$, together with the corresponding shape parameter, $\varepsilon$, and stability measure, $\|C_N\|_{\infty}$. 
	In all cases, Wendland's compactly supported RBF with smoothness parameter $k=1$ was used.}
\label{tab:min_error_Wendland}
\end{table}

It might be hard to identify the smallest errors as well as the corresponding shape parameter and stability measure from \cref{fig:error_2d_Genz1_k1} and \cref{fig:error_2d_Genz4_k1}. 
Hence, these are listed separately in \cref{tab:min_error_Wendland}.

\subsection{Gaussian RBF} 
\label{sub:num_Gaussian}

Here, we perform a similar investigation of stability and accuracy as in \cref{sub:num_compact} for the Gaussian RBF, given by $\varphi(r) = \exp( \varepsilon^2 r^2 )$. 

\begin{figure}[tb]
	\centering 
	\begin{subfigure}[b]{0.33\textwidth}
		\includegraphics[width=\textwidth]{%
      	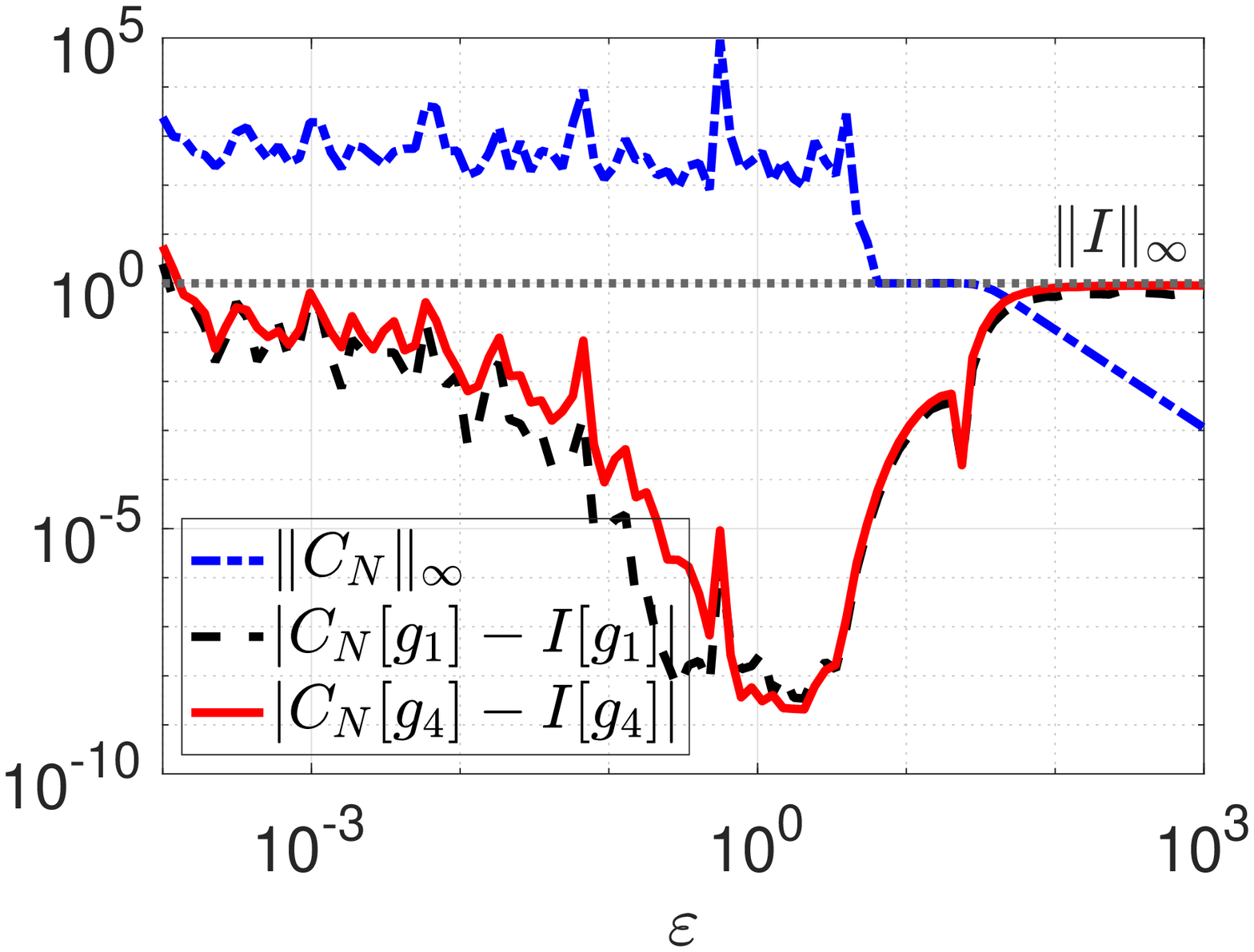} 
    \caption{Equidistant, $d=-1$}
    \label{fig:G_Genz14_N400_equid_noPol}
  	\end{subfigure}%
	~
	\begin{subfigure}[b]{0.33\textwidth}
		\includegraphics[width=\textwidth]{%
      	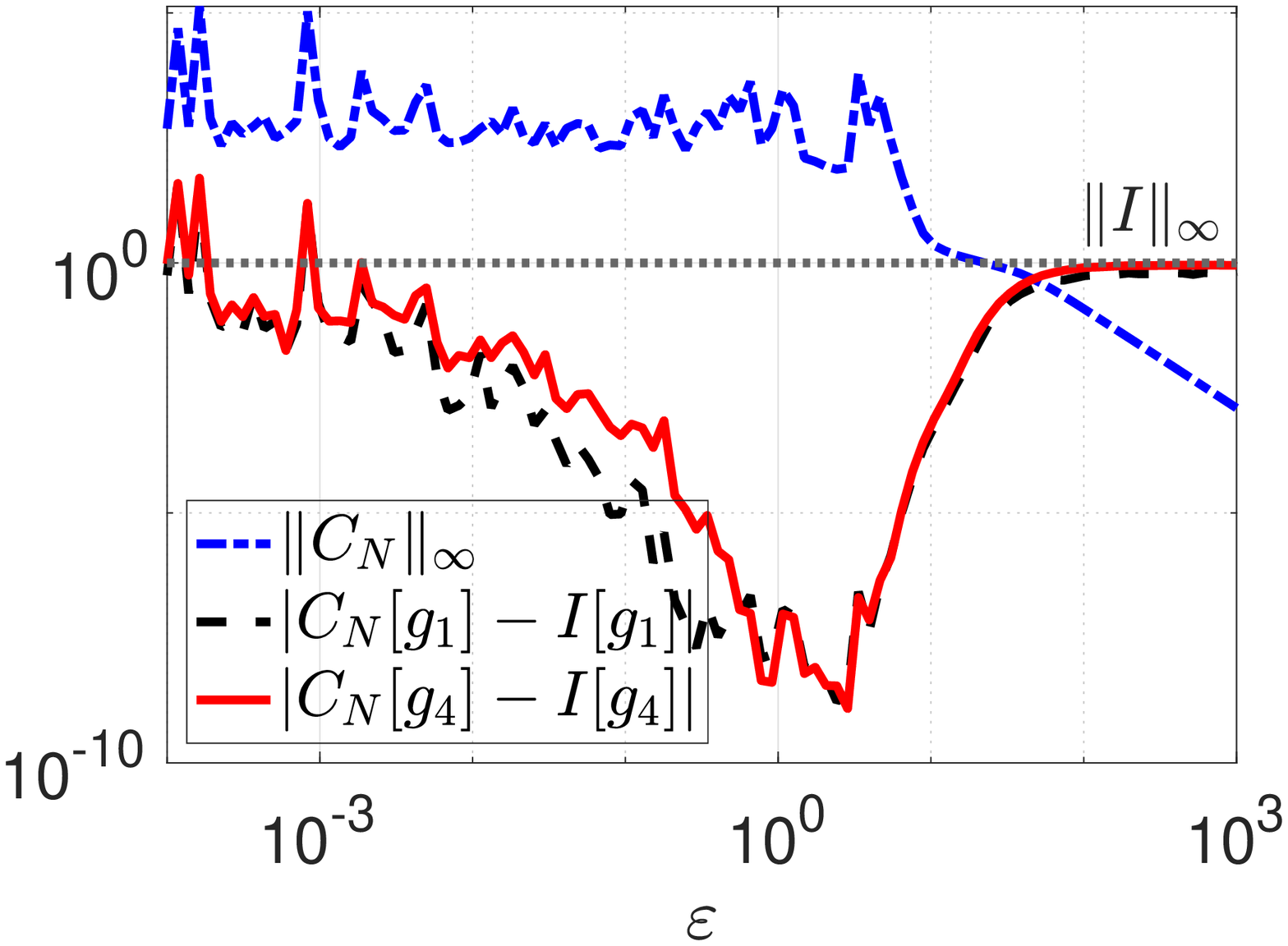} 
    \caption{Halton, $d=-1$}
    \label{fig:G_Genz14_N400_Halton_noPol}
  	\end{subfigure}%
	~ 
	\begin{subfigure}[b]{0.33\textwidth}
		\includegraphics[width=\textwidth]{%
      	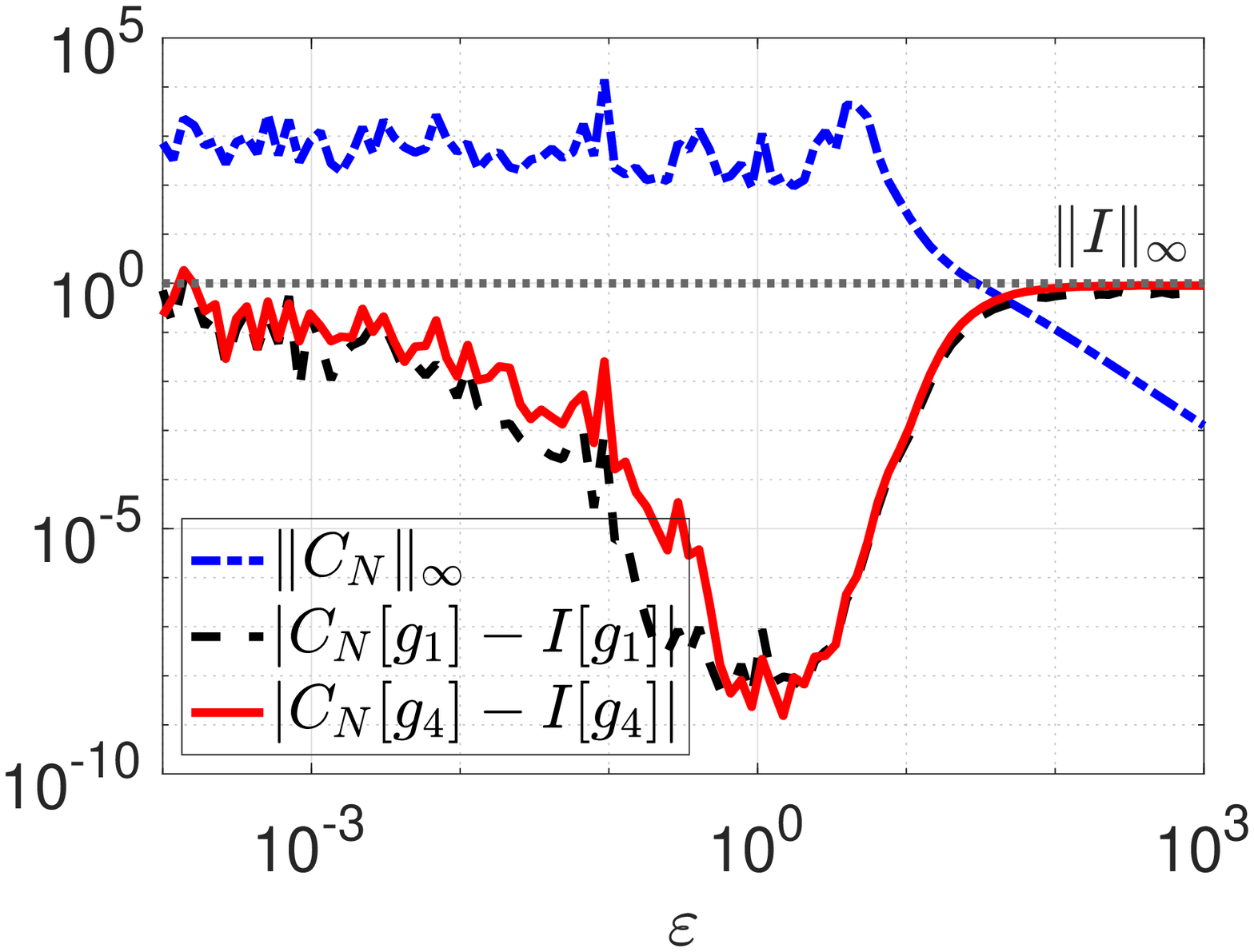} 
    \caption{Random, $d=-1$}
    \label{fig:G_Genz14_N400_random_noPol}
  	\end{subfigure}%
	\\
  	\begin{subfigure}[b]{0.33\textwidth}
		\includegraphics[width=\textwidth]{%
      	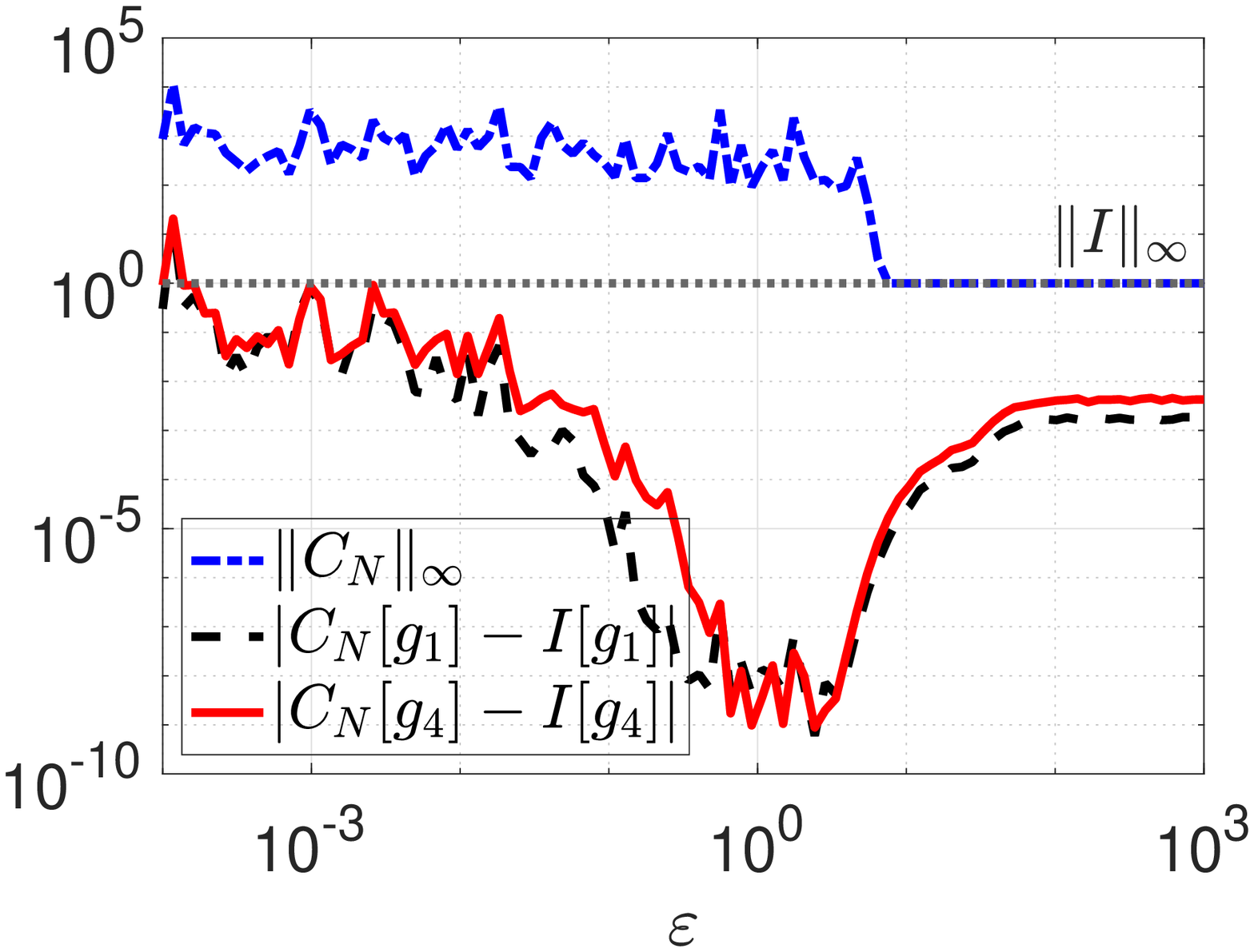} 
    \caption{Equidistant, $d=0$}
    \label{fig:G_Genz14_N400_equid_d0}
  	\end{subfigure}%
	~
	\begin{subfigure}[b]{0.33\textwidth}
		\includegraphics[width=\textwidth]{%
      	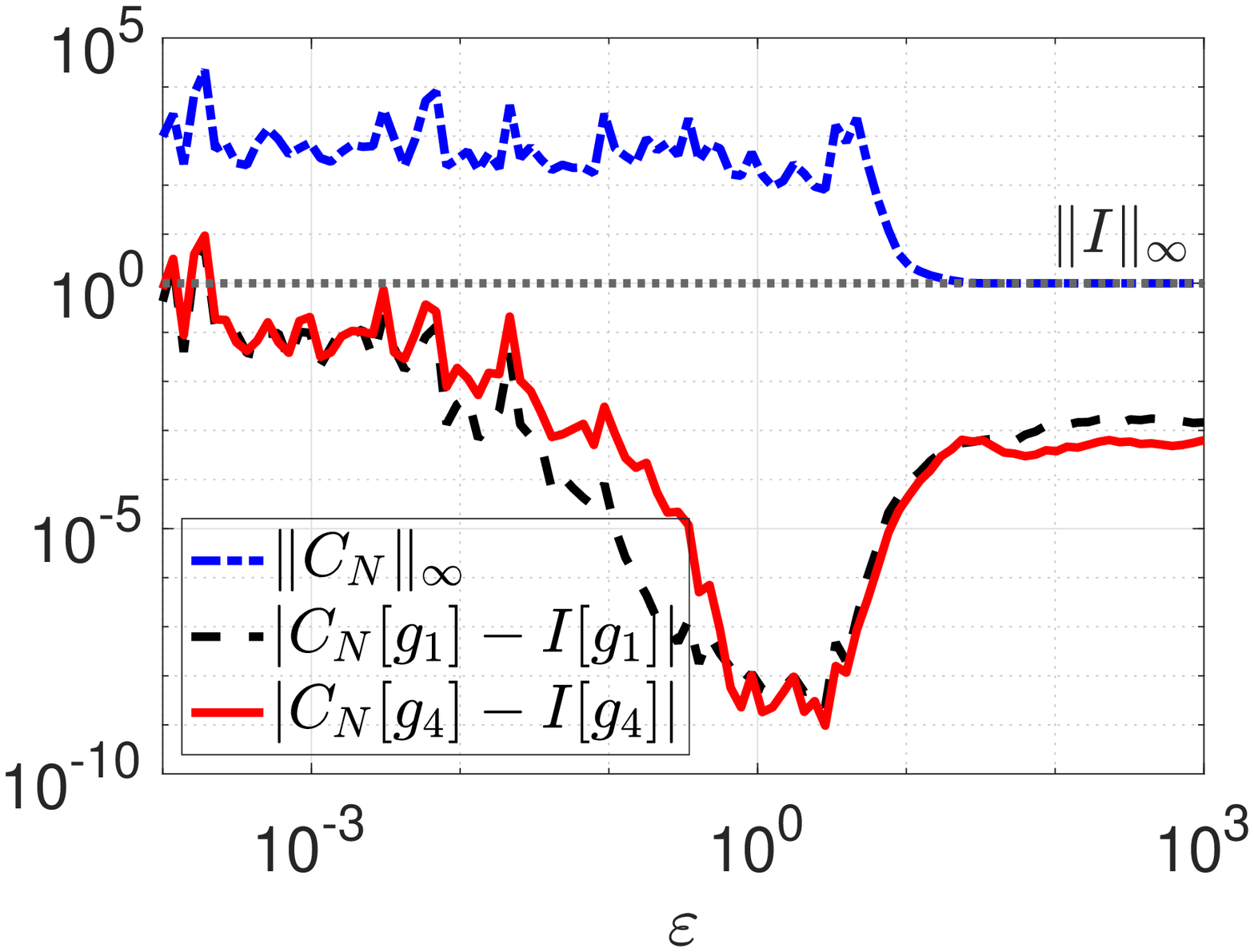} 
    \caption{Halton, $d=0$}
    \label{fig:G_Genz14_N400_Halton_d0}
  	\end{subfigure}%
	~ 
	\begin{subfigure}[b]{0.33\textwidth}
		\includegraphics[width=\textwidth]{%
      	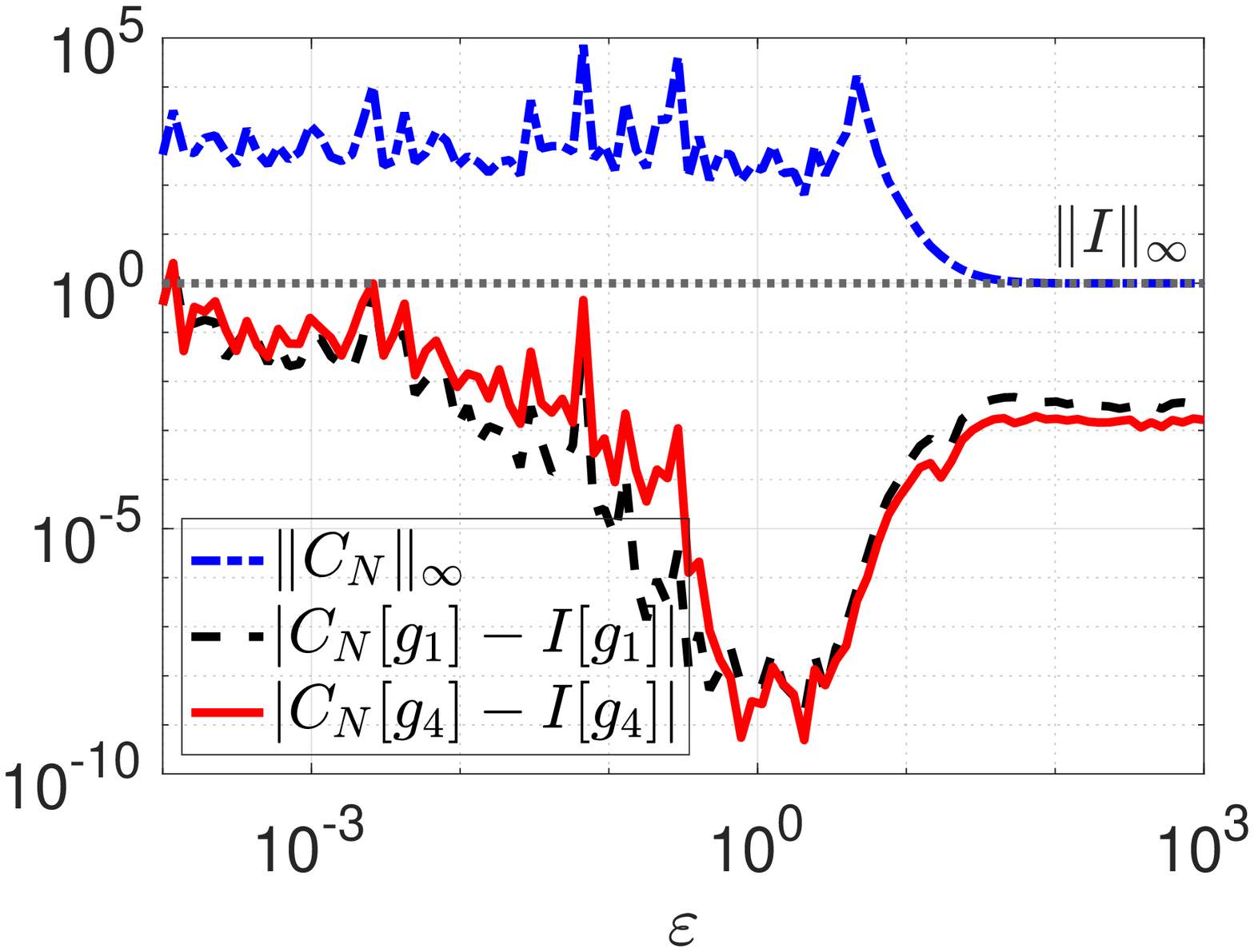} 
    \caption{Random, $d=0$}
    \label{fig:G_Genz14_N400_random_d0}
  	\end{subfigure}%
	\\
	\begin{subfigure}[b]{0.33\textwidth}
		\includegraphics[width=\textwidth]{%
      	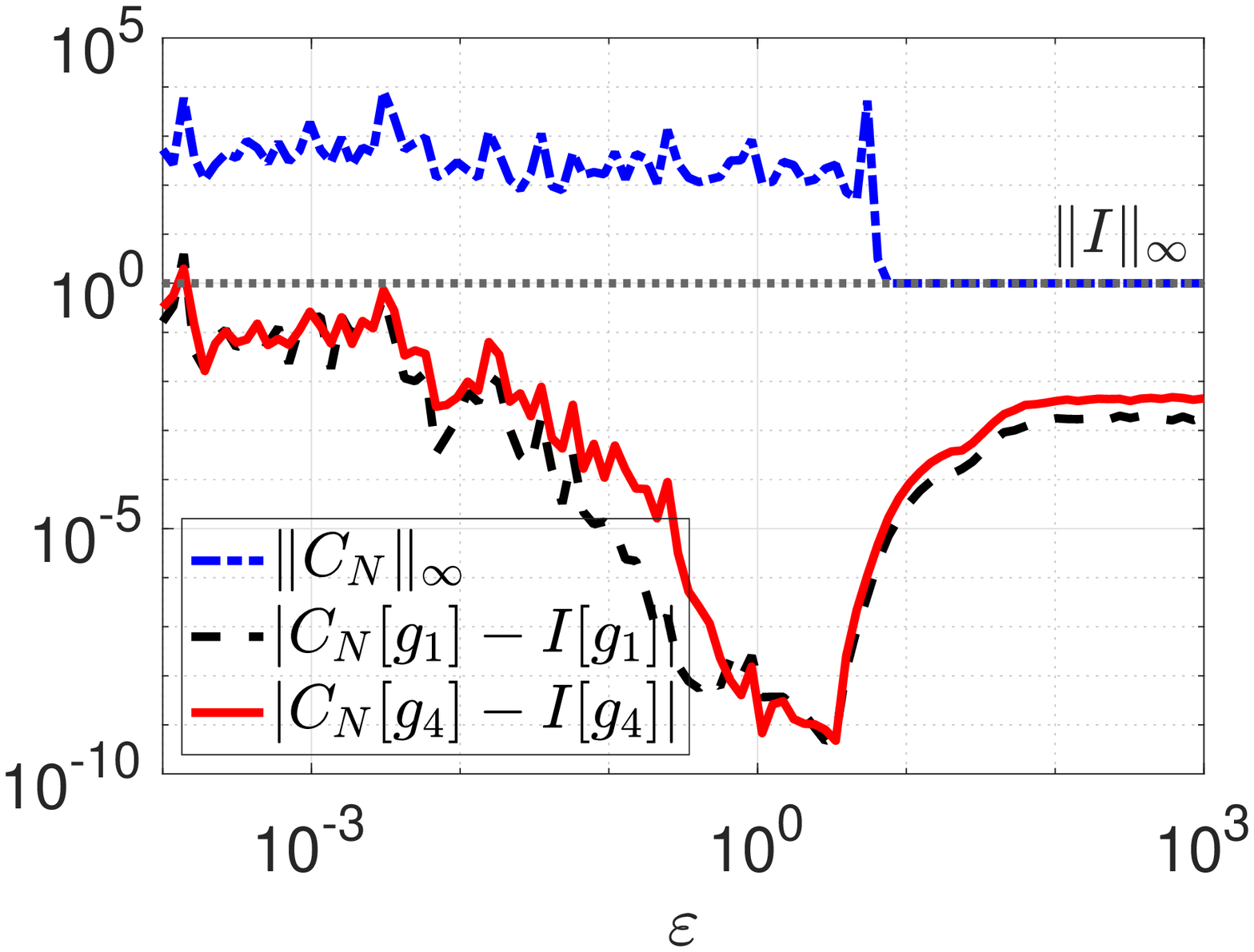} 
    \caption{Equidistant, $d=1$}
    \label{fig:G_Genz14_N400_equid_d1}
  	\end{subfigure}%
	~
	\begin{subfigure}[b]{0.33\textwidth}
		\includegraphics[width=\textwidth]{%
      	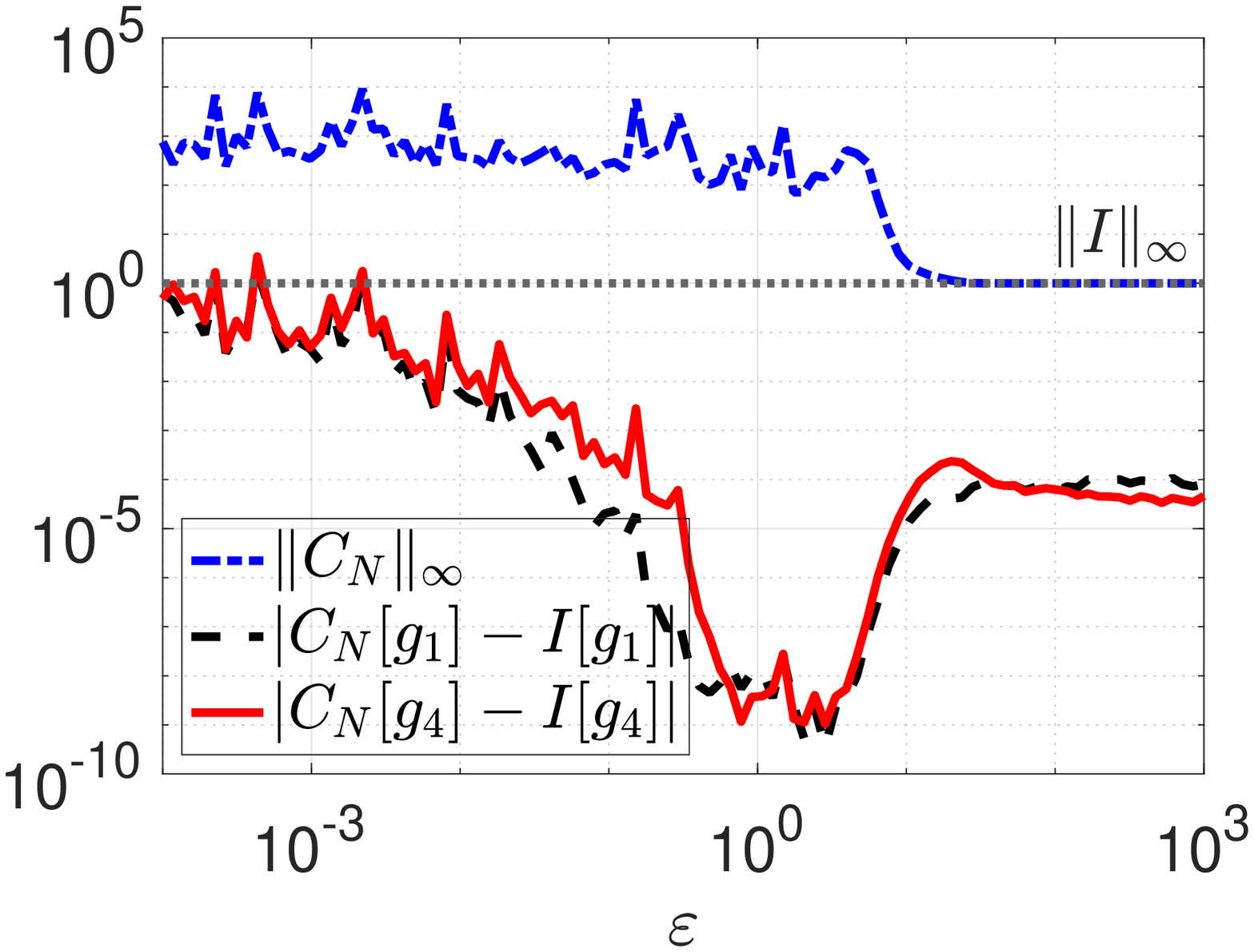} 
    \caption{Halton, $d=1$}
    \label{fig:G_Genz14_N400_Halton_d1}
  	\end{subfigure}%
	~ 
	\begin{subfigure}[b]{0.33\textwidth}
		\includegraphics[width=\textwidth]{%
      	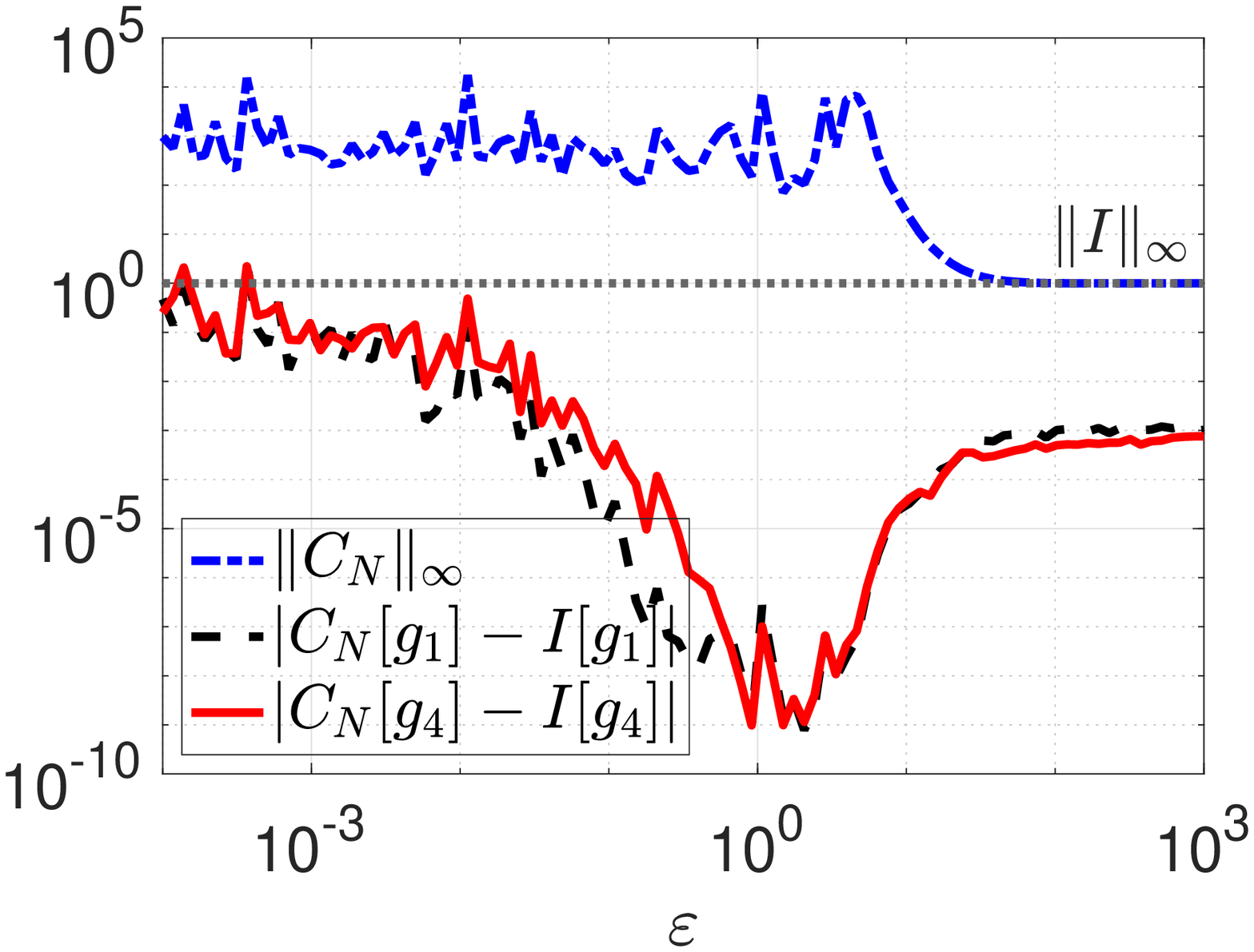} 
    \caption{Random, $d=1$}
    \label{fig:G_Genz14_N400_random_d1}
  	\end{subfigure}%
  	\caption{
	Error analysis for the Gaussian RBF $\varphi(r) = \exp( \varepsilon^2 r^2 ) $ in two dimensions for the first and fourth Genz test function $g_1, g_4$ on $\Omega = [0,1]^2$; see \cref{eq:Genz}.  
	In all cases, $N=400$ data points (equidistant, Halton, or random) were considered, while the reference shape parameter $\varepsilon$ was allowed to vary.
  	}
  	\label{fig:error_Gauss_2d_Genz14}
\end{figure}

In particuar, \cref{fig:error_Gauss_2d_Genz14} reports on the stability measure $\|C_N\|_\infty$ for the Gaussian RBF-CF and the corresponding errors for the first and fourth Genz test function on $\Omega = [0,1]^2$ for $N=400$ data points. 
These are given as equidistant, Halton and random points, respectively. 
Furthermore, the shape parameter was allowed to vary from $10^{-4}$ to $10^3$ and the RBF-CF was computed by augmenting the RBF basis with no ($d=-1$) polynomials, a constant ($d=0$), or a linear term ($d=1$). 
Also for the Gaussian RBFs, we observe the RBF-CFs to be stable for a sufficiently large shape parameter. 
It might be argued that this is because the Gaussian RBF can be considered as being  ``close" to a compactly supported RBF for large shape parameter.\footnote{Of course, strictly speaking, the Gaussian RBF does not have compact support. Yet, for large $\varepsilon^2 r^2$ its function value will lie below machine precision, making it compactly supported in a numerical sense.}
At the same time, however, the Gaussian RBF-CF are observed to become unstable for decreasing shape parameter $\varepsilon$. 
Furthermore, we observe the smallest error to occur in a region of instability in this case. 
Roughly speaking, this shape parameter---providing a minimal error---usually lies slightly below the smallest shape parameter that yields a stable RBF-CF. 
This might be explained by this shape parameter balancing out the two terms in \cref{eq:L-inequality}. 
One the one hand, the RBF space $\mathcal{S}_{N,d,\varepsilon}$ should provide a best approximation that is as close as possible to the underlying function, $f$. 
This is reflected in the term $\norm{ f - s }_{L^{\infty}(\Omega)}$ on the right hand side of \cref{eq:L-inequality}. 
On the other hand, the stability measure of the corresponding RBF-CF should be as small as possible. 
This is reflected in the term $\| I \|_{\infty} + \| C_N \|_{\infty}$, by which $\norm{ f - s }_{L^{\infty}(\Omega)}$ is multiplied in \cref{eq:L-inequality}. 
While for Gaussian RBFs the best approximation becomes more accurate for a decreasing shape parameter, the stability measure benefits from increasing shape parameters. 
In this case, the balance between these two objectives---and therefore the smallest error---is found outside of the region of stability. 

\begin{figure}[tb]
	\centering 
	\begin{subfigure}[b]{0.33\textwidth}
		\includegraphics[width=\textwidth]{%
      	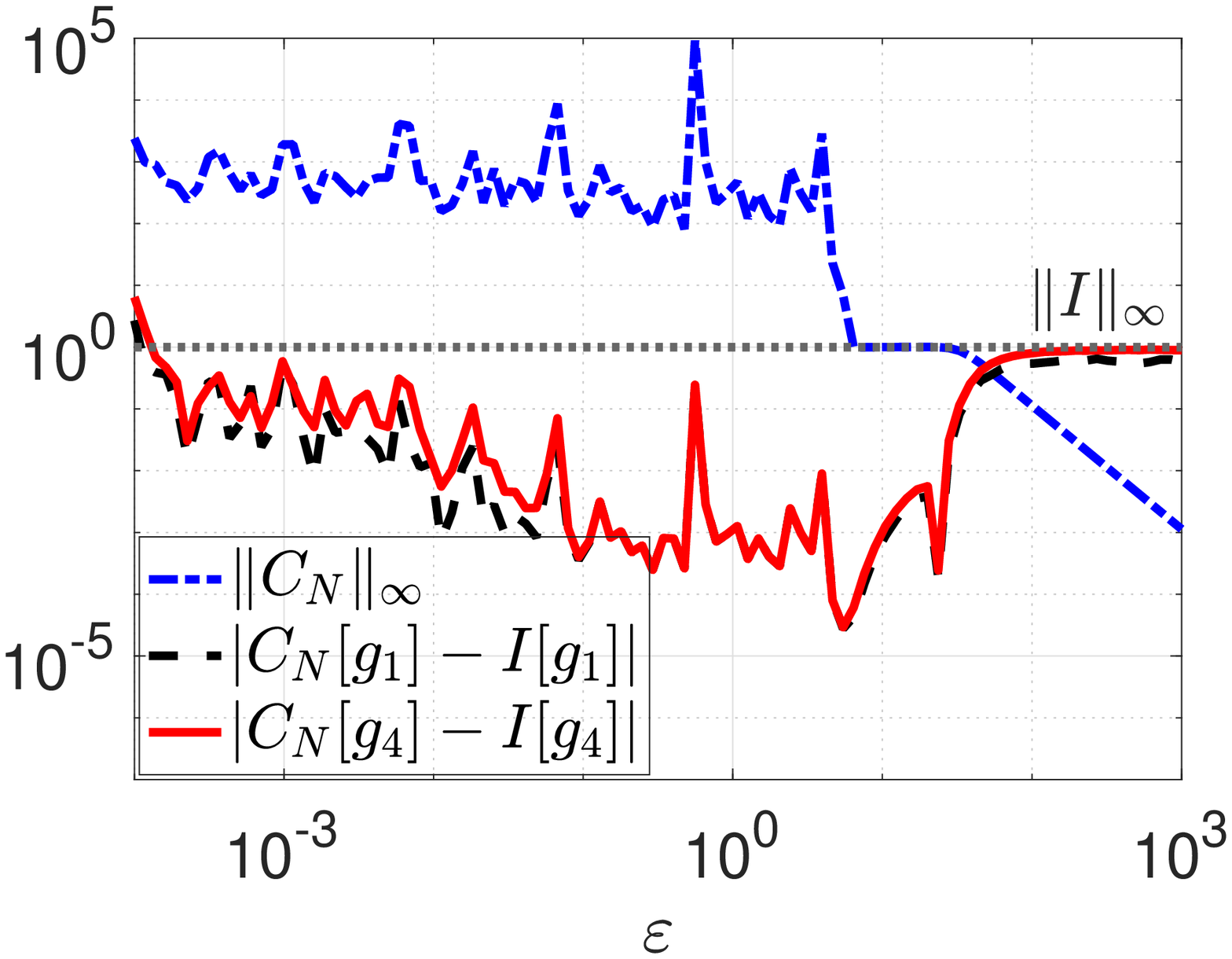} 
    \caption{Equidistant, $d=-1$}
    \label{fig:G_Genz14_N400_equid_noPol_noise4}
  	\end{subfigure}%
	~
	\begin{subfigure}[b]{0.33\textwidth}
		\includegraphics[width=\textwidth]{%
      	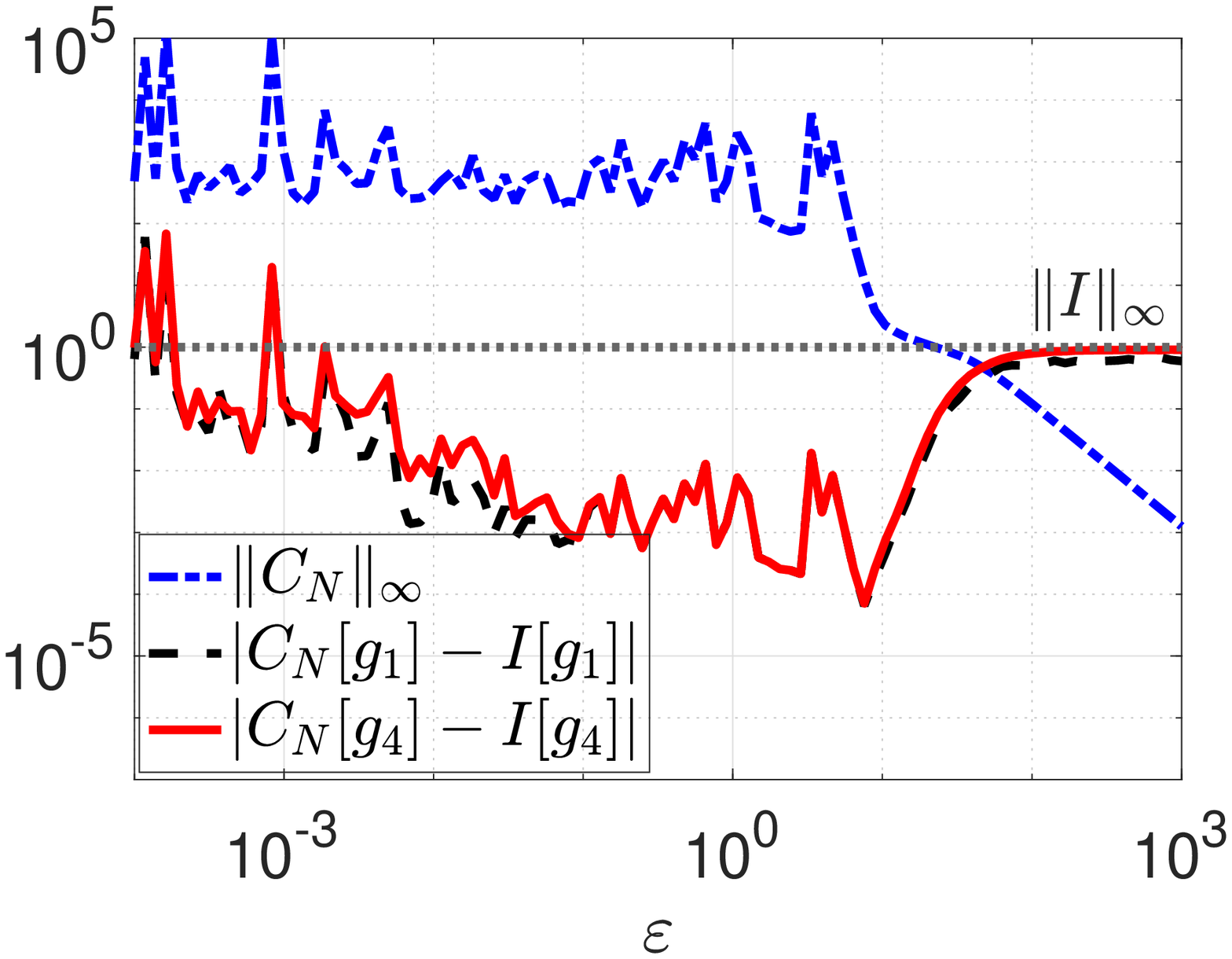} 
    \caption{Halton, $d=-1$}
    \label{fig:G_Genz14_N400_Halton_noPol_noise4}
  	\end{subfigure}%
	~ 
	\begin{subfigure}[b]{0.33\textwidth}
		\includegraphics[width=\textwidth]{%
      	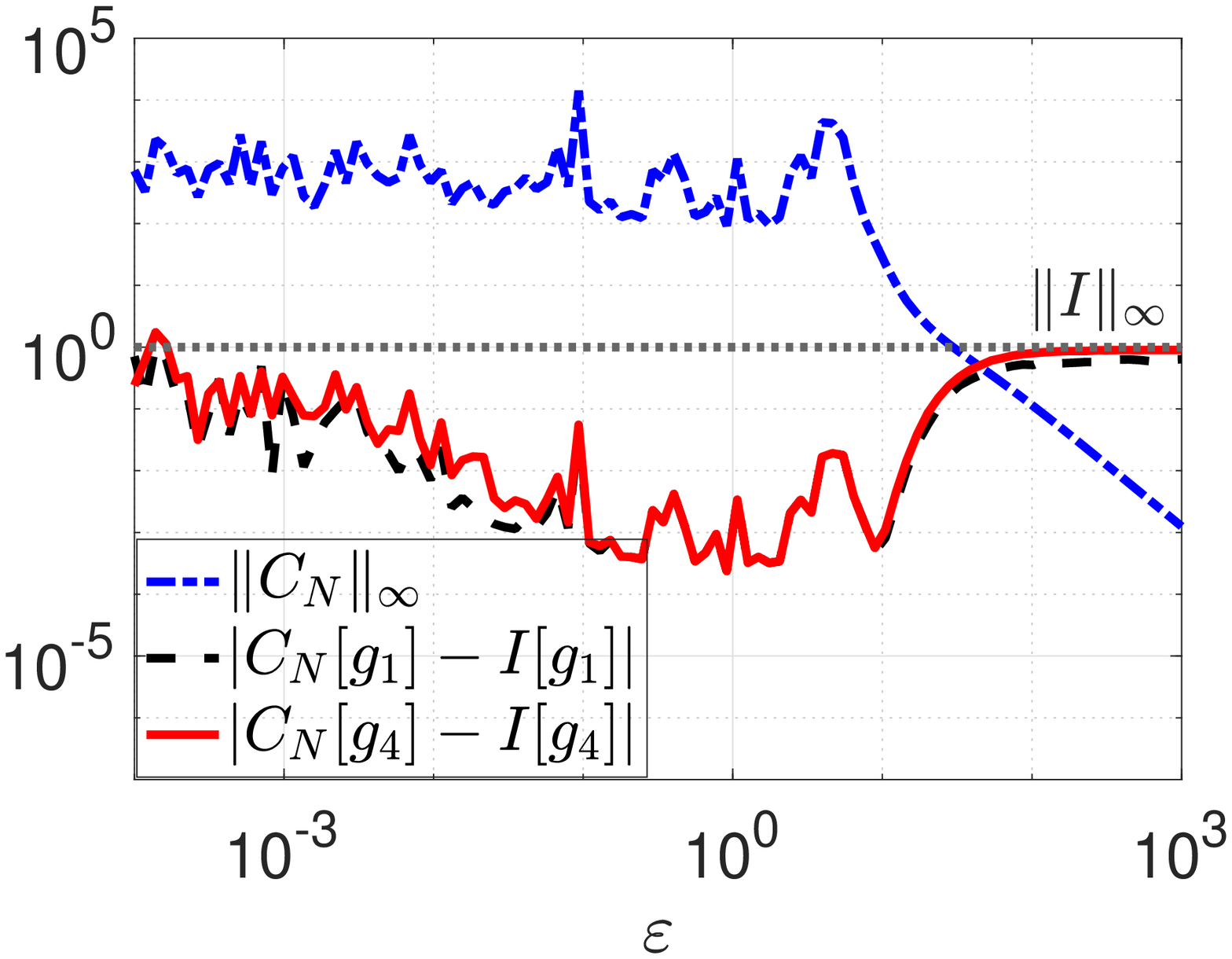} 
    \caption{Random, $d=-1$}
    \label{fig:G_Genz14_N400_random_noPol_noise4}
  	\end{subfigure}%
	\\
  	\begin{subfigure}[b]{0.33\textwidth}
		\includegraphics[width=\textwidth]{%
      	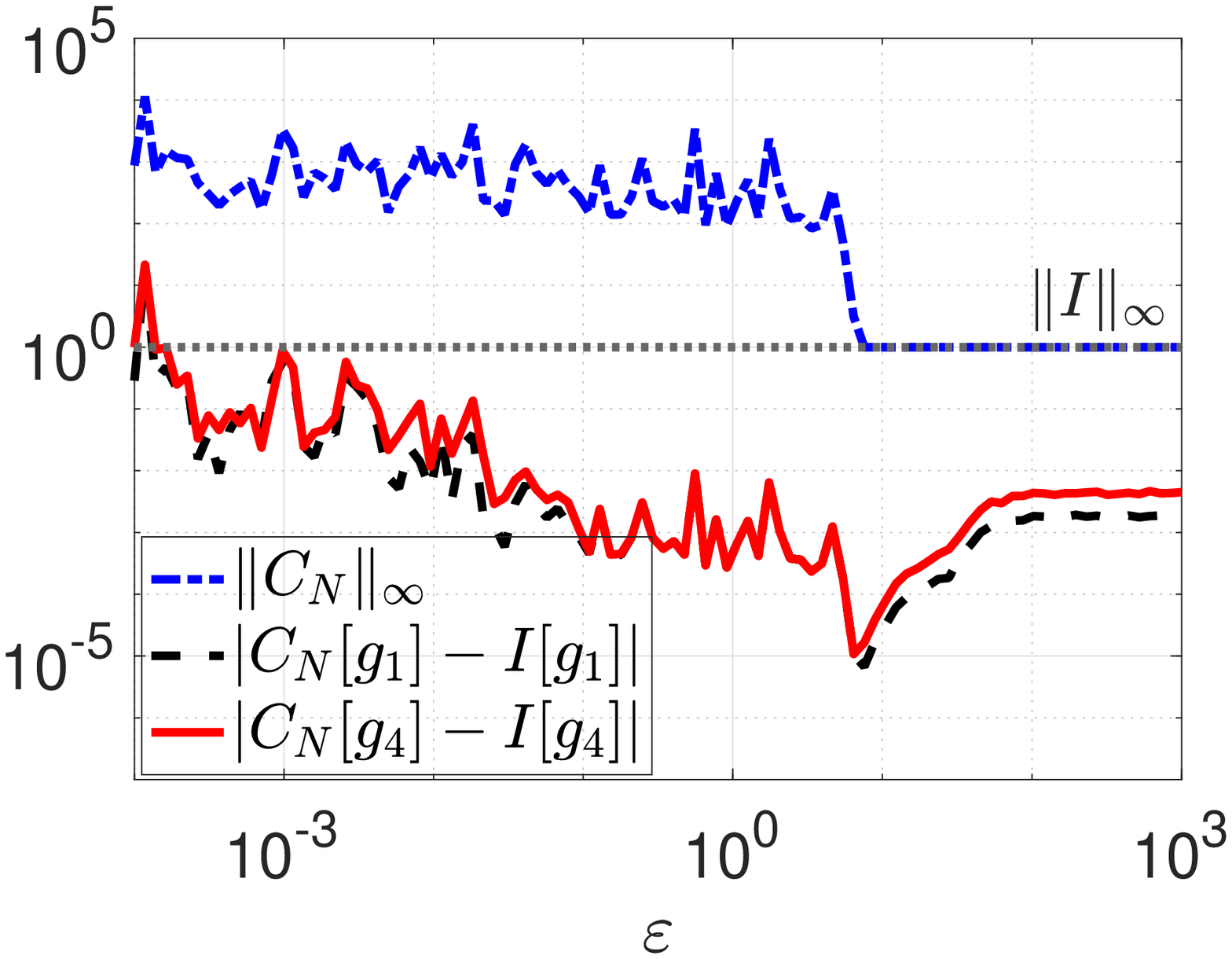} 
    \caption{Equidistant, $d=0$}
    \label{fig:G_Genz14_N400_equid_d0_noise4}
  	\end{subfigure}%
	~
	\begin{subfigure}[b]{0.33\textwidth}
		\includegraphics[width=\textwidth]{%
      	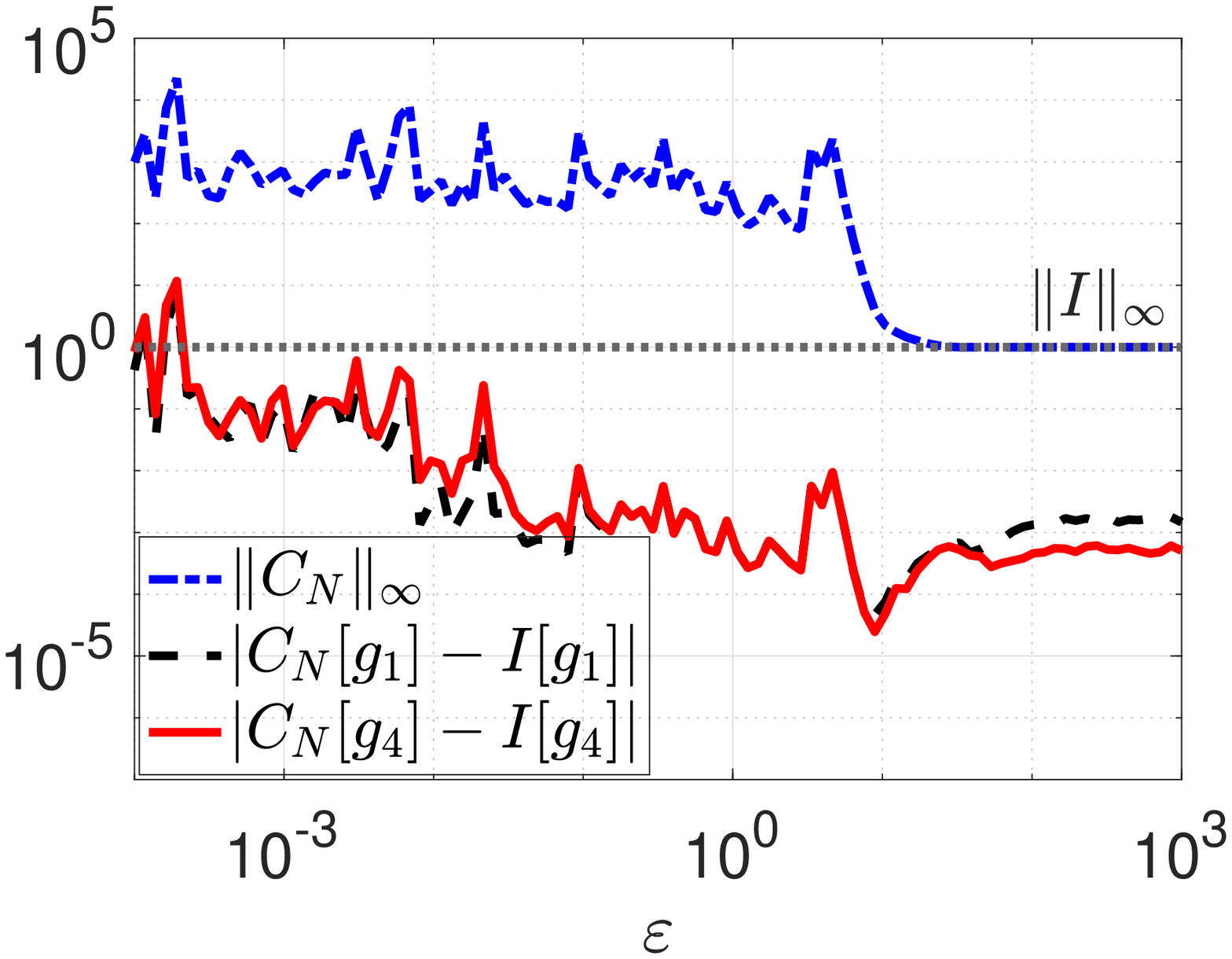} 
    \caption{Halton, $d=0$}
    \label{fig:G_Genz14_N400_Halton_d0_noise4}
  	\end{subfigure}%
	~ 
	\begin{subfigure}[b]{0.33\textwidth}
		\includegraphics[width=\textwidth]{%
      	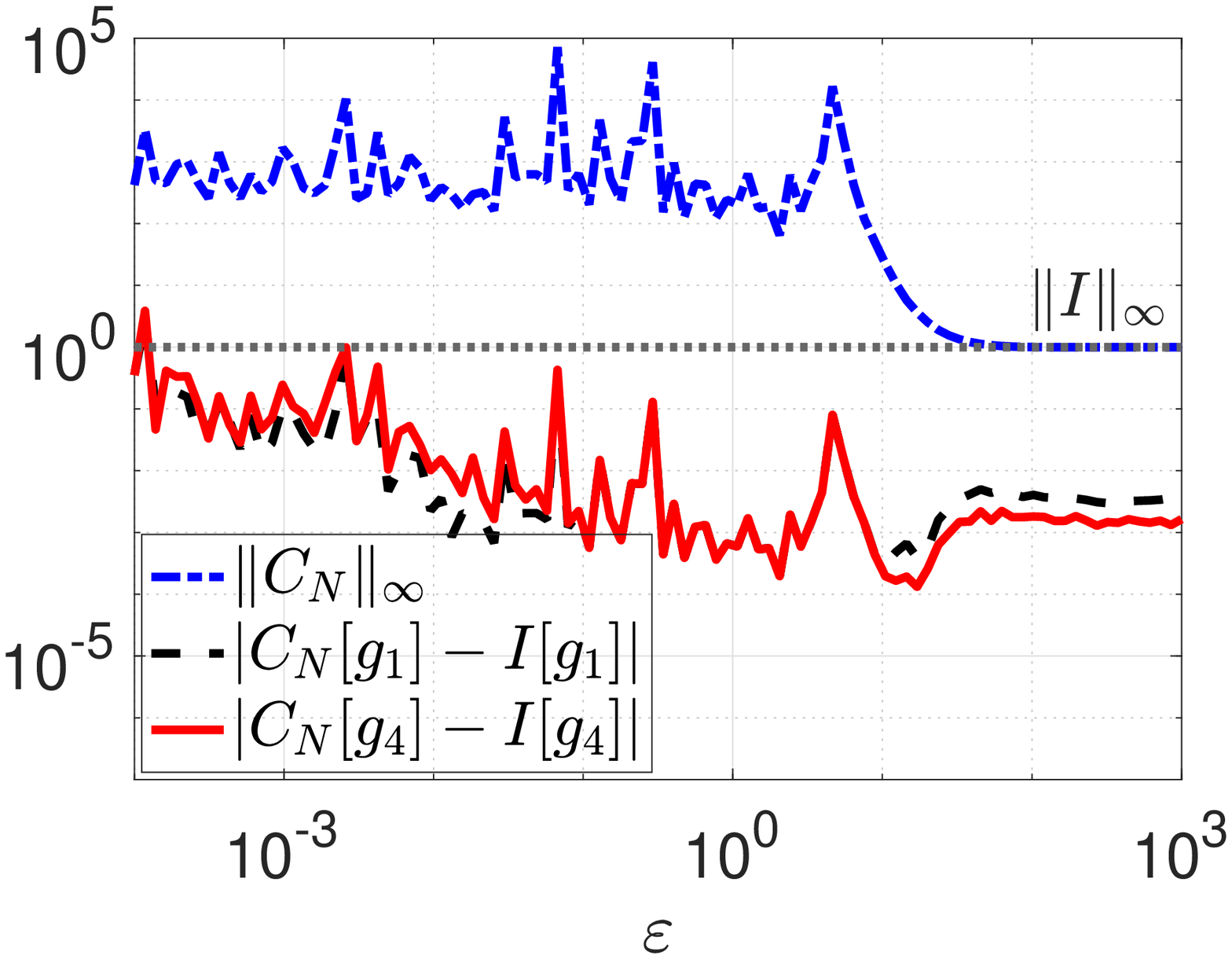} 
    \caption{Random, $d=0$}
    \label{fig:G_Genz14_N400_random_d0_noise4}
  	\end{subfigure}%
	\\
	\begin{subfigure}[b]{0.33\textwidth}
		\includegraphics[width=\textwidth]{%
      	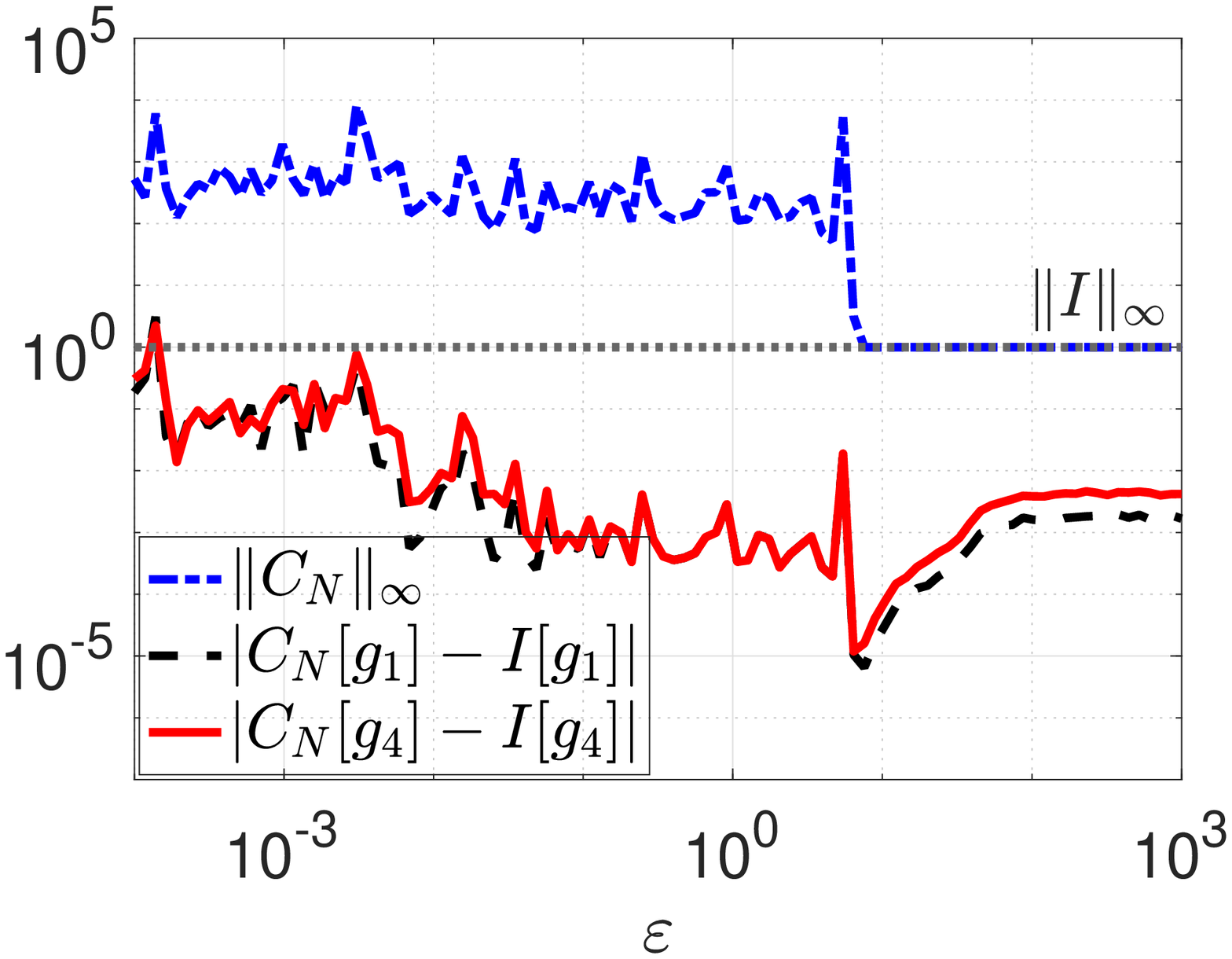} 
    \caption{Equidistant, $d=1$}
    \label{fig:G_Genz14_N400_equid_d1_noise4}
  	\end{subfigure}%
	~
	\begin{subfigure}[b]{0.33\textwidth}
		\includegraphics[width=\textwidth]{%
      	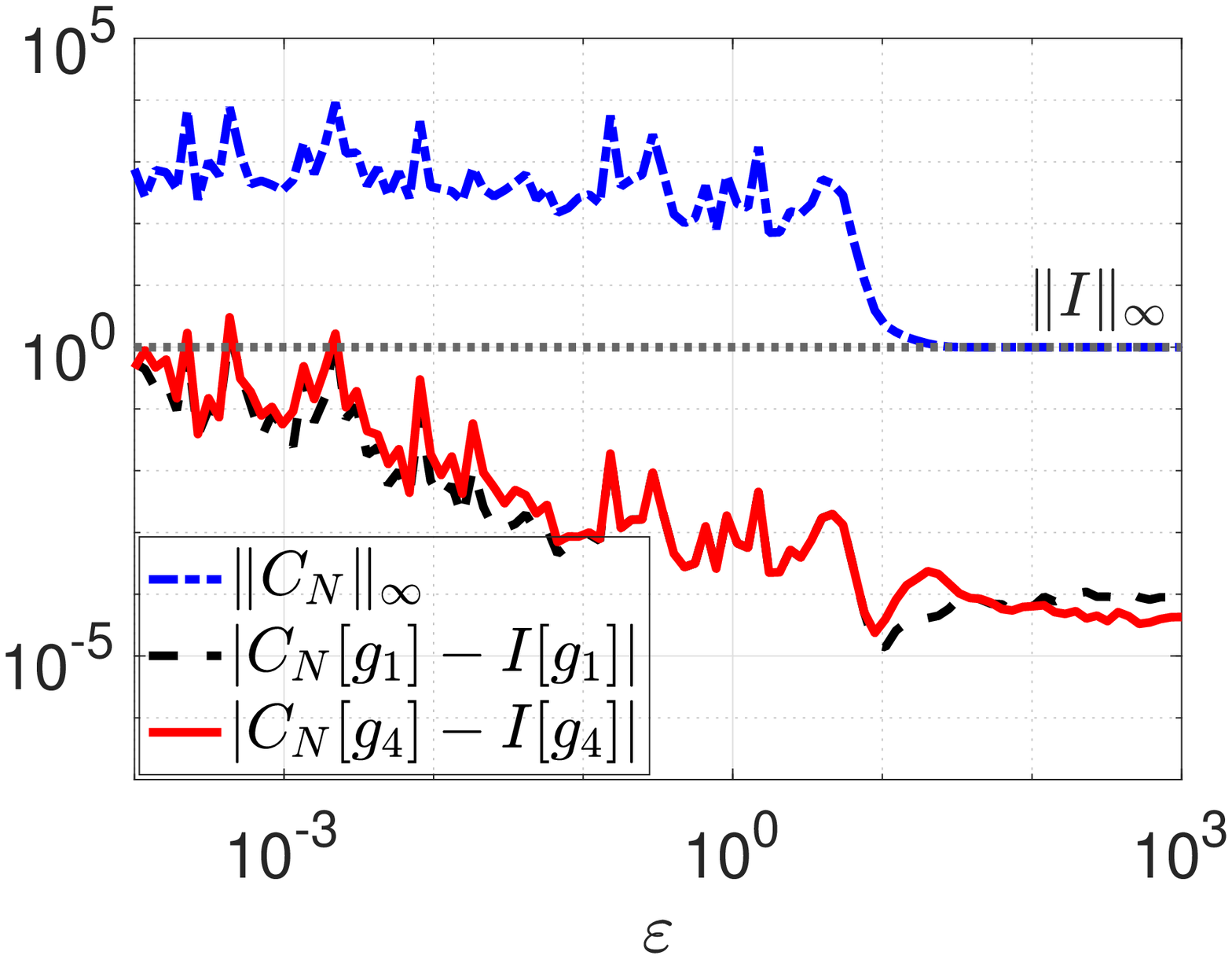} 
    \caption{Halton, $d=1$}
    \label{fig:G_Genz14_N400_Halton_d1_noise4}
  	\end{subfigure}%
	~ 
	\begin{subfigure}[b]{0.33\textwidth}
		\includegraphics[width=\textwidth]{%
      	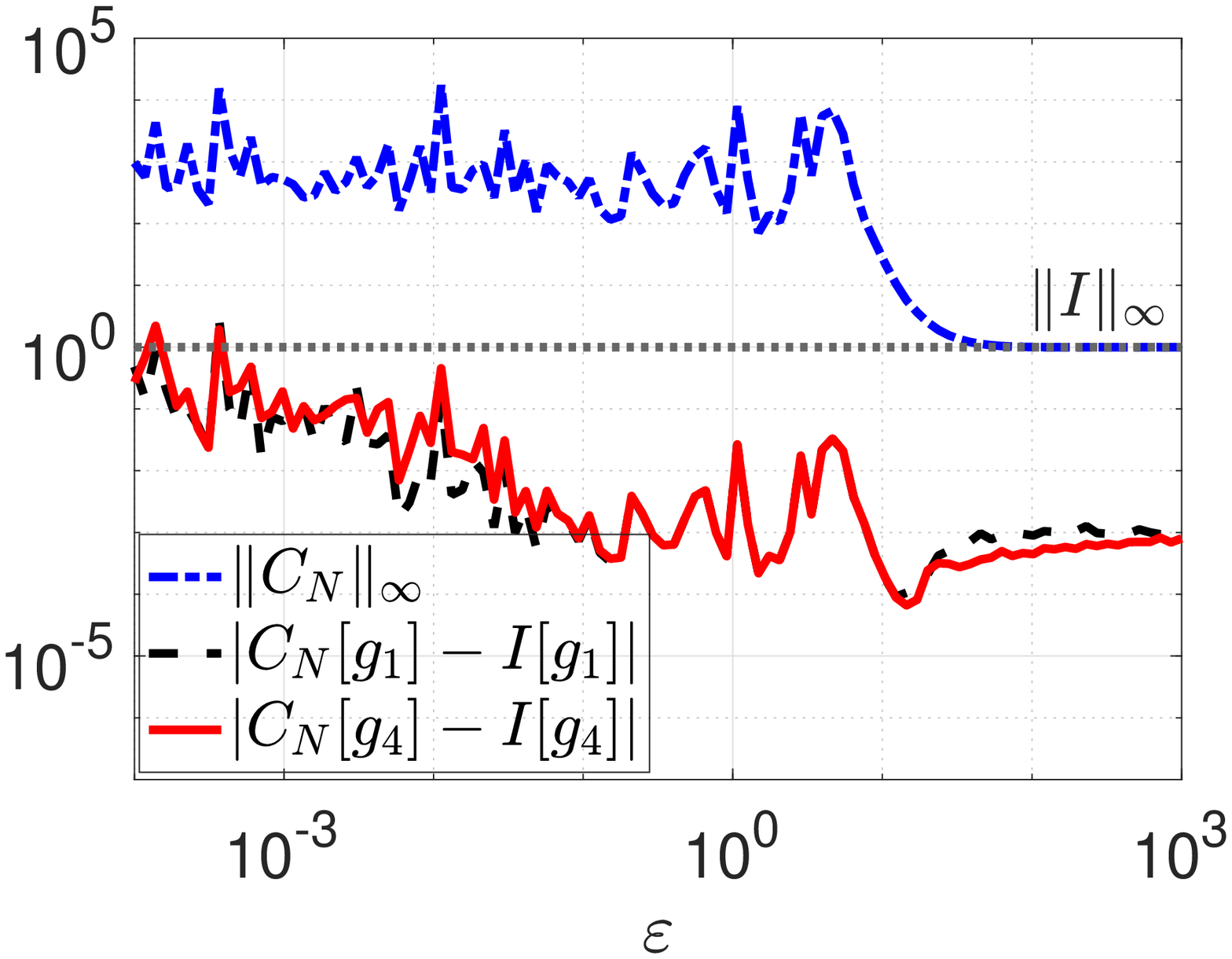} 
    \caption{Random, $d=1$}
    \label{fig:G_Genz14_N400_random_d1_noise4}
  	\end{subfigure}%
  	\caption{
	Error analysis for the Gaussian RBF $\varphi(r) = \exp( \varepsilon^2 r^2 ) $ in two dimensions for the first and fourth Genz test function $g_1, g_4$ on $\Omega = [0,1]^2$; see \cref{eq:Genz}. 
	Uniform white noise $\mathbf{n} \in \R^N$ with $\| \mathbf{n} \|_\infty \leq 10^{-4}$ was added to the function values.  
	In all cases, $N=400$ data points (equidistant, Halton, or random) were considered, while the reference shape parameter $\varepsilon$ was allowed to vary.
  	}
  	\label{fig:error_Gauss_2d_Genz14_noise4}
\end{figure}

That said, the situation changes if the data (function values) used in the RBF-CFs are perturbed by noise, which is often the case in applications. 
Such a situation is reported in \cref{fig:error_Gauss_2d_Genz14_noise4}. 
Here, uniform white noise $\mathbf{n} \in \R^N$ with $\| \mathbf{n} \|_\infty \leq 10^{-4}$ was added to the function values of the first and fourth Genz test function. 
As a result, the term including the stability measure $\| C_N \|_{\infty}$ in \cref{eq:L-inequality} gains in importance. 
In accordance with this, the minimal errors in \cref{fig:error_Gauss_2d_Genz14_noise4} are now attained for larger shape parameters that correspond to the RBF-CF having a smaller stability measure $\| C_N \|_{\infty}$ as before. 
Also see \cref{tab:min_error_Gauss_g1} and \cref{tab:min_error_Gauss_g4} below. 
In particular, this demonstrates the increased importance of stability of CF when these are used in real-world applications where the presence of noise can often not be avoided. 

\begin{table}[tb]
\renewcommand{\arraystretch}{1.4}
\centering 
  	\begin{tabular}{c c c c c c c c c} 
	\toprule 
	& & \multicolumn{3}{c}{$g_1$ without noise} & & \multicolumn{3}{c}{$g_1$ with noise} \\ \hline 
	& & $e_{\text{min}}$ & $\varepsilon$ & $\|C_N\|_{\infty}$ 
	& & $e_{\text{min}}$ & $\varepsilon$ & $\|C_N\|_{\infty}$ \\ \hline 
	& & \multicolumn{7}{c}{Equidistant Points} \\ \hline 
	$d=0$ 	& & 6.1e-10 & 2.4e+00 & 1.1e+02 
				& & 6.6e-06 & 7.5e+00 & 1.0e+00 \\ 
	$d=1$ 	& & 5.4e-10 & 3.3e+00 & 2.6e+02 			
				& & 6.7e-06 & 7.5e+00 & 1.0e+00 \\ \hline 
	& & \multicolumn{7}{c}{Halton Points} \\ \hline 
	$d=0$ 	& & 2.4e-09 & 2.8e+00 & 8.1e+01 
				& & 4.6e-05 & 8.9e+00 & 3.9e+00 \\ 
	$d=1$ 	& & 4.1e-10 & 2.8e+00 & 1.4e+02 
				& & 1.3e-05 & 1.0e+01 & 2.2e+00 \\ \hline 
	& & \multicolumn{7}{c}{Random Points} \\ \hline 
	$d=0$ 	& & 1.5e-09 & 2.0e+00 & 6.4e+01 
				& & 1.9e-04 & 2.0e+00 & 6.4e+01 \\ \hline 
	$d=1$ 	& & 7.8e-10 & 2.0e+00 & 1.0e+02 
				& & 9.1e-05 & 1.2e+01 & 1.0e+01 \\
	\bottomrule
\end{tabular} 
\caption{Minimal errors, $e_{\text{min}}$, for the first Genz test function, $g_1$, with and without noise together with the corresponding shape parameter, $\varepsilon$, and stability measure, $\|C_N\|_{\infty}$. 
	In all cases, the Gaussian RBF was used.}
\label{tab:min_error_Gauss_g1}
\end{table}

\begin{table}[tb]
\renewcommand{\arraystretch}{1.4}
\centering 
  	\begin{tabular}{c c c c c c c c c} 
	\toprule 
	& & \multicolumn{3}{c}{$g_4$ without noise} & & \multicolumn{3}{c}{$g_4$ with noise} \\ \hline 
	& & $e_{\text{min}}$ & $\varepsilon$ & $\|C_N\|_{\infty}$ 
	& & $e_{\text{min}}$ & $\varepsilon$ & $\|C_N\|_{\infty}$ \\ \hline 
	& & \multicolumn{7}{c}{Equidistant Points} \\ \hline 
	$d=0$ 	& & 7.8e-10 & 2.4e+00 & 1.1e+02 
				& & 1.0e-05 & 6.4e+00 & 3.1e+00 \\ 
	$d=1$ 	& & 4.6e-10 & 3.3e+00 & 2.6e+02 			
				& & 1.0e-05 & 6.4e+00 & 3.1e+00 \\ \hline 
	& & \multicolumn{7}{c}{Halton Points} \\ \hline 
	$d=0$ 	& & 1.0e-09 & 2.8e+00 & 8.1e+01 
				& & 2.8e-05 & 8.9e+00 & 3.9e+00 \\ 
	$d=1$ 	& & 1.0e-09 & 2.8e+00 & 1.4e+02 
				& & 2.0e-05 & 8.9e+00 & 3.9e+00 \\ \hline 
	& & \multicolumn{7}{c}{Random Points} \\ \hline 
	$d=0$ 	& & 4.8e-10 & 2.0e+00 & 6.4e+01 
				& & 1.3e-04 & 1.7e+01 & 3.6e+00 \\ \hline 
	$d=1$ 	& & 9.7e-10 & 9.1e-01 & 1.5e+02 
				& & 6.6e-05 & 1.4e+01 & 5.7e+00 \\
	\bottomrule
\end{tabular} 
\caption{Minimal errors, $e_{\text{min}}$, for the fourth Genz test function, $g_4$, with and without noise together with the corresponding shape parameter, $\varepsilon$, and stability measure, $\|C_N\|_{\infty}$. 
	In all cases, the Gaussian RBF was used.}
\label{tab:min_error_Gauss_g4}
\end{table}

It might be hard to identify the smallest errors as well as the corresponding shape parameter and stability measure from \cref{fig:error_Gauss_2d_Genz14} and \cref{fig:error_Gauss_2d_Genz14_noise4}. 
Hence, these are listed separately in \cref{tab:min_error_Gauss_g1} and \cref{tab:min_error_Gauss_g4} for the first and fourth Genz test function with and without noise, respectively.

\subsection{Polyharmonic Splines} 
\label{sub:num_PHS}

\begin{figure}[tb]
	\centering 
	\begin{subfigure}[b]{0.33\textwidth}
		\includegraphics[width=\textwidth]{%
      	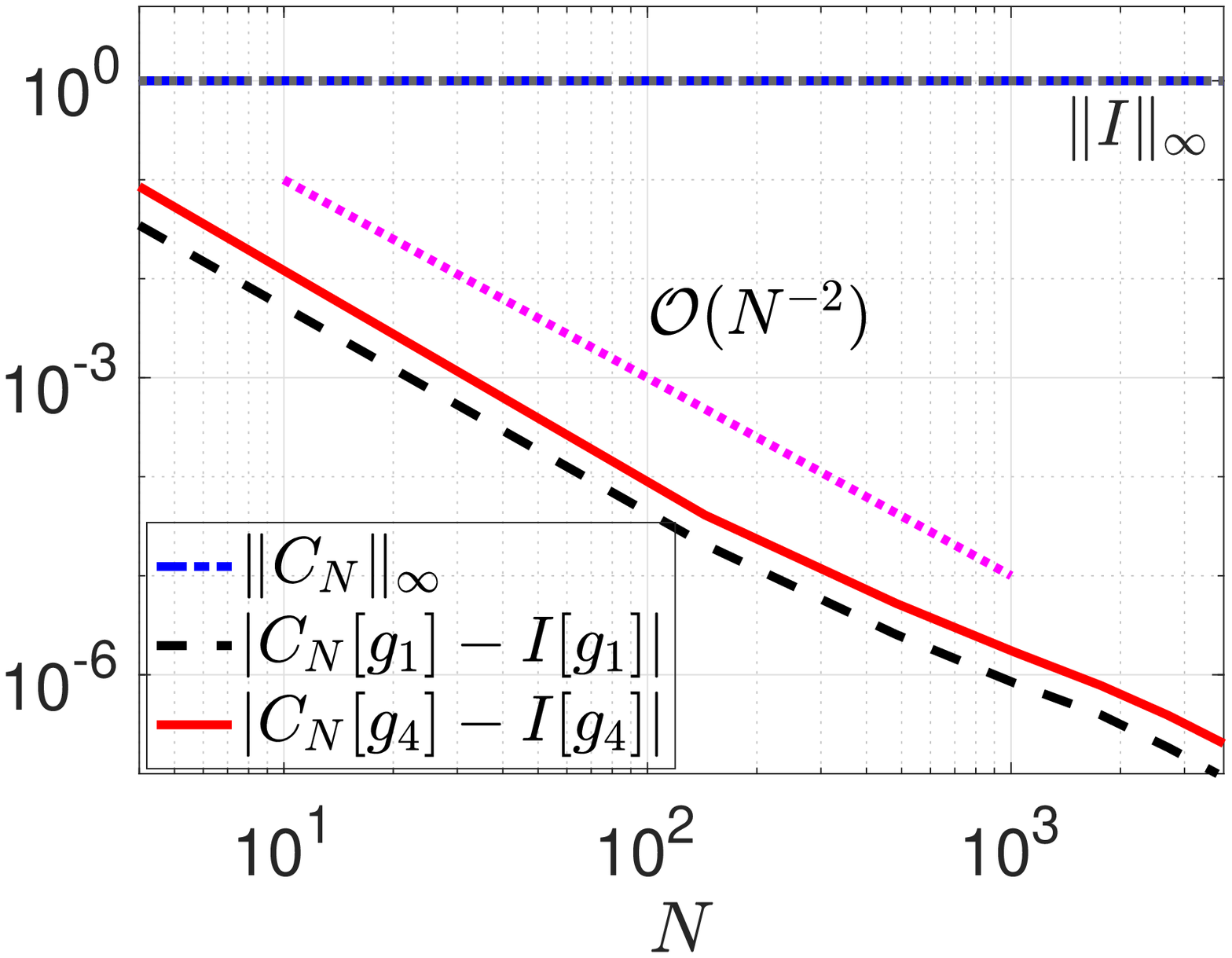} 
    \caption{TPS, equidistant}
    \label{fig:S73_TPS_Genz14_equid_d1}
  	\end{subfigure}%
	~
	\begin{subfigure}[b]{0.33\textwidth}
		\includegraphics[width=\textwidth]{%
      	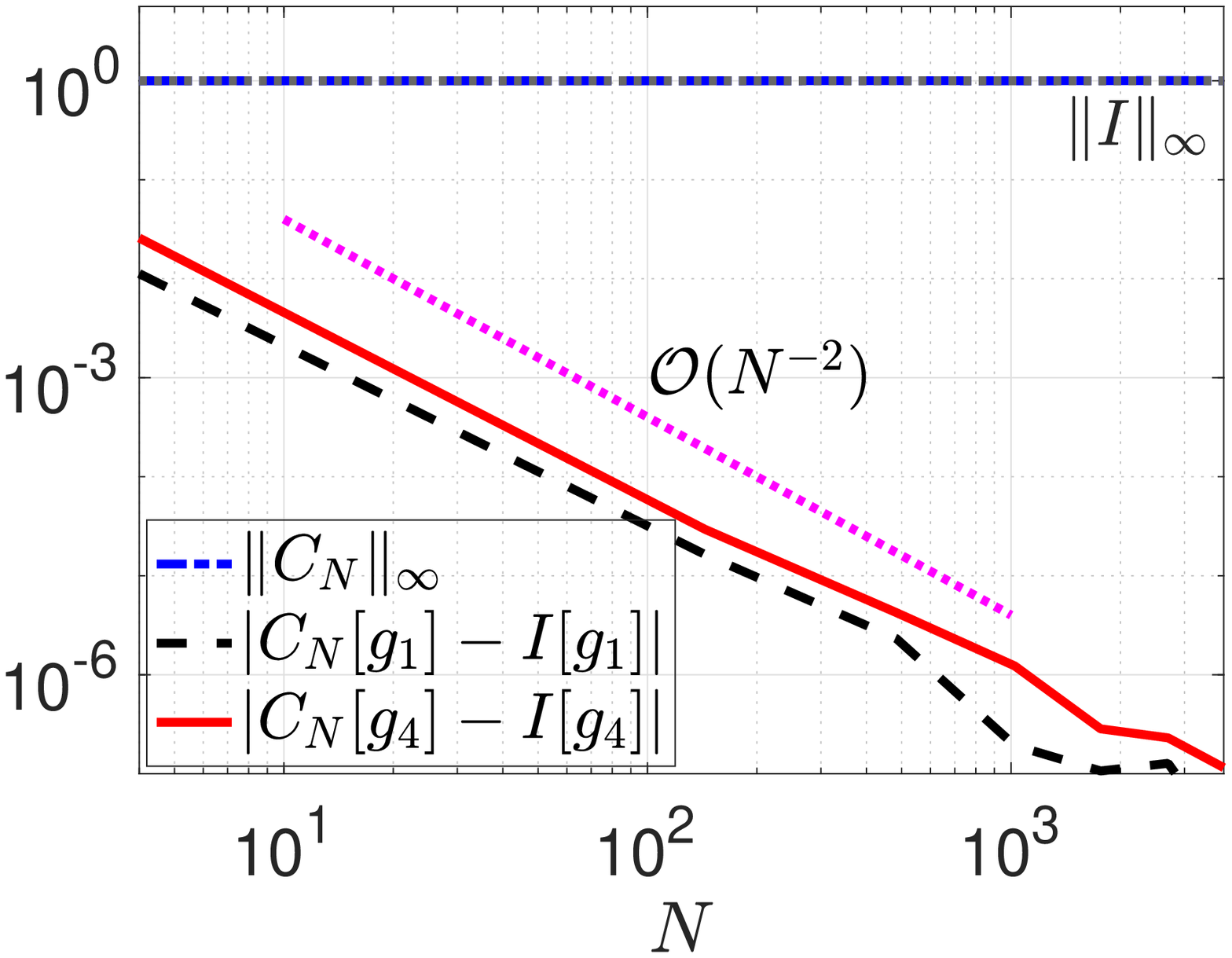} 
    \caption{TPS, Halton}
    \label{fig:S73_TPS_Genz14_Halton_d1}
  	\end{subfigure}%
	~ 
	\begin{subfigure}[b]{0.33\textwidth}
		\includegraphics[width=\textwidth]{%
      	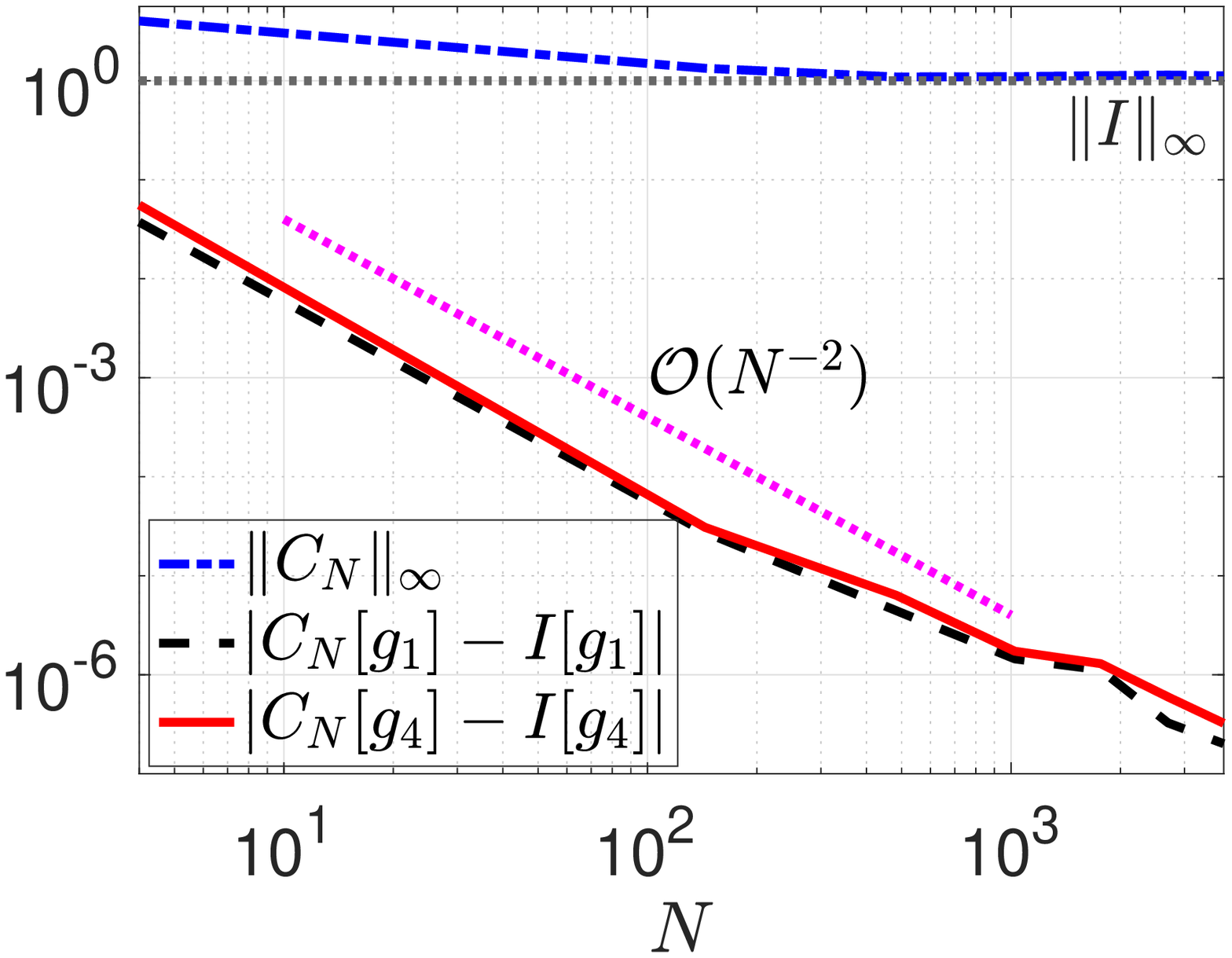} 
    \caption{TPS, random}
    \label{fig:S73_TPS_Genz14_random_d1}
  	\end{subfigure}%
	\\
  	\begin{subfigure}[b]{0.33\textwidth}
		\includegraphics[width=\textwidth]{%
      	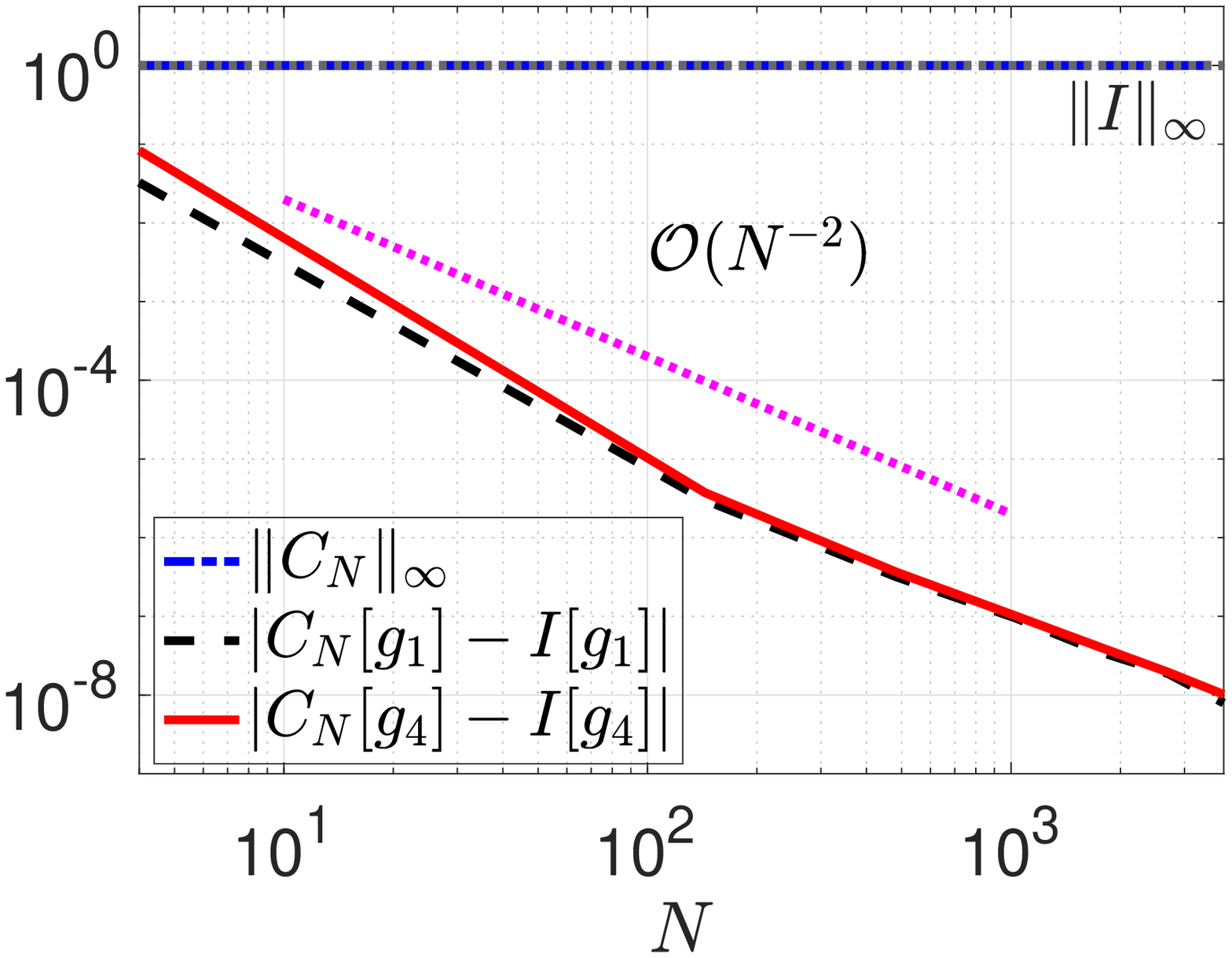} 
    \caption{cubic, equidistant}
    \label{fig:S73_cubic_Genz14_equid_d1}
  	\end{subfigure}%
	~
	\begin{subfigure}[b]{0.33\textwidth}
		\includegraphics[width=\textwidth]{%
      	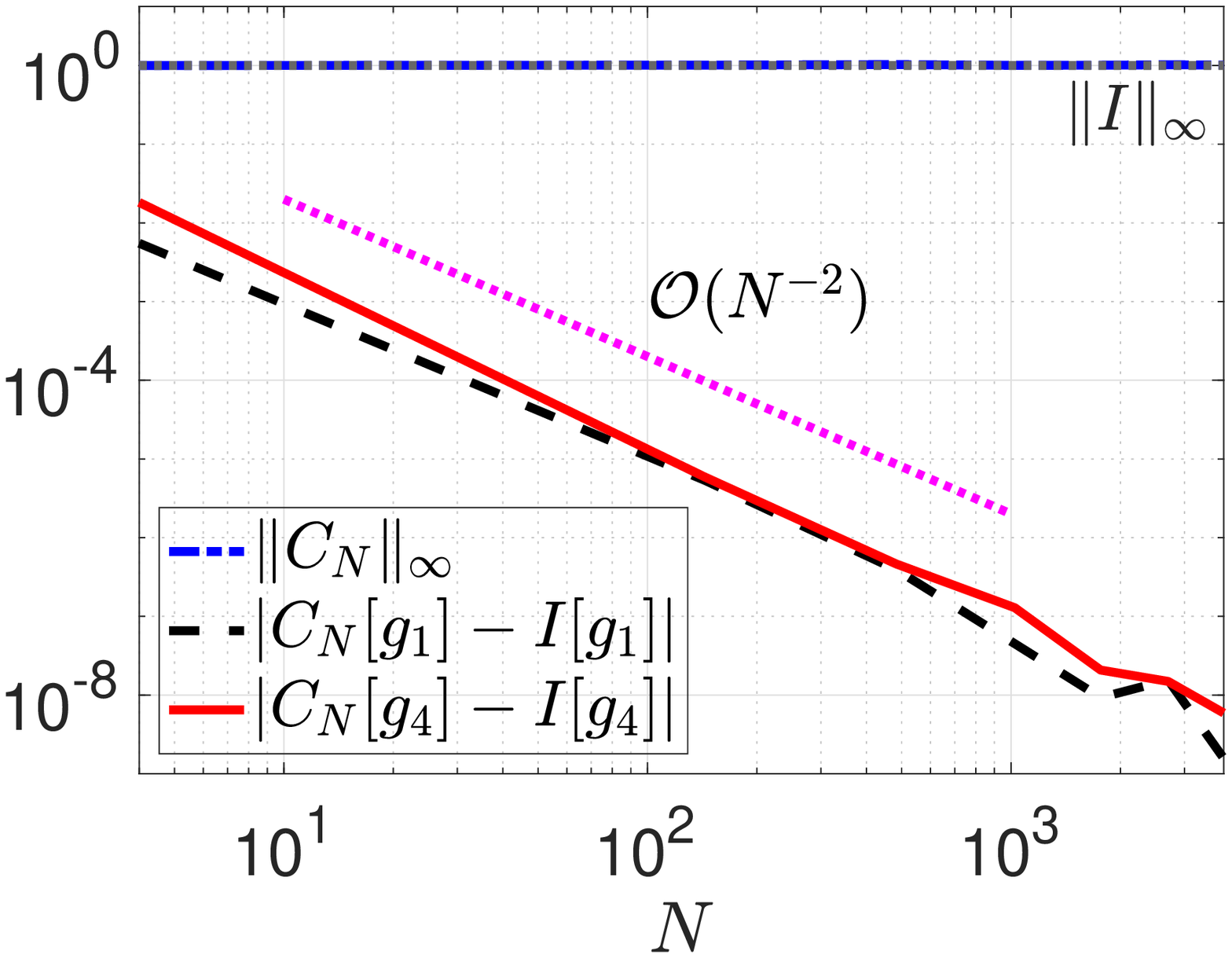} 
    \caption{cubic, Halton}
    \label{fig:S73_cubic_Genz14_Halton_d1}
  	\end{subfigure}%
	~ 
	\begin{subfigure}[b]{0.33\textwidth}
		\includegraphics[width=\textwidth]{%
      	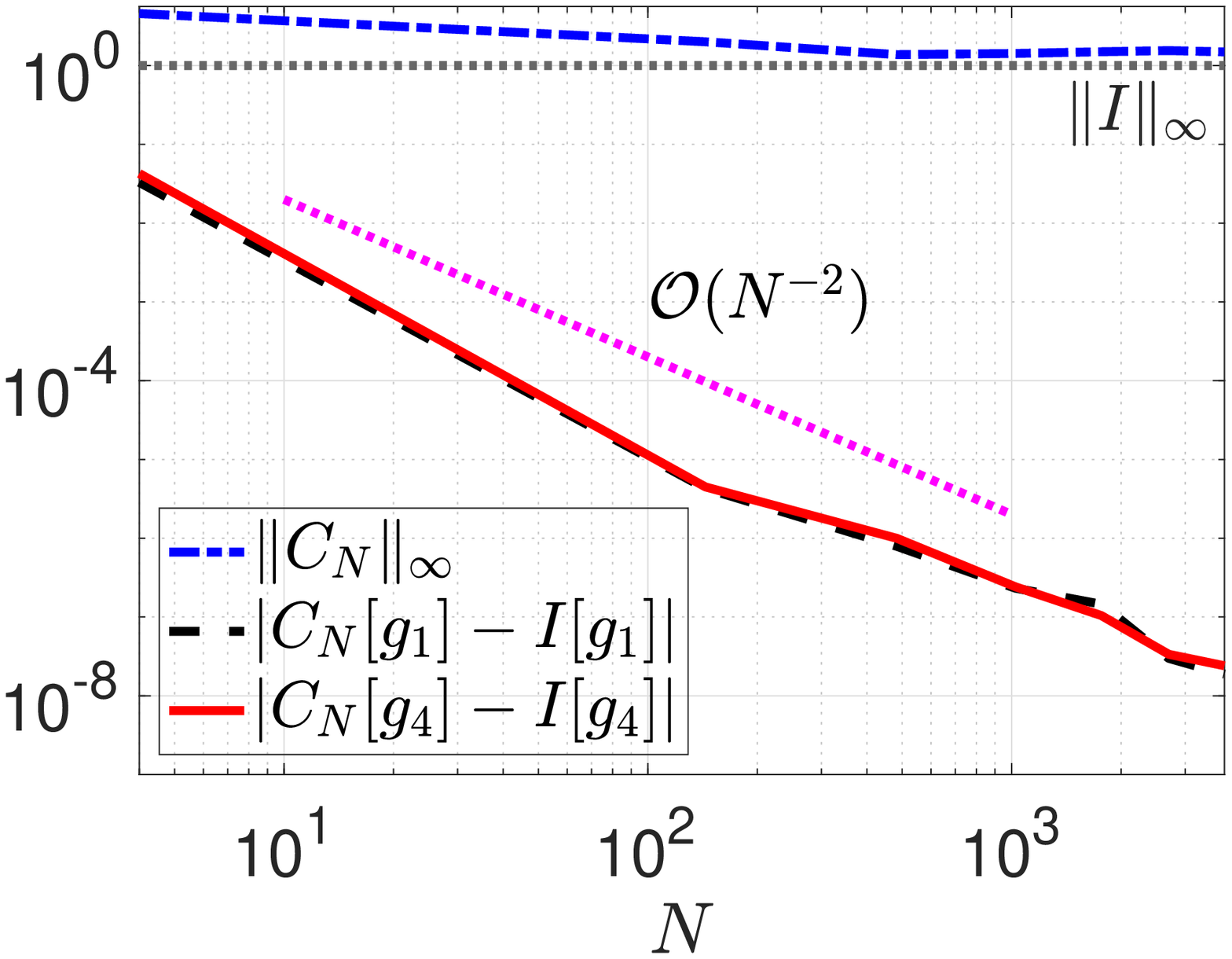} 
    \caption{cubic, random}
    \label{fig:S73_cubic_Genz14_random_d1}
  	\end{subfigure}%
	\\
	\begin{subfigure}[b]{0.33\textwidth}
		\includegraphics[width=\textwidth]{%
      	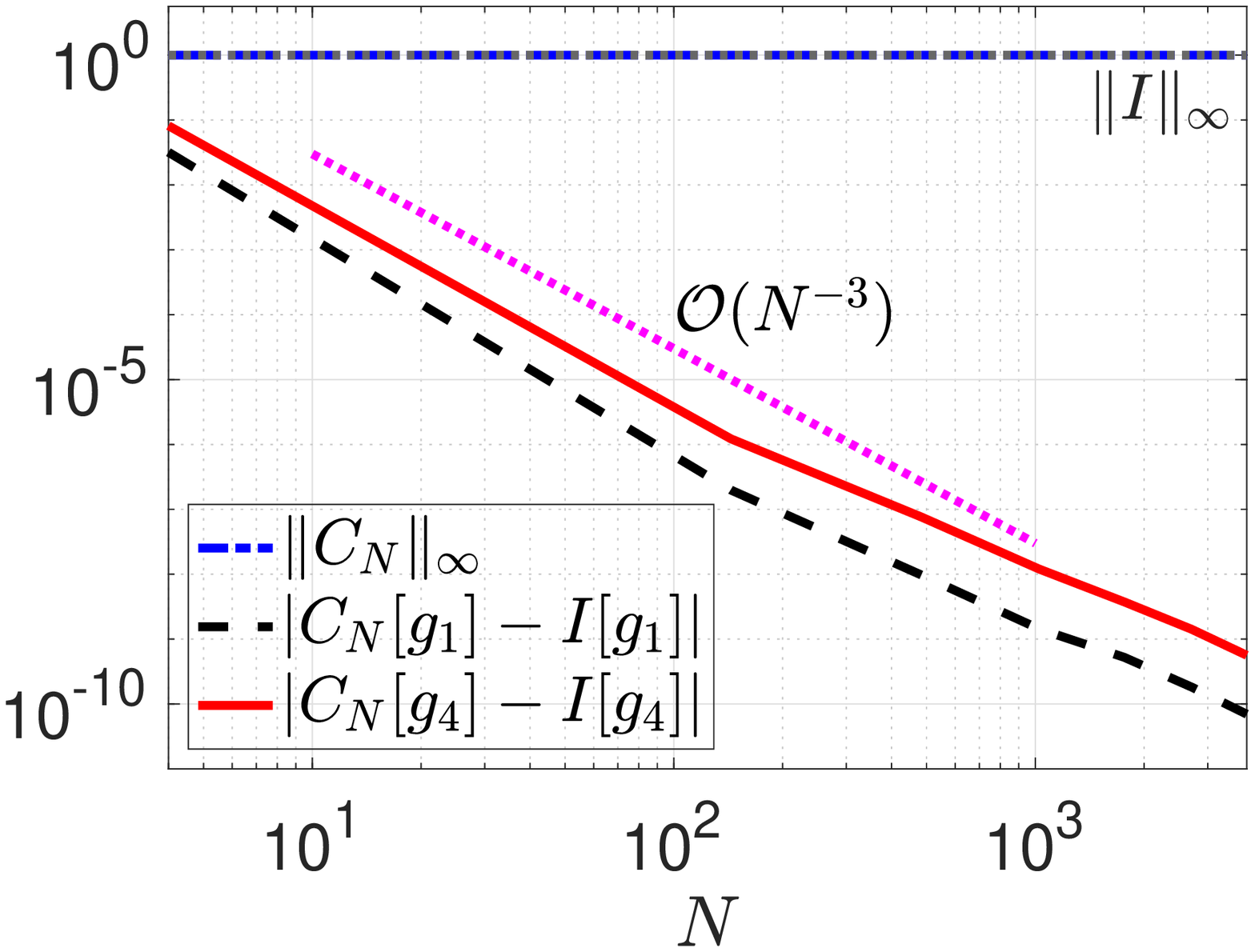} 
    \caption{quintic, equidistant}
    \label{fig:S73_quintic_Genz14_equid_d1}
  	\end{subfigure}%
	~
	\begin{subfigure}[b]{0.33\textwidth}
		\includegraphics[width=\textwidth]{%
      	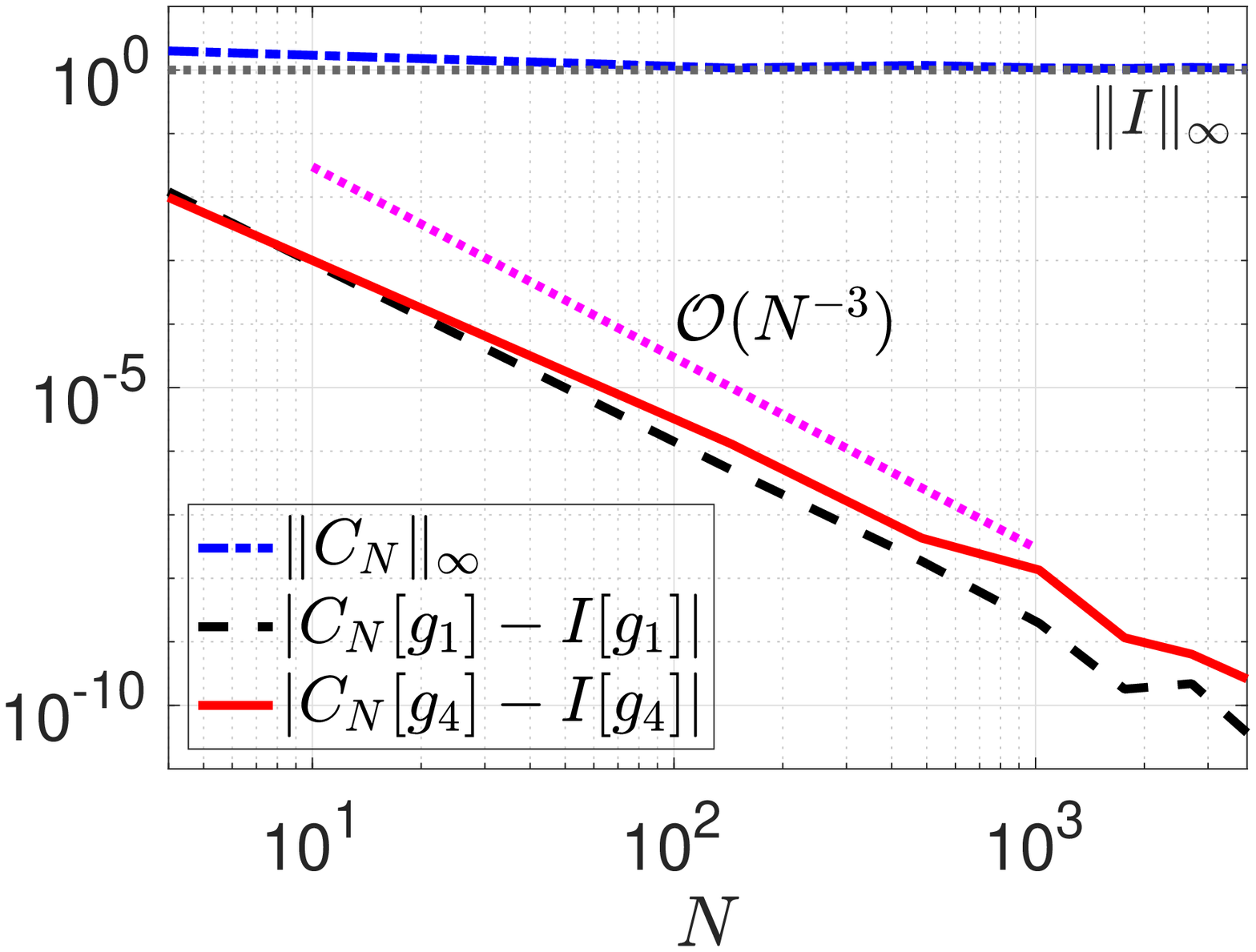} 
    \caption{quintic, Halton}
    \label{fig:S73_quintic_Genz14_Halton_d1}
  	\end{subfigure}%
	~ 
	\begin{subfigure}[b]{0.33\textwidth}
		\includegraphics[width=\textwidth]{%
      	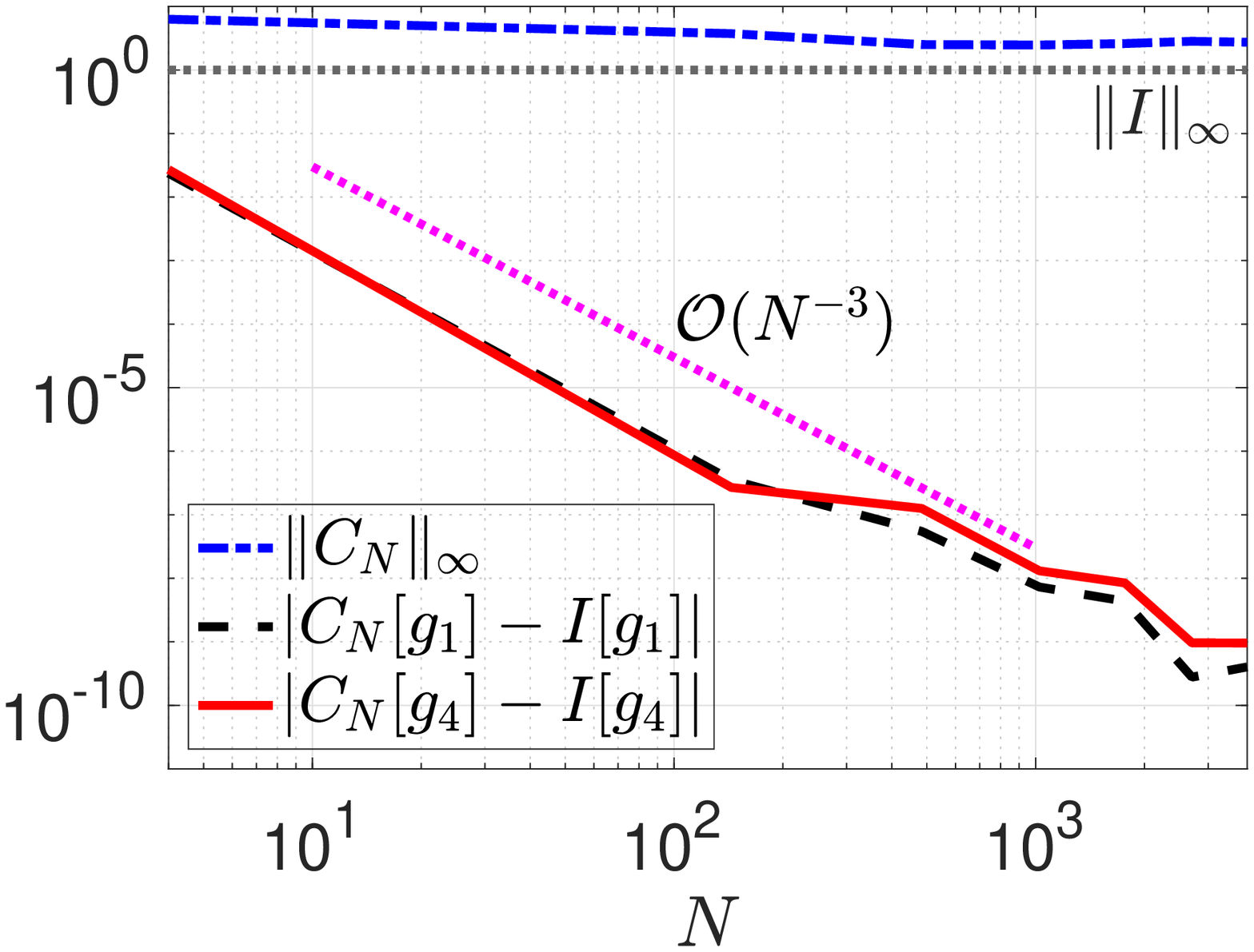} 
    \caption{quintic, random}
    \label{fig:S73_quintic_Genz14_random_d1}
  	\end{subfigure}%
  	\caption{
	Error analysis for the TPS ($\varphi(r) = r^2 \log r$), cubic ($\varphi(r) = r^3$) and quintic ($\varphi(r) = r^5$) in two dimensions. 
	The first and fourth Genz test functions $g_1, g_4$ were considered on $\Omega = [0,1]^2$; see \cref{eq:Genz}. 
	In all cases, linear terms were incorporated, i.\,e., $d=1$. 
  	}
  	\label{fig:error_PHS_Genz14}
\end{figure} 

We end this section by providing a similar investigation for PHS. 
Again, the first and fourth Genz test functions on $\Omega = [0,1]^2$ are considered. 
However, for PHS no shape parameter is involved and we therefore consider their stability and accuracy for an increasing number of equidistant, Halton and random data points. 
The results for the TPS ($\varphi(r) = r^2 \log r$), cubic ($\varphi(r) = r^3$) and quintic ($\varphi(r) = r^5$) PHS RBFs can be found in \cref{fig:error_PHS_Genz14}. 
In all cases, the corresponding PHS basis was augmented with a linear term ($d=1$). 
We can observe from \cref{fig:error_PHS_Genz14} that all RBF-CFs converge (with the rate of convergence depending on the order of the PHS) while also remaining stable or at least being asymptotically stable. 
It would be of interest to provide a theoretical investigation on (asymptotic) stability of PHS-CFs and under which conditions this might be ensured. 
This might be addressed in future works.  
\section{Concluding Thoughts} 
\label{sec:summary} 

In this work, we investigated stability of RBF-CFs. 
We started by showing that stability of RBF-CFs can be connected to the famous Lebesgue constant of the underlying RBF interpolant. 
While this indicates that RBF-CFs might benefit from low Lebesgue constants, it was also demonstrated that RBF-CFs often have superior stability properties compared to RBF interpolation. 
Furthermore, stability was proven for RBF-CFs based on compactly supported RBFs under the assumption of a sufficiently large number of (equidistributed) data points and the shape parameter(s) lying above a certain threshold. 
Finally, we showed that under certain conditions asymptotic stability of RBF-CFs is independent of polynomial terms that are usually included in RBF approximations. 
The above findings were accompanied by a series of numerical tests. 

While we believe this work to be a valuable step towards a more mature stability theory of RBF-CFs, the present work also demonstrates that further steps in this direction would be highly welcome.


\appendix 
\section{Moments} 
\label{sec:app_moments} 

Henceforth, we provide the moments for different RBFs. 
The one-dimensional case is discussed in \cref{sub:app_moments_1d}, while two-dimensional moments are derived in \cref{sub:app_moments_2d}.

\subsection{One-Dimensional Moments}
\label{sub:app_moments_1d}

Let us consider the one-dimensional case of $\Omega = [a,b]$ and distinct data points $x_1,\dots,x_N \in [a,b]$.

\subsubsection{Gaussian RBF} 

For $\varphi(r) = \exp( - \varepsilon^2 r^2 )$, the moment of the translated Gaussian RBF, 
\begin{equation}\label{eq:moment_1d_G}
	m_n = m(\varepsilon,x_n,a,b) = \int_a^b \exp( - \varepsilon^2 | x - x_n |^2 ) \intd x,
\end{equation} 
is given by 
\begin{equation} 
	m_n = \frac{\sqrt{\pi}}{2 \varepsilon} \left[ \mathrm{erf}( \varepsilon(b-x_n) ) - \mathrm{erf}( \varepsilon(a-x_n) ) \right].
\end{equation} 
Here, $\mathrm{erf}(x) = 2/\sqrt{\pi} \int_0^x \exp( -t^2 ) \intd t$ denotes the usual \emph{error function}, \cite[Section 7.2]{dlmf2021digital}.

\subsubsection{Polyharmonic Splines} 

For $\varphi(r) = r^k$ with odd $k \in \N$, the moment of the translated PHS, 
\begin{equation} 
	m_n = m(x_n,a,b) = \int_a^b \varphi( x - x_n ) \intd x,
\end{equation} 
is given by 
\begin{equation} 
	m_n = \frac{1}{k+1} \left[ (a-x_n)^{k+1} + (b-x_n)^{k+1} \right], 
	\quad n=1,2,\dots,N.
\end{equation} 
For $\varphi(r) = r^k \log r$ with even $k \in \N$, on the other hand, we have 
\begin{equation} 
	m_n 
		= (x_n - a)^{k+1} \left[ \frac{\log( x_n - a )}{k+1} - \frac{1}{(k+1)^2} \right] 
		+ (b - x_n)^{k+1} \left[ \frac{\log( b - x_n )}{k+1} - \frac{1}{(k+1)^2} \right].
\end{equation} 
Note that for $x_n = a$ the first term is zero, while for $x_n = b$ the second term is zero.

\subsection{Two-Dimensional Moments}
\label{sub:app_moments_2d} 

Here, we consider the two-dimensional case, where the domain is given by a rectangular of the form $\Omega = [a,b] \times [c,d]$.

\subsubsection{Gaussian RBF}

For $\varphi(r) = \exp( - \varepsilon^2 r^2 )$, the two-dimensional moments can be written as products of one-dimensional moments. 
In fact, we have 
\begin{equation} 
	\int_a^b \int_c^d \exp( - \varepsilon^2 \|(x-x_n,y-y_n\|_2^2 ) 
		=  m(\varepsilon,x_n,a,b) \cdot m(\varepsilon,y_n,c,d).
\end{equation}
Here, the multiplicands on the right-hand side are the one-dimensional moments from \cref{eq:moment_1d_G}.

\subsubsection{Polyharmonic Splines and Other RBFs} 

If it is not possible to trace the two-dimensional moments back to the one-dimensional ones, we are in need of another approach. 
This is, for instance, the case for PHS. 
We start by noting that for a data points $(x_n,y_n) \in [a,b] \times [c,d]$ the corresponding moment can be rewritten as follows: 
\begin{equation} 
	m(x_n,y_n) 
		= \int_{a}^b \int_{c}^d \varphi( \| (x-x_n,y-y_n)^T \|_2 ) \intd y \intd x 
		= \int_{\tilde{a}}^{\tilde{b}} \int_{\tilde{c}}^{\tilde{d}} \varphi( \| (x,y)^T \|_2 ) \intd y \intd x
\end{equation}
with translated boundaries $\tilde{a} = a - x_n$, $\tilde{b} = b - x_n$, $\tilde{c} = c - y_n$, and $\tilde{d} = d - y_n$. 
We are not aware of an explicit formula for such integrals for most popular RBFs readily available from the literature.  
That said, such formulas were derived in \cite{reeger2016numericalA,reeger2016numericalB,reeger2018numerical} (also see \cite[Chapter 2.3]{watts2016radial}) for the integral of $\varphi$ over a \emph{right triangle} with vertices $(0,0)^T$, $(\alpha,0)^T$, and $(\alpha,\beta)^T$. 
Assuming $\tilde{a} < 0 < \tilde{b}$ and $\tilde{c} < 0 < \tilde{d}$, we therefore partition the shifted domain ${\tilde{\Omega} = [\tilde{a},\tilde{b}] \times [\tilde{c},\tilde{d}]}$ into eight right triangles. 
Denoting the corresponding integrals by $I_1, \dots, I_8$, the moment $m(x_n,y_n)$ correspond to the sum of these integrals. 
The procedure is illustrated in \cref{fig:moments}. 

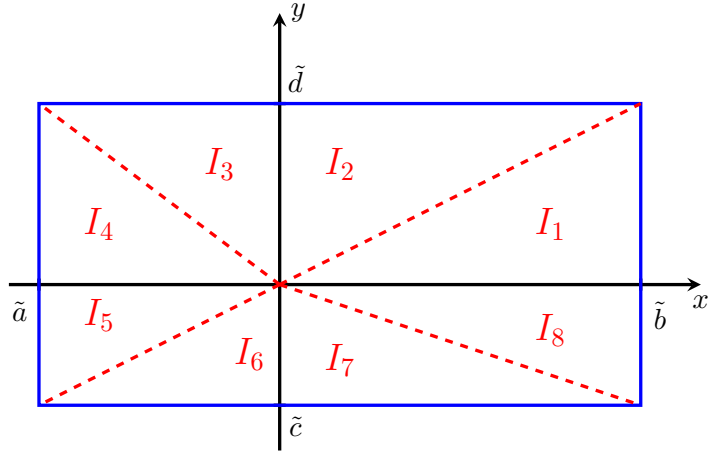
\begin{figure}[tb]
	\centering 
  	\begin{tikzpicture}[domain = -6.5:6.5, scale=0.8, line width=1.25pt]

			\draw[->,>=stealth] (-4.5,0) -- (7,0) node[below] {$x$};
    			\draw[->,>=stealth] (0,-2.75) -- (0,4.5) node[right] {$y$};
			\draw (-4,0.1) -- (-4,-0.1) node [below] {$\tilde{a}$ \ \ \ \ };
			\draw (6,0.1) -- (6,-0.1) node [below] {\ \ \ \ $\tilde{b}$};
			\draw (-0.1,-2) -- (0.1,-2) node [below] {\ \ $\tilde{c}$};
			\draw (-0.1,3) -- (0.1,3) node [above] {\ \ $\tilde{d}$};

			\draw[blue] (-4,-2) rectangle (6,3);
			
			\draw[red,dashed] (0,0) -- (6,3) {};
			\draw[red,dashed] (0,0) -- (-4,3) {};
			\draw[red,dashed] (0,0) -- (-4,-2) {};
			\draw[red,dashed] (0,0) -- (6,-2) {};
			
			\draw[red] (4.5,1) node {\Large $I_1$};
			\draw[red] (1,2) node {\Large $I_2$};
			\draw[red] (-1,2) node {\Large $I_3$};
			\draw[red] (-3,1) node {\Large $I_4$};
			\draw[red] (-3,-0.5) node {\Large $I_5$};
			\draw[red] (-0.5,-1.15) node {\Large $I_6$};
			\draw[red] (1,-1.25) node {\Large $I_7$};
			\draw[red] (4.5,-0.75) node {\Large $I_8$};

	\end{tikzpicture}
  	\caption{Illustration of how the moments can be computed on a rectangle in two dimensions}
  	\label{fig:moments}
\end{figure}

The special cases where one (or two) of the edges of the rectangle align with one of the axes can be treated similarly. 
However, in this case, a smaller subset of the triangles is considered. 
We leave the details to the reader, and note the following formula for the weights: 
\begin{equation} 
\begin{aligned}
	m(x_n,y_n) 
		& = \left[ 1 - \delta_0\left(\tilde{b} \tilde{d}\right) \right] \left( I_1 + I_2 \right) 
		+ \left[ 1 - \delta_0\left(\tilde{a} \tilde{d}\right) \right] \left( I_3 + I_4 \right) \\
		& + \left[ 1 - \delta_0\left(\tilde{a} \tilde{c}\right) \right] \left( I_5 + I_6 \right) 
		+ \left[ 1 - \delta_0\left(\tilde{b} \tilde{c}\right) \right] \left( I_7 + I_8 \right) 
\end{aligned}
\end{equation} 
Here, $\delta_0$ denotes the usual Kronecker delta defined as $\delta_0(x) = 1$ if $x = 0$ and $\delta_0(x) = 0$ if $x \neq 0$. 
The above formula holds for general $\tilde{a}$, $\tilde{b}$, $\tilde{c}$, and $\tilde{d}$. 
Note that all the right triangles can be rotated or mirrored in a way that yields a corresponding integral of the form 
\begin{equation}\label{eq:refInt} 
	I_{\text{ref}}(\alpha,\beta) 
		= \int_0^{\alpha} \int_0^{\frac{\beta}{\alpha}x} \varphi( \| (x,y)^T \|_2 ) \intd y \intd x.
\end{equation} 
More precisely, we have 
\begin{equation}
\begin{alignedat}{4}
	& I_1 = I_{\text{ref}}(\tilde{b},\tilde{d}), \quad 
	&& I_2 = I_{\text{ref}}(\tilde{d},\tilde{b}), \quad  
	&& I_3 = I_{\text{ref}}(\tilde{d},-\tilde{a}), \quad 
	&& I_4 = I_{\text{ref}}(-\tilde{a},\tilde{d}), \\ 
	& I_5 = I_{\text{ref}}(-\tilde{a},-\tilde{c}), \quad 
	&& I_6 = I_{\text{ref}}(-\tilde{c},-\tilde{a}), \quad 
	&& I_7 = I_{\text{ref}}(-\tilde{c},\tilde{b}), \quad 
	&& I_8 = I_{\text{ref}}(\tilde{b},-\tilde{c}).
\end{alignedat} 
\end{equation}
Finally, explicit formulas of the reference integral $I_{\text{ref}}(\alpha,\beta)$ over the right triangle with vertices $(0,0)^T$, $(\alpha,0)^T$, and $(\alpha,\beta)^T$ for some PHS can be found in \cref{tab:moments}. 
Similar formulas are also available, for instance, for Gaussian, multiquadric and inverse multiquadric RBFs. 

\begin{table}[t]
  \centering 
  \renewcommand{\arraystretch}{1.5}
  \begin{tabular}{c|c}
    $\varphi(r)$ & $I_{\text{ref}}(\alpha,\beta)$ \\ 
    \midrule 
    $r^2 \log r$ & $\frac{\alpha}{144} \left[ 24\alpha^3 \arctan\left( \beta/\alpha \right) + 6 \beta (3\alpha^2 + \beta^2) \log( \alpha^2 + \beta^2 ) - 33\alpha^2\beta - 7\beta^3 \right]$ \\ 
    $r^3$ & $\frac{\alpha}{40} \left[ 3 \alpha^4 \arcsinh\left( \beta/\alpha \right) + \beta (5\alpha^2 + 2 \beta^2) \sqrt{ \alpha^2 + \beta^2} \right]$ \\ 
    $r^5$ & $\frac{\alpha}{336} \left[ 15 \alpha^6 \arcsinh\left( \beta/\alpha \right) + \beta (33\alpha^4 + 26\alpha^2\beta^2 + 8 \beta^4) \sqrt{ \alpha^2 + \beta^2} \right]$ \\ 
    $r^7$ & $\frac{\alpha}{3346} \left[ 105 \alpha^8 \arcsinh\left( \beta/\alpha \right) + \beta (279\alpha^6 + 326\alpha^4\beta^2 + 200\alpha^2\beta^4 + 48 \beta^6) \sqrt{ \alpha^2 + \beta^2} \right]$ \\ 
  \end{tabular} 
  \caption{The reference integral $I_{\text{ref}}(\alpha,\beta)$---see \cref{eq:refInt}---for some PHS}
  \label{tab:moments}
\end{table} 

We note that the approach presented above is similar to the one in \cite{sommariva2006numerical}, where the domain $\Omega = [-1,1]^2$ was considered. 
Later, the same authors extended their findings to simple polygons \cite{sommariva2006meshless} using the Gauss--Grenn theorem.  
Also see the recent work \cite{sommariva2021rbf}, addressing polygonal regions that may be nonconvex or even multiply connected, and references therein. 
It would be of interest to see if these approaches also carry over to computing products of RBFs corresponding to different centers or products of RBFs and their partial derivatives, again corresponding to different centers. 
Such integrals occur as elements of mass and stiffness matrices in numerical PDEs. 
In particular, they are desired to construct linearly energy stable (global) RBF methods for hyperbolic conservation laws \cite{glaubitz2020shock,glaubitz2021stabilizing,glaubitz2021towards}.  

\bibliographystyle{siamplain}
\bibliography{literature}

\end{document}